\documentclass[12pt,centertags]{amsart}
\usepackage%[hypertex]
{hyperref}
\usepackage[headsepline]{scrpage2}
\ihead{\textit{Xiaonan Ma \& George Marinescu\,{\rm:} Toeplitz operators on symplectic manifolds}}

\usepackage{amsmath,amstext,amsthm,a4,amssymb,amscd,typearea}
\usepackage[mathscr]{eucal}
\usepackage{mathrsfs}
\numberwithin{equation}{section}

\newcommand{\field}[1]{\mathbb{#1}}
\newcommand{\Z}{\field{Z}}
\newcommand{\R}{\field{R}}
\newcommand{\C}{\field{C}}
\newcommand{\N}{\field{N}}
\newcommand{\proj}{\field{P}}

\def\bJ{{\boldsymbol J}}

\newcommand{\cali}[1]{\mathscr{#1}}
\newcommand{\cC}{\cali{C}} 
\newcommand{\cO}{\cali{O}} 
\newcommand{\cH}{\cali{H}} 
 \newcommand{\cL}{\cali{L}}
 
\newcommand{\cK}{\cali{K}} \newcommand{\cT}{\cali{T}}

\newcommand{\calig}[1]{\mathcal{#1}}

 \newcommand\mO{\calig{O}}
\newcommand\mQ{\calig{Q}} \newcommand\mR{\calig{R}}
 \newcommand{\cP}{\cali{P}}
\newcommand\mK{\calig{K}}
\def\bE{{\boldsymbol E}}

\newcommand{\imat}{\sqrt{-1}}

\DeclareMathOperator{\End}{End} 
\DeclareMathOperator{\ke}{Ker} 
 
  \DeclareMathOperator{\spec}{Spec}
\DeclareMathOperator{\Id}{Id} \DeclareMathOperator{\supp}{supp}
\DeclareMathOperator{\tr}{Tr} 
 
 \newcommand{\spin}{$\text{spin}^c$ }
\DeclareMathOperator{\db}{\overline\partial}

\DeclareMathOperator{\pr}{pr} 

\newcommand{\norm}[1]{\lVert#1\rVert} \newcommand{\abs}[1]{\lvert#1\rvert}
\newcommand{\om}{\omega}

\newcommand{\sbullet}{\scriptscriptstyle{\bullet}}

\newtheorem{thm}{Theorem}[section]
\newtheorem{lemma}[thm]{Lemma}
\newtheorem{prop}[thm]{Proposition}

\theoremstyle{definition}
\newtheorem{defn}[thm]{Definition}
\newtheorem{condition}[thm]{Condition}
\newtheorem{notation}[thm]{Notation}

\newtheorem{rem}[thm]{Remark}
\newtheorem{ex}[thm]{Example}

\newcommand{\be}{\begin{equation}}
\newcommand{\ee}{\end{equation}}
\newcommand{\wi}{\widetilde}
\newcommand{\var}{\varepsilon}
\newcommand{\ov}{\overline}
\newcommand{\comment}[1]{}

\begin{document}
\title{Toeplitz operators on symplectic manifolds}
\date{\today}
%    Information for first author
\author{\textbf{Xiaonan Ma}}
\address{Universit{\'e} Denis Diderot - Paris 7,
UFR de Math{\'e}matiques, Case 7012,
Site Chevaleret,
75205 Paris Cedex 13, France}
\email{ma@math.jussieu.fr}
%    Information for second author

\author{\textbf{George Marinescu}}
\address{Universit{\"a}t zu K{\"o}ln,  Mathematisches Institut,
    Weyertal 86-90,   50931 K{\"o}ln, Germany\\
    \& Institute of Mathematics `Simion Stoilow', Romanian Academy,
Bucharest, Romania}
\thanks{Second-named author partially supported by the SFB/TR 12}
\email{gmarines@math.uni-koeln.de}



\dedicatory{Dedicated to Professor Gennadi Henkin with the occasion of his 65th anniversary}

\begin{abstract} We study the Berezin-Toeplitz quantization
on symplectic manifolds making use of the full off-diagonal
asymptotic expansion of the Bergman kernel. 
We give also a characterization of Toeplitz operators 
in terms of their asymptotic expansion.
The semi-classical limit properties of 
the Berezin-Toeplitz quantization for non-compact manifolds and orbifolds 
are also established.
\end{abstract}

\maketitle

\setcounter{section}{-1}
%%%%%%%%%%%%%%%%%%%%%%%%%%%%%%%%%%%%%%%%%%%
%%%%%%%%%%%%%%%%%%%%%%%%%%%%%%%%%%%%%%%%%%%
\section{Introduction} \label{s1}

Quantization is a procedure that leads from a classical dynamical system 
to an associated algebra whose behavior reduces to that of 
the given classical system in an appropriate limit. 
In the usual case, the limit involves Planck's constant $\hbar$ approaching zero.
%(In the usual case, the limit involves Planck's constant h approaching zero.)
The aim of the geometric quantization theory 
\cite{BFFLS,Berez:74,Fedo:96,Kos:70,Sou:70} is to relate the classical observables 
(smooth functions) on a phase space (a symplectic manifold) to 
the quantum observables (bounded linear operators) on the quantum space 
(sections of a line bundle). One particular way to quantize the phase space 
is the Berezin-Toeplitz quantization,
which we briefly describe.

Let us consider a compact K{\"a}hler manifold $X$ with K{\"a}hler form $\om$. 
On $X$ we are given a holomorphic Hermitian line bundle $(L,h^L)$
 endowed with the Chern connection $\nabla^L$ with curvature $R^L$. 
We assume that the prequantization condition  
$\frac{\sqrt{-1}}{2\pi}R^L=\omega$ is fulfilled. 
For any $p\in\N$ let $L^p:=L^{\otimes p}$ be the $p^{\text{th}}$ tensor power of $L$,
$L^2(X,L^p)$ be the space of $L^2$-sections of $L^p$ 
with norm induced by $h^L$ and $\om$, and 
$P_p:L^2(X,L^p)\to H^0(X,L^p)$ be the orthogonal projection on the space of 
holomorphic sections.
To any function $f\in\cC^\infty(X)$
we associate a sequence of linear operators 
\begin{equation}\label{0.0}
T_{f,\,p}:L^2(X,L^p)\to L^2(X,L^p), \quad T_{f,\,p}=P_p\,f\,P_p\,,
\end{equation} 
where for simplicity we denote by $f$ the operator of multiplication with $f$.
Then as $p\to \infty$, the following properties hold:
\begin{equation}\label{0.1}
\begin{split}
&\lim_{p\to\infty}\norm{T_{f,\,p}}={\norm f}_\infty
:=\sup_{x\in X} |f(x)|\,,\\
&[T_{f,\,p}\,,T_{g,\,p}]=\frac{\sqrt{-1}}{\, p}T_{\{f,g\},\,p}+\mO(p^{-2}),
\end{split}
\end{equation}
where $\{ \,\cdot\, , \,\cdot\, \}$ is the Poisson bracket on $(X,2\pi \om)$
(cf. \eqref{toe4.1}) and $\|\cdot\|$ is the operator norm.
Thus the Poisson algebra $(\cC^\infty(X),\{ \cdot , \cdot \})$ 
is approximated by the operator algebras of Toeplitz operators 
in the norm sense as $p\to \infty$; the role of the Planck constant is 
played by $\hbar=1/p$. This is the so-called semi-classical limit process.

The relations \eqref{0.1} were proved first in some special cases: 
in \cite{KliLe:92} for Riemannian surfaces, in \cite{Cob:92} for $\C^n$ 
and in \cite{BLU:93} for bounded symmetric domains in $\C^n$, 
by using explicit calculations.
Then Bordemann, Meinrenken and Schlichenmaier \cite{BMS} 
%(cf. also Guillemin \cite{Guill:95}) 
treated the case of a compact K{\"a}hler manifold
using the theory of Toeplitz structures (generalized Szeg{\"o} operators) by
Boutet de Monvel and Guillemin \cite{BoG}.
%Guillemin \cite{Guill:95} notices that the method is implicit in \cite{BoG}.
Moreover, Schlichenmaier \cite{Schlich:00} 
(cf. also  \cite{KS01}, \cite{Charles1})
continued this train of thought and showed that for any $f,g\in \cC^\infty(X)$,
 the product $T_{f,\,p}\,T_{g,\,p}$ has an asymptotic expansion 
\begin{equation}\label{0.2}
T_{f,\,p}\,T_{g,\,p}=\sum_{k=0}^\infty T_{C_k(f,\,g)}p^{-k}+\cO(p^{-\infty})
\end{equation}
 in the sense of \eqref{atoe2.2}, where $C_k$ are bidifferential operators, 
satisfying $C_0(f,g)=fg$ and $C_1(f,g)-C_1(g,f)=\sqrt{-1}\,\{f,g\}$. 
As a consequence, one constructs geometrically an associative star product,
 defined by setting for any $f,g\in \cC^\infty(X)$,
\begin{equation}\label{0.3}
f*g:=\sum_{k=0}^\infty C_k(f,g) \hbar^{k}\in\cC^\infty(X)[[\hbar]].
\end{equation}
For previous work on Berezin-Toeplitz star products in special cases
see \cite{MoOr:83,CaGuRa:90}.
%%%%%%%%%%%%%%%%%%%%%%%%%%%%%%%%%%%%%%%%%%

The papers \cite{BMS,Charles1,KS01,Schlich:00} rely on the method and results of 
Boutet de Monvel, Guillemin and Sj{\"o}strand \cite{BoG,BS:76}. 
They perform the analysis on the principal bundle associated to $L$, 
i.e., the circle bundle $Y$ of the dual bundle $L^*$ of $L$.
Actually $Y=\partial D$, where $D:=\{v\in L^*\,:\,\abs{v}_{h^{L^*}}<1\}$, which is a strictly pseudoconvex domain, due to the positivity of $(L,h^L)$ (this fact is a basic observation due to Grauert).  

Let us endow $Y$ with the volume form $d\theta\wedge\varrho^*\omega^n$, where $\varrho:Y\to X$ is the bundle projection. Consider the space $L^2(Y)$ and for each $p\in\Z$ the subspace
$L^2(Y)_p$ of functions on $Y$  transforming under the $S^1$-action on $Y$ according to the rule $\varphi(e^{i\theta}y)=e^{ip\theta}\varphi(y)$. There is a canonical isometry $L^2(Y)_p\cong L^2(X,L^p)$ which together with the
Fourier decomposition $L^2(Y)\cong\oplus_{p\in\Z} L^2(Y)_p$ (the latter is a Hilbert space direct sum)
delivers a canonical isometry 
$L^2(Y)\cong\oplus_{p\in\Z} L^2(X,L^p)$.

Let $\db_b$ denote the tangential Cauchy-Riemann operator on $Y$.
A function $\varphi\in L^2(Y)$ is called Cauchy-Riemann (CR for short) if it satisfies 
the tangential Cauchy-Riemann equations $\db_b \varphi=0$ (in the sense of distributions).
Let $\cH^2(Y)\subset L^2(Y)$ be the space of CR functions (Hardy space).
For every $p\in\N$ let us denote $\cH^2_p(Y)= L^2(Y)_p\cap\cH^2(Y)$. Then we have the Hilbert sum decomposition
$\cH^2(Y)=\oplus_{p\in\N}\cH^2_p(Y)$. Moreover, $\cH^2_p(Y)$  
is identified through the canonical isometry $L^2(Y)_p\cong L^2(X,L^p)$ to the subspace 
$H^0(X,L^p)$. Thus
$\cH^2(Y)\cong\oplus_{p\in\N}H^0(X,L^p)$.

Therefore, in order to study the Bergman projections $P_p$, one can replace the family $\{P_p\}_{p\in\N}$ 
with the orthogonal projection $\boldsymbol{S}:L^2(Y)\to\oplus_{p\in\N}\cH^p(Y)$, 
called Szeg{\"o} projection. The key result is that $\boldsymbol{S}$ is 
a Fourier integral operator of order $0$ of Hermite type 
(Boutet de Monvel-Sj{\"o}strand \cite{BS:76}) and this allows to 
apply the theory of Fourier integral operators to obtain the properties of Toeplitz structures.

In the framework of Toeplitz structures, Guillemin \cite{Guill:95} (cf. also \cite{BU} for related results) constructed 
a star product on compact symplectic manifolds by replacing 
the CR functions with functions annihilated by 
a first order pseudodifferential operator
$D_b$ on the circle bundle of $L^*$ introduced in \cite{BoG}. 
The operator $D_b$ has the same microlocal structure as 
the tangential Cauchy-Riemann operator $\db_b$ and it is 
derived actually by first constructing the Szeg{\"o} kernel.

In this paper we propose a different approach to the study of
Berezin-Toeplitz quantization and Toeplitz operators. 
This consists in applying the off-diagonal 
asymptotic expansion as $p\to \infty$ of the Bergman kernel $P_p(x,x')$, 
which is the Schwartz kernel of the Bergman projection $P_p$\,. 

We can actually treat the case of symplectic manifolds.
Let  $(X,\om)$ be a compact symplectic manifold of real dimension $2n$.
Let $(L,h^L)$ be a Hermitian line bundle on $X$ endowed 
with a Hermitian connection $\nabla^L$. 
The curvature of this connection is given by $R^L=(\nabla^L)^2$. 
We will assume throughout the paper that $(L,h^L,\nabla^L)$
satisfies the \textbf{\emph{prequantization condition}}\/:
\begin{equation}\label{0.-1}
\frac{\sqrt{-1}}{2\pi}R^L=\omega\,.
\end{equation}
$(L,h^L,\nabla^L)$ is called a \textbf{\emph{prequantum line bundle}}.
Due to the analogy to the complex manifolds the bundle $L$ 
will be also called positive.
We also consider a twisting Hermitian vector bundle $(E,h^E)$ on $X$
with Hermitian connection $\nabla^E$.

Let $J$ be an almost complex structure on $TX$ such that $\omega$ is compatible with $J$ and $\omega(\cdot,J\cdot)>0$. Let $g^{TX}$ be a Riemannian metric on $TX$ compatible with $J$.

A natural geometric generalization of the operator $\sqrt{2}(\db+\db^*)$ 
acting on $\Omega^{0,\scriptscriptstyle{\bullet}}(X,L^p)$ is 
the \textbf{\emph{spin$^c$ Dirac operator $D_p$}} acting on  
$\Omega^{0,\bullet}(X,L^p\otimes E)$ (cf. Definition \ref{defDirac1}) 
associated to $J, g^{TX}$, $\nabla^L,\nabla^E$. 

We refer to the orthogonal projection $P_p$ from $\Omega^{0,\scriptscriptstyle{\bullet}}(X,L^p\otimes E)$ onto $\ke (D_p)$ as the \textbf{\emph{Bergman projection}} of $D_p$\,. The Schwartz kernel $P_p(\cdot,\cdot)$ of $P_p$ is called \textbf{\emph{Bergman kernel}} of $D_p$ (cf.\,Definition \ref{defBergman})\,. 
For $f\in \cC^\infty (X,\End(E))$, we define the 
\textbf{\emph{Berezin-Toeplitz quantization}} of $f$ as in \eqref{0.0} by 
\begin{equation}\label{0.6}
T_{f,\,p}:=P_p\, f\, P_p\in \End(L^2(X,\Lambda (T^{*(0,1)}X)\otimes L^p\otimes E)).
\end{equation}

Dai, Liu and Ma \cite{DLM04a} proved the asymptotic expansion as $p\to \infty$ 
of the Bergman kernel $P_p(x,x^\prime)$ of $D_p$ 
on the symplectic manifold $(X,\om)$
by working directly on the base manifold. The main idea of their proof is that 
the positivity of the bundle $L$ implies the existence of a spectral gap 
of the square of the $\text{spin}^c$ Dirac operator, which in turn insures 
that the problem can be localized and transferred to 
the tangent space of a point of the manifold. 

We are thus lead to study the model operator $\cL$ on $\C^n$, its
Bergman projection $\cP$ and Bergman kernel $\cP(Z,Z')$.
The strategy of our approach is to first study the calculus of kernels 
of the type $(F\cP)(Z,Z')$ on $\C^n$, where $F\in\C[Z,Z']$ is a polynomial. 

Using this calculus, the asymptotic expansion as $p\to\infty$ of 
the Bergman kernel of $D_p$ from \cite{DLM04a} and 
the Taylor series expansion of the sections $f$ and 
$g\in \cC^\infty(X,\End(E))$, we find the asymptotic expansion 
of the kernel of $T_{f,\,p}$ (cf. Lemma  \ref{toet2.3}),
and we establish that this kind of asymptotic expansion is also 
a sufficient condition for a family of operators to be a Toeplitz operator
(cf. Theorem \ref{toet3.1}).
In this way, we conclude from the asymptotic expansion of $T_{f,\,p}\,T_{g,\,p}$
that $T_{f,\,p}\,T_{g,\,p}$ is a Toeplitz operator 
in the sense of Definition \ref{toe-def}. 

The following result is one of our main results in this paper.
%----------------------------------------------------------------------------
\begin{thm}\label{toet4.1}
Let $(X,J,\om)$ be a compact symplectic manifold,
$(L,h^L,\nabla^L)$, $(E,h^E,\nabla^E)$ be Hermitian vector bundles as above,
and $g^{TX}$ be an $J$-compatible Riemannian metric on $TX$.

Let $f,g\in\cC^\infty(X,\End(E))$. Then the product of the Toeplitz operators 
$T_{f,\,p}$ and  $T_{g,\,p}$ is a Toeplitz operator, 
more precisely, it admits the asymptotic expansion
in the sense of \eqref{atoe2.2}{\rm:}
\begin{equation}\label{toe4.2}
T_{f,\,p}\,T_{g,\,p}=\sum^\infty_{r=0}p^{-r}T_{C_r(f,g),\,p}+\mO(p^{-\infty}),
\end{equation} 
where $C_r$ are bidifferential operators and $C_r(f,g)\in\cC^\infty(X,\End(E))$ and $C_0(f,g)=fg$.

If $f,g\in\cC^\infty(X)$, we have 
\begin{equation}\label{toe4.3}
C_1(f,g)-C_1(g,f)= \sqrt{-1}\{f,g\} \Id_E,
\end{equation} 
and therefore
\begin{equation}\label{toe4.4}
[T_{f,\,p}\,,T_{g,\,p}]=\frac{\sqrt{-1}}{\, p}T_{\{f,g\},\,p}+\mO(p^{-2}).
\end{equation} 
\end{thm}
%-------------------------------------------------------------------------
In conclusion, the set of Toeplitz operators forms an algebra. 
Moreover, the Berezin-Toeplitz quantization has the correct 
semi-classical behavior (cf.\,Theorem \ref{toet4.2}). 
In particular, when $(X,J,\om)$ is a compact K{\"a}hler manifold and $E=\C$,
$g^{TX}=\om(\cdot,J\cdot)$, these results give a new proof 
of \eqref{0.1}-\eqref{0.3} (cf.\,Remark \ref{toet5.6}).
Some related results %to this paper 
were also announced in \cite{BU}.

Note that we have established the off-diagonal 
asymptotic expansion of the Bergman kernel for certain non-compact
manifolds \cite[\S 3.5]{MM04a} (e.g., quasi-projective manifolds) and for orbifolds \cite[\S 5.2]{DLM04a}.
By combining these results and the method in this paper,
we carry the Berezin-Toeplitz quantization over to these cases 
(cf.\,Theorems \ref{toet5.1}, \ref{toet6.7}, \ref{toet6.12}).

As explained as above, an interesting corollary of our results 
is a canonical geometric construction 
of associated star products \eqref{0.3} in several cases.
We refer to Fedosov's book \cite{Fedo:96} for 
a construction of formal star products on symplectic manifolds
and to Pflaum \cite{Pflaum:03} for the generalization to orbifolds.
%for its extension to orbifolds by Pflaum see \cite{Pflaum:03}.
Related results appear in \cite{Charles2,PPT07}. 

We refer the readers to our book \cite{MM:07}
for a comprehensive study of the Berezin-Toeplitz quantization
along the lines of the present paper.

For the reader's convenience, we conclude the introduction with a brief outline of the paper.
We begin in Section \ref{toes1} by explaining the formal calculus on $\C^n$ for the model operator $\cL$.
In Section \ref{s3}, we recall the definition of the spin$^c$ Dirac operator 
and the asymptotic expansion of the Bergman kernel obtained in \cite{DLM04a}.
In Section \ref{ts2}, we establish the characterization of Toeplitz operators
in terms of their kernel. As a consequence, we establish
that the set of Toeplitz operators forms an algebra.
Finally, in Sections \ref{toes5} and \ref{pbs4}, we study the Berezin-Toeplitz 
quantization for non-compact manifolds and orbifolds.

We will use the following notations throughout. 
For $\alpha=(\alpha_1,\cdots,\alpha_{n})\in \N^{n}$, 
$B=(B_1,\cdots,B_n)\in \C^n$, we set
\begin{equation*}
|\alpha|= \sum_{j=1}^n \alpha_j,\quad \alpha ! =\prod_j (\alpha_j !),
\quad B^\alpha = \prod_j B_j^{\alpha_j}\,.
\end{equation*}

%%%%%%%%%%%%%%%%%%%%%%%%%%%%%%%%%%%%%%%%%%%%%%%%%%%%%%%%%%%%%%%%%%%%%%%
\section{Kernel calculus on $\C^n$}\label{toes1}
%%%%%%%%%%%%%%%%%%%%%%%%%%%%%%%%%%%%%%%%%%%%%%%%%%%%%%%%%%%%%%%%%%%%%%%

In this Section we explain the formal calculus on $\C^n$ 
for our model operator $\cL$, and we derive the properties of the calculus
of the kernels $(F\cP)(Z,Z^\prime)$, where $F\in\C[Z,Z^\prime]$ and
$\cP(Z,Z^\prime)$ is the kernel of the projection on the null space 
of the model operator $\cL$. 
This calculus is the main ingredient of our approach.

Let us consider the canonical coordinates $(Z_1,\dotsc,Z_{2n})$ 
on the real vector space $\R^{2n}$. On the complex vector space $\C^n$ 
we consider the complex coordinates $(z_1,\dotsc,z_n)$. 
The two sets of coordinates are linked by the relation 
$z_j=Z_{2j-1}+\imat Z_{2j}$, $j=1,\dotsc,n$.

We consider the $L^2$-norm 
$\norm{\,\cdot\,}_{L^2}=\big(\int_{\R^{2n}}\abs{{\,\cdot\,}}^2\,dZ\big)^{1/2}$
on $\R^{2n}$, where $dZ=dZ_1\cdots dZ_{2n}$ 
is the standard Euclidean volume form.

Let $0<a_1\leqslant a_2\leqslant\dotsc\leqslant a_n$.
We define the differential operators:
\begin{equation}\label{toe1.1}
\begin{split}
&b_i=-2{\tfrac{\partial}{\partial z
_i}}+\frac{1}{2}a_i\overline{z}_i\,,\quad
b^{+}_i=2{\tfrac{\partial}{\partial\overline{z}_i}}+\frac{1}{2}a_i
z_i\,,\\
&b=(b_1,\ldots,b_n)\,.
\end{split}
\end{equation}
Then $b^{+}_i$ is the adjoint of $b_i$ on 
$(L^2(\R^{2n}), \norm{\,\cdot\,}_{L^2})$. Set
\begin{equation}\label{toe1.1a}
\cL=\sum_i b_i\, b^{+}_i\,.
\end{equation}
Then $\cL$ acts as a densely defined self-adjoint operator on 
$(L^2(\R^{2n}),\norm{\,\cdot\,}_{L^2})$. 
%----------------------------------------------------------------------------
\begin{thm}\label{bkt2.17}
The spectrum of $\cL$ on $L^2(\R^{2n})$ is given by
\begin{equation}\label{bk2.68}
\spec(\cL)=
\Big\lbrace2\sum_{i=1}^n\alpha_i a_i\,:\, 
 \alpha =(\alpha_1,\cdots,\alpha_n)\in\N ^n\Big\rbrace
\end{equation}
and an orthogonal basis of the eigenspace of $2\sum_{i=1}^n\alpha_i a_i$
is given by
\begin{equation}\label{bk2.69}
b^{\alpha}\big(z^{\beta}\exp\big({-\frac{1}{4}\sum_{i=1}^n
a_i|z_i|^2}\big)\big)\,,\quad\text{with $\beta\in\N^n$}\,.
\end{equation}
In particular, an orthonormal basis of 
$\ke (\cL)$ is
\begin{equation}\label{bk2.70}
\varphi_\beta(z)=\left(\frac{a ^\beta}{(2\pi)^n 2 ^{|\beta|} \beta!}
\prod_{i=1}^n a_i\right)^{1/2}z^\beta
\exp\Big (-\frac{1}{4} \sum_{j=1}^n a_j |z_j|^2\Big )\,,\quad \beta\in\N^n\,.
\end{equation}

\end{thm}
%---------------------------------------------------------------------------
\noindent
For a proof we refer to 
\cite[Theorem 1.15]{MM04a} (cf. also \cite[Theorem\,4.1.20]{MM:07}). 
Let $\cP(Z,Z')$ denote the kernel of the orthogonal projection 
$\cP:L^2 (\R^{2n})\longrightarrow\ke(\cL)$ with respect to $dZ$. 
We call $\cP(\cdot,\cdot)$ the Bergman kernel of $\cL$.

It is easy to see that 
$\cP(Z,Z')=\sum_{\beta}\varphi_\beta(z)\,\overline{\varphi_\beta(z')}$.
We infer the following formula for the kernel $\cP(Z,Z')$:
\begin{equation}\label{toe1.3}
\cP(Z,Z') =\prod_{i=1}^n
\frac{a_i}{2\pi}\:\:\exp\Big(-\frac{1}{4}\sum_{i=1}^n
a_i\big(|z_i|^2+|z^{\prime}_i|^2 -2z_i\overline{z}_i'\big)\Big)\,.
\end{equation}

In the calculations involving the kernel $\cP(\cdot,\cdot)$,
 we prefer however to use the orthogonal decomposition of $L^2(\R^{2n})$ 
given in Theorem \ref{bkt2.17} and the fact that $\cP$ is 
an orthogonal projection, rather than
integrating against the expression \eqref{toe1.3} of $\cP(\cdot,\cdot)$. 
This point of view helps simplify a lot the computations 
and understand better the operations.
As an example, if $\varphi(Z)= b^\alpha z^\beta
\exp\Big (-\frac{1}{4} \sum_{j=1}^n a_j |z_j|^2\Big )$ 
with $\alpha,\beta\in \N^n$, then Theorem \ref{bkt2.17} 
implies immediately that
 \begin{align}\label{toe1.4}
(\cP \varphi)(Z)= \left \{  \begin{array}{ll}
\displaystyle{z^\beta \exp\Big (-\frac{1}{4} \sum_{j=1}^n a_j |z_j|^2\Big ) }
& \mbox{if} \,\, |\alpha|=0,  \\
 0& \mbox{if} \,\, |\alpha|>0.\end{array}\right. 
\end{align}

In the rest of this Section, all operators are defined by their kernels
with respect to $dZ$. In this way, if $F$ is a polynomial on $Z,Z^\prime$,
then $F\cP$ is an operator on $L^2(\R^{2n})$ with kernel 
$F(Z,Z^\prime)\cP(Z,Z^\prime)$ with respect to $dZ$.

We will add a subscript $z$ or
$z^{\prime}$ when we need to specify the operator is acting on the
variables $Z$ or $Z^{\prime}$.
%---------------------------------------------------------------------
\begin{lemma}\label{toet1.1}
For any polynomial $F(Z, Z^{\prime})\in\C[Z, Z^{\prime}]$,
there exist polynomials $F_\alpha\in\C[z, Z^{\prime}]$ 
and $F_{\alpha,0}\in\C[z, \overline{z}^{\prime}]$, $(\alpha \in \N ^{n})$
such that
%----------------------------------------------------------------------------- 
\begin{gather}
(F\cP)(Z, Z^{\prime}) =\sum_{\alpha} b^{\alpha}_z
(F_{\alpha}\cP)( Z, Z^{\prime}),\label{toe1.5}\\
((F\cP) \circ \cP)(Z, Z^{\prime}) =
 \sum_{\alpha} b^{\alpha}_z F_{\alpha,0}(z, \overline{z}^{\prime})
 \cP( Z, Z^{\prime}).\label{toe1.5a}
\end{gather}
%----------------------------------------------------------------------------- 
Moreover, $|\alpha|+ \deg F_{\alpha}$,
$|\alpha|+ \deg F_{\alpha,0}$ have the same parity
with the degree of $F$ in $Z, Z^{\prime}$.
In particular, $F_{0,0}(z, \overline{z}^{\prime})$
is a polynomial in $z$, $\overline{z}^{\prime}$ and its degree
has the same parity with $\deg F$.

For any polynomials
$F,G\in\C[Z, Z^{\prime}]$ there exist polynomial
$\cK[F,G]\in\C[Z, Z^{\prime}]$ such that
\begin{equation}\label{toe1.6}
((F\cP) \circ (G\cP))(Z, Z^{\prime}) =
\cK[F,G](Z, Z^{\prime}) \cP( Z, Z^{\prime}).
\end{equation}
\end{lemma}
%---------------------------------------------------------------------------- 
\begin{proof}
Note that from \eqref{toe1.1} and \eqref{toe1.3}, 
for any polynomial $g(z,\ov{z})\in\C[z,\ov{z}]$, we get
\begin{align}\label{toe1.7}
\begin{split}
%b^+_{j,z}\cP=0,\quad
&b_{j\,,z}\,\cP(Z, Z^{\prime}) = a_j (\ov{z}_j- \ov{z}_j^{\prime})
\cP(Z, Z^{\prime}),\\
&[g(z,\ov{z}),b_{j\,,z}]=2\frac{\partial}{\partial z_j}g(z,\ov{z})\,.
\end{split}\end{align}
Let $F(Z, Z^{\prime})\in\C[Z, Z^{\prime}]$.
Using repeatedly \eqref{toe1.7} 
we can replace $\ov{z}$ in the expression of $F(Z, Z^{\prime})$ 
by a combination of $b_{j,z}$ and $\ov{z}^\prime$ and \eqref{toe1.5} follows.
We deduce from \eqref{toe1.4} and \eqref{toe1.5} that there exists 
$F_0\in \C[z,Z']$ such that
 \begin{equation}\label{toe1.8}
(\cP\circ (F\cP))(Z,Z^\prime)= (F_0\cP)(Z,Z^\prime)\,.
\end{equation}
We apply now \eqref{toe1.8} for $\ov{F}$ instead of $F$ and 
take the adjoint of the so obtained equality. 
Since $\cP$ is self-adjoint, this implies the existence of 
a polynomial $F^\prime$ in $Z,\ov{z}^\prime$ such that
\begin{equation*}
((F\cP)\circ\cP)(Z,Z^\prime) = F^\prime(Z, \ov{z}^\prime) \cP(Z,Z^\prime). 
\end{equation*}
The latter formula together with \eqref{toe1.5} imply \eqref{toe1.5a}.
Finally, \eqref{toe1.6} results from \eqref{toe1.5} and \eqref{toe1.5a}.
\end{proof}
%------------------------------------------------------------------------- 
\begin{ex}
We illustrate how Lemma \ref{toet1.1} works. 
Observe that \eqref{toe1.7} entails
\begin{equation}\label{toe1.7a1}
\ov{z}_j\,\cP(Z, Z^{\prime})=\frac{b_{j\,,z}}{a_j}\,
\cP(Z, Z^{\prime})+\ov{z}_j^{\,\prime}\cP(Z, Z^{\prime}).
\end{equation} 
Moreover, specializing \eqref{toe1.7} for $g(z,\overline{z})=z_i$ we get
\begin{equation}\label{toe1.7b1}
z_i\,b_{j\,,z}\cP(Z, Z^{\prime})
=b_{j\,,z}(z_i\cP)(Z, Z^{\prime})+2\delta_{ij}\cP(Z, Z^{\prime}),
\end{equation}
Formulas \eqref{toe1.7a1} and \eqref{toe1.7b1} give
\begin{align} \label{toe1.11}
\begin{split}
z_i\ov{z}_j\,\cP(Z, Z^{\prime})
&=\frac{1}{a_j}\,b_{j\,,z}\,z_i\cP(Z, Z^{\prime})
+\frac{2}{a_j}\, \delta_{ij}
\cP(Z, Z^{\prime})+z_i\ov{z}^{\,\prime}_j\,\cP(Z, Z^{\prime})\,.
\end{split}\end{align}

Using the preceding formula we calculate further some examples for 
the expression $\cK[F,G]$ introduced \eqref{toe1.6}. 
{}Indeed, equations \eqref{toe1.4}, 
\eqref{toe1.7a1} and \eqref{toe1.11} imply that
\begin{align} \label{toe1.12}
\begin{split}
&\cK[1,\ov{z}_j]\cP =\cP\circ(\ov{z}_j\cP)
=\ov{z}^\prime_j\cP,\quad
\cK[1,z_j]\cP= \cP\circ(z_j\cP)= z_j\cP,\\
&\cK[z_i,\ov{z}_j]\cP =(z_i \cP)\circ(\ov{z}_j\cP)
= z_i  \cP\circ(\ov{z}_j\cP)=z_i\ov{z}^\prime_j\cP,\\
&\cK[\ov{z}_i,z_j]\cP= (\ov{z}_i\cP)\circ(z_j\cP)
= \ov{z}_i\cP\circ(z_j\cP) = \ov{z}_i z_j\cP,\\
&\cK[z_i^\prime,\ov{z}_j]\cP =(z_i^\prime \cP)\circ(\ov{z}_j\cP)
=  \cP\circ(z_i\ov{z}_j\cP)= \frac{2}{a_j}\delta_{ij} \cP
+z_i \ov{z}_j^\prime\cP,\\
&\cK[\ov{z}_i^\prime,z_j]\cP= (\ov{z}_i^\prime\cP)\circ(z_j\cP)
= \cP\circ(\ov{z}_i z_j\cP) = \frac{2}{a_j}\delta_{ij} \cP
+\ov{z}_i^\prime z_j\cP.
\end{split}\end{align}
Thus we get 
\begin{align} \label{toe1.13}
\begin{split}
& \cK[1,\ov{z}_j]= \ov{z}^\prime_j,\quad \cK[1,z_j]=z_j,\\
&\cK[z_i,\ov{z}_j]=z_i\ov{z}^\prime_j,\quad \cK[\ov{z}_i,z_j]= \ov{z}_i z_j,\\
&\cK[\ov{z}_i^\prime,z_j]= \cK[z_j^\prime,\ov{z}_i]
= \frac{2}{a_j}\delta_{ij} +\ov{z}_i^\prime z_j.
\end{split}\end{align}
\end{ex}
%-------------------------------------------------------------------------
\begin{notation} \label{toet1.4}
To simplify our calculations, we introduce the following notation. 
For any polynomial $F\in\C[Z,Z^{\prime}]$ we denote by $(F\cP)_p$ 
the operator defined by the kernel 
$p^{n}(F\cP)(\sqrt{p}Z, \sqrt{p}Z^{\prime})$, that is,
\begin{equation}
((F\cP)_p \varphi)(Z)=\int_{\R^{2n}}p^{n}
(F\cP)(\sqrt{p}Z, \sqrt{p}Z^{\prime})\varphi(Z^{\prime})\,dZ^{\prime}\,,\quad
\text{for any $\varphi\in L^2(\R^{2n})$.}
\end{equation}
Let $F,G\in \C[Z,Z^\prime]$. By a change of variables we obtain
\begin{align} \label{toe1.15}
((F\cP)_p\circ (G\cP)_p)(Z,Z^\prime)
= p^{n}((F\cP)\circ (G\cP))(\sqrt{p}Z, \sqrt{p}Z^{\prime}).
\end{align}
\end{notation}

%%%%%%%%%%%%%%%%%%%%%%%%%%%%%%%%%%%%%%%%%%%
\section{Bergman kernels on symplectic manifolds}\label{s3}

This Section is organized as follows. 
We recall the definition of the 
spin$^c$ Dirac operator in Section \ref{s3.0}, and in Section \ref{s5.3},
 we explain the asymptotic expansion of the Bergman kernel.

\subsection{The spin$^c$ Dirac operator}
\label{s3.0}

Let $X$ be a compact manifold  of real dimension $2n$ with 
almost complex structure $J$. Let $g^{TX}$ be a Riemannian metric on 
$X$ compatible with $J$, i.e., $g^{TX}(J\cdot, J\cdot)=g^{TX}(\cdot, \cdot)$.

The almost complex structure $J$ induces a splitting
of the complexification of the tangent bundle, 
$TX\otimes_\R \C=T^{(1,0)}X\oplus T^{(0,1)}X$, 
where $T^{(1,0)}X$ and $T^{(0,1)}X$
are the eigenbundles of $J$ corresponding to the eigenvalues $\sqrt{-1}$
and $-\sqrt{-1}$ respectively.
Let $P^{(1,0)}=\frac{1}{2}(1-\sqrt{-1}J)$ 
and  $P^{(0,1)}$ be the natural projections from 
$TX\otimes_\R \C$ onto  $T^{(1,0)}X$ and $T^{(0,1)}X$. 
 Accordingly, we have a decomposition of the
complexified cotangent bundle: $T^{\ast}X\otimes_\R \C=T^{\ast\,(1,0)}X\oplus
T^{\ast\, (0,1)}X$. The exterior algebra bundle decomposes as
$\Lambda (T^{\ast}X)\otimes_\R \C=\oplus_{p,q}\Lambda^{p,q}(T^{\ast}X)$, where
$\Lambda^{p,q}(T^{\ast}X):
=\Lambda^{p}(T^{\ast\,(1,0)}X)\otimes \Lambda^{q}(T^{\ast\,(0,1)}X)$.

Let $\nabla^{TX}$ be the Levi--Civita connection of $(TX, g^{TX})$ 
with associated curvature $R^{TX}$. %and scalar curvature $r^X$. 
Let $\nabla^XJ\in T^*X\otimes \End(TX)$ be the covariant derivative 
of $J$ induced by $\nabla^{TX}$. Set 
\begin{equation}
\begin{split}
\nabla^{T^{(1,0)}X}&=P^{(1,0)}\,\nabla^{TX}\,P^{(1,0)}\,,\quad
\nabla^{T^{(0,1)}X}=P^{(0,1)}\,\nabla^{TX}\,P^{(0,1)}\,,\\
{^0\nabla}^{TX}&=\nabla^{T^{(1,0)}X}\oplus\nabla^{T^{(0,1)}X},\quad
A_2=\nabla^{TX}-{^0\nabla}^{TX}.
\end{split}
\end{equation}
Then  $\nabla^{T^{(1,0)}X}$ and $\nabla^{T^{(0,1)}X}$ are the canonical 
Hermitian connections on $T^{(1,0)}X$ and $T^{(0,1)}X$ respectively
with curvatures $R^{T^{(1,0)}X}$ and $R^{T^{(0,1)}X}$.
Moreover, ${^0\nabla}^{TX}$ is an Euclidean connection on $TX$.
The tensor  $A_2\in T^{\ast}X\otimes\End(TX)$ satisfies 
\begin{align}\label{frame0}
A_2=\frac{1}{2} J(\nabla^XJ),\quad\quad J\,A_2=-A_2\,J.
\end{align}

For any $v\in TX$ with decomposition 
$v=v_{1,0}+v_{0,1} \in T^{(1,0)}X\oplus T^{(0,1)}X$, let
${\overline v^\ast_{1,0}}\in T^{\ast\,(0,1)}X$ be the metric dual of 
$v_{1,0}$. Then 
\begin{align}
\mathbf{c}(v)=\sqrt{2}({\overline v^\ast_{1,0}}\wedge-i_{v_{\,0,1}})
\end{align}
defines the Clifford action of $v$ on 
$\Lambda^{0,\scriptscriptstyle\bullet}=\Lambda^{\text{even}}
(T^{\ast\,(0,1)}X)\oplus\Lambda^{\text{odd}} (T^{\ast\,(0,1)}X)$, 
where $\wedge$ and $i$ 
denote the exterior and interior product respectively.

The connection $\nabla^{T^{(1,0)}X}$ on $T^{(1,0)}X$ induces naturally 
a Hermitian connection 
 $\nabla^{\Lambda^{0,\scriptscriptstyle\bullet}}$ on 
$\Lambda^{0,\scriptscriptstyle\bullet}=\Lambda^{\bullet}
(T^{\ast\,(0,1)}X)$ which preserves the natural $\Z$-grading on 
$\Lambda^{0,\scriptscriptstyle\bullet}$.
 Let $\{w_j\}_{j=1}^n$ be a local orthonormal frame of $T^{(1,0)}X$.
Let $\{w^j\}_{j=1}^n$ be the dual frame of $\{w_j\}_{j=1}^n$. Then
\begin{equation}\label{frame1}
e_{2j-1}=\tfrac{1}{\sqrt{2}}(w_j+\overline{w}_j)\quad\text{and}\quad
e_{2j}=\tfrac{\sqrt{-1}}{\sqrt{2}}(w_j-\overline{w}_j)\,, 
\quad j=1,\dotsc,n\,,
\end{equation}
form an orthonormal frame of $TX$. Set
\begin{equation}\label{local}
\begin{split}
\mathbf{c}(A_2)&=\frac{1}{4} \sum_{i,j}  
\big\langle A_2 e_i, e_j\big\rangle \mathbf{c}(e_i)\mathbf{c}(e_j)\\
&=\frac{1}{2} \sum_{l,m}\Big( \big\langle A_2 w_l,w_m
\big\rangle\,i_{\overline{w}_l}\,i_{\overline{w}_m}+
\big\langle A_2\overline{w}_l,\overline{w}_m
\big\rangle\,\overline{w}^l\wedge\,\overline{w}^m\wedge \Big)\,, \\
\nabla^{\text{Cliff}}&=\nabla^{\Lambda^{0,\scriptscriptstyle\bullet}}
+\mathbf{c}(A_2).
\end{split}
\end{equation}
The connection $\nabla^{\text{Cliff}}$ is the Clifford connection
on $\Lambda^{0,\scriptscriptstyle{\bullet}}$
induced canonically by $\nabla^{TX}$ (cf. \cite[\S 2]{MM02}).
(Note that in the definition of the Clifford connection in \cite[(2.3)]{MM02}, 
one should add the term ``\,$+\frac{1}{2} \tr|_{T^{(0,1)}X} \Gamma$\,'' 
in the right hand side of
the first line, and the second line should read 
``\,$=d+\sum_{lm}\lbrace\big\langle\Gamma w_l,\overline{w}_m
\big\rangle\,\overline{w}^{\,m}\wedge\,i_{\overline{w}_l}\,+\,$\,''.)

Let $(E,h^E)$ be a Hermitian vector bundle on $X$ 
with Hermitian connection $\nabla^E$ and curvature $R^E$.
Let $(L,h^L)$ be a Hermitian line bundle over $X$ 
endowed with a Hermitian connection $\nabla^L$ with curvature 
$R^L=(\nabla^L)^2$. We assume that $(L,\nabla^L)$ satisfies the 
\textbf{\em prequantization condition}\/, that is 
\begin{equation}\label{bk1.1}
\om(\cdot, J\cdot)>0, \quad \om(J\cdot, J\cdot)=\om(\cdot, \cdot)\,, 
\quad \text{where $\omega:=\frac{\sqrt{-1}}{2\pi}R^L$}\,.
\end{equation} 
This implies in particular that $\om$ is a symplectic form on $X$.

We relate $g^{TX}$ with $\om$ by means of the skew--adjoint 
linear map ${  \bJ}:TX\longrightarrow TX$ which satisfies the relation
\be \label{to1.2}
\om(u,v)=g^{TX}({  \bJ}u,v)\quad\text{for}\quad u,v \in TX.
\ee
Then $J$ commutes with ${  \bJ}$, and $J={  \bJ}(-{  \bJ}^2)^{-\frac{1}{2}}$.

We denote
\begin{equation}
E_p:=\Lambda^{0,\scriptscriptstyle{\bullet}}\otimes L^p\otimes E.
\end{equation}
Along the fibers of $E_p$,
we consider the pointwise Hermitian product $\langle\cdot,\cdot\rangle$  
induced by $g^{TX}$, $h^L$ and $h^E$. 
Let $dv_X$ be the Riemannian volume form of $(TX, g^{TX})$.
The $L^2$--Hermitian product on the space 
$\Omega^{0,\scriptscriptstyle{\bullet}}(X,L^p\otimes E)$ 
of smooth sections of $E_p$ is given by 
\begin{equation}\label{l2}
\langle s_1,s_2\rangle=\int_X\langle s_1(x),s_2(x)\rangle\,dv_X(x)\,.
\end{equation}
We denote the corresponding norm with $\norm{\cdot}_{L^2}$ and 
with $L^2(X,E_p)$ the completion of 
$\Omega^{0,\scriptscriptstyle{\bullet}}(X,L^p\otimes E)$ 
with respect to this norm.

Let $\nabla^{L^p\otimes E}$ be the connection on $L^p\otimes E$ induced 
by $\nabla^L$ and  $\nabla^E$. Let $\nabla^{E_p}$  be the connection on 
$E_p$ induced by $\nabla^{\text{Cliff}}$, $\nabla^{L^p\otimes E}$: 
\begin{equation}
\nabla^{E_p}=
\nabla^{\text{Cliff}}\otimes\Id+\Id\otimes\nabla^{L^p\otimes E}.
\end{equation}

%------------------------------------------------------------------------
\begin{defn}\label{defDirac1}
The \textbf{\em\spin Dirac operator} $D_p$ is defined by 
\begin{equation}\label{defDirac}
D_p=\sum_{j=1}^{2n}\mathbf{c}(e_j)\nabla^{E_p}_{e_j}:
\Omega^{0,\scriptscriptstyle{\bullet}}(X,L^p\otimes E)\longrightarrow
\Omega^{0,\scriptscriptstyle{\bullet}}(X,L^p\otimes E)\,.
\end{equation}
\end{defn}
%---------------------------------------------------------------------
\noindent
$D_p$ is a formally self--adjoint, first order elliptic differential operator 
on $\Omega^{0,\scriptscriptstyle{\bullet}}(X,L^p\otimes{E})$,
which interchanges $\Omega^{0,\text{even}}(X,L^p\otimes E)$
and $\Omega^{0,\text{odd}}(X,L^p\otimes E)$ (cf. \cite[\S 1.3]{MM:07}).
%------------------------------------------------------------------------------
\begin{defn}\label{defBergman}
The orthogonal projection
\begin{equation}
P_p:L^2(X,E_p)\longrightarrow\ke(D_p)
\end{equation}
is called the \textbf{\em Bergman projection}\/. 
Let $\pi_1$ and $\pi_2$ be the projections of $X\times X$ on the first and 
second factor. Since $P_p$ is a smoothing operator, 
the Schwartz kernel theorem \cite[p.\,296]{Tay1:96}, \cite[Th.\,B.2.7]{MM:07}
shows that the Schwartz kernel of $P_p$ is smooth, i.e., 
there exists a section 
$P_p(\cdot,\cdot)\in\cC^\infty(X\times X,\pi_1^*(E_p)\otimes \pi_2^*(E_p^*))$ 
such that for any $s\in L^2(X,E_p)$ we have
\begin{equation}\label{to1.12}
(P_p\,s)(x)=\int_X P_p(x,x')s(x')\,dv_X(x')\,.
\end{equation}

The smooth kernel $P_p(\cdot,\cdot)$ is called the \textbf{\em Bergman kernel} 
of $D_p$. Observe that $P_p(x,x)$ is an element of 
$\End(\Lambda (T^{*(0,1)}X)\otimes E)_x$\,.  
\end{defn}
%------------------------------------------------------------------------------

We wish to describe the kernel and spectrum of $D_p$ in the sequel.
For any operator $A$, we denote by $\spec (A)$ the spectrum of $A$.

Recall that $\{w_i\}$ is an orthonormal frame of $(T^{(1,0)}X, g^{TX})$.
Set
\begin{align}\label{to1.13}
\begin{split}
&\om_d=-\sum_{l,m} R^L (w_l,\overline{w}_m)\,\overline{w}^m\wedge
\,i_{\overline{w}_l}\,,\\
& \tau(x)= \sum_j R^L (w_j,\overline{w}_j)= -\pi \tr|_{TX} [J\bJ]\,,\\
&\mu_0=\inf \,\lbrace R^L_x(u,\overline{u})/|u|^2_{g^{TX}}\,:\, 
u\in T_x^{(1,0)}X,\,x\in X\rbrace>0\,.
\end{split}\end{align}
The following result was proved in \cite[Theorems\,1.1,\,2.5]{MM02} 
as an application of
the Lichnerowicz formula \cite[Theorem\,3.52]{BeGeVe} 
(cf. also \cite[Theorem\, 1.3.5]{MM:07}) for $D_p^2$.
%------------------------------------------------------------------------------
\begin{thm}\label{specDirac}
There exists $C>0$ such that for any $p\in\N$,
$s\in\Omega^{0,>0}(X,L^p\otimes E)
:=\bigoplus_{k>0}\Omega^{0,k}(X,L^p\otimes E)$, 
\begin{equation}\label{main1}
\norm{D_{p}s}^2_{L^2}\geqslant(2p\mu_0-C)\norm{s}^2_{L^2}\, .
\end{equation}
Moreover,  
\begin{equation}\label{diag5}
\spec (D^2_p) \subset \{0\}\cup [2p\mu_0-C,+\infty[\,. 
\end{equation}
\end{thm}
%------------------------------------------------------------------------------

%%%%%%%%%%%%%%%%%%%%%%%%%%%%%%%%%%%%%%%%%%%%%%%%%%%%%%%%%%%%%%%%%%%%%%
\subsection{Off-diagonal asymptotic expansion of Bergman kernel}\label{s5.3}
The existence of the spectral gap expressed in Theorem \ref{specDirac}
allows us to \textit{\textbf{localize}} the behavior of the Bergman kernel.

Let $a^X$ be the injectivity radius of $(X, g^{TX})$.
We denote by $B^{X}(x,\var)$ and  $B^{T_xX}(0,\var)$ the open balls
in $X$ and $T_x X$ with center $x$ and radius $\var$, respectively.
Then the exponential map $ T_x X\ni Z \to \exp^X_x(Z)\in X$ is a
diffeomorphism from $B^{T_xX} (0,\var)$ onto $B^{X} (x,\var)$ for
$\var\leqslant a^X$.  {}From now on, we identify $B^{T_xX}(0,\var)$
with $B^{X}(x,\var)$ via the exponential map for $\var \leqslant a^X$.
Throughout what follows, $\varepsilon$ runs in the fixed intervall $]0, a_X/4[$. % we shall fix $\var\in ]0,a^X/4[$\,.

Let ${\mathbf{f}} : \R \to [0,1]$ be a smooth even function such that
${\mathbf{f}}(v)=1$ for $|v| \leqslant  \var/2$, 
and ${\mathbf{f}}(v) = 0$ for $|v| \geqslant \var$. Set
\begin{equation} \label{0c3}
F(a)= \Big(\int_{-\infty}^{+\infty}{\mathbf{f}}(v) dv\Big)^{-1} 
\int_{-\infty}^{+\infty} e ^{i v a}\, {\mathbf{f}}(v) dv.
\end{equation}
Then $F(a)$ is an even function and lies in the Schwartz space 
$\mathcal{S} (\R)$ and $F(0)=1$.

\noindent
By \cite[Proposition\,4.1]{DLM04a}, we have the \textbf{\em far
off-diagonal} behavior of the Bergman kernel: 
%------------------------------------------------------------------------------
\begin{prop}\label{0t3.0}
For any $l,m\in\N$ and $\var>0$, there exists $C_{l,m,\var}>0$ 
such that for any $p\geqslant 1$, $x,x'\in X$, the following estimate holds:
\begin{equation}\label{0c7}
\left|F\big(D_p)(x,x') 
- P_{p}(x,x')\right|_{\cC ^m(X\times X)}
\leqslant C_{l,m,\var} p^{-l}.
\end{equation}
Especially, for $d(x,x')> \var$, 
\begin{equation}\label{1c3}
|P_p(x,x')|_{\cC^m(X\times X)} \leqslant C_{l,m,\var}\, p^{-l}\,,
\end{equation}
The $\cC ^m$ norm in \eqref{0c7} and \eqref{1c3} is induced by
$\nabla^L$, $\nabla^E$, $h^L$, $h^E$ and $g^{TX}$.
\end{prop}

We consider the orthogonal projection:
\begin{equation}\label{2c3}
I_{\C\otimes E}: \bE:=\Lambda (T^{*(0,1)}X)\otimes E
\longrightarrow \C\otimes E\,.
\end{equation}
Let $\pi : TX\times_{X} TX \to X$ be the natural projection from the
fiberwise product of $TX$ on $X$.
  Let $\nabla ^{\End (\bE)}$ be the connection on
$\End (\Lambda (T^{*(0,1)}X)\otimes E)$ induced by
$\nabla ^{\text{Cliff}}$ and $\nabla ^E$.

Let us elaborate on the identifications we use in the sequel, 
which we state as a Lemma.
%-------------------------------------------------------        
\begin{lemma}\label{0l3.0}
Let $x_0\in X$ be fixed and consider 
the diffeomorphism $B^{T_{x_0}X}(0,4\varepsilon)\ni Z 
\to \exp^X_{x_0}(Z)\in B^{X}(x_0,4\varepsilon)$. 
We denote the pull-back of the vector
bundles $L$, $E$ and $E_p$ via this diffeomorphism by the same symbols. 
\begin{itemize}
%---------------------------------------------------------      
        \item[(i)] There exist trivializations of $L$, $E$ and $E_p$ over 
$B^{T_{x_0}X} (0,4\varepsilon)$ given by unit frames 
which are parallel with respect to $\nabla ^L$, $\nabla^E$ and
$\nabla^{E_p}$ along the curves $\gamma_Z:[0,1]\to B^{T_{x_0}X}(0,4\varepsilon)$ defined for every $Z\in B^{T_{x_0}X}(0,4\varepsilon)$ by
$\gamma_Z(u)=\exp^X_{x_0} (uZ)$.
%---------------------------------------------------------------------  
        \item[(ii)] With the previous trivializations, 
$P_p(x,x')$ induces a smooth section 
$B^{T_{x_0}X}(0,4\var)\ni Z,Z'\mapsto P_{p,\,x_0}(Z,Z')$
of $\pi ^* (\End (\Lambda (T^{*(0,1)}X)\otimes E))$ over $TX\times_{X} TX$, which depends smoothly on $x_0$.
%---------------------------------------------------------------------- 
\item[(iii)] $\nabla ^{\End (\bE)}$
induces naturally a $\cC^m$-norm with respect to the parameter $x_0\in X$.
%------------------------------------------------------------------     
\item[(iv)] If $dv_{TX}$ is the Riemannian volume form 
on $(T_{x_0}X, g^{T_{x_0}X})$, there exists 
a smooth positive function $\kappa_{x_0}:T_{x_0}X\to\R$, $Z\mapsto\kappa_{x_0}(Z)$ defined by 
\begin{equation} \label{atoe2.7}
dv_X(Z)= \kappa_{x_0}(Z) dv_{TX}(Z),\quad \kappa_{x_0}(0)=1,
\end{equation}
where the subscript $x_0$ of  
$\kappa_{x_0}(Z)$ indicates the base point $x_0\in X$.
%-------------------------------------------------------------------    
\item[(v)] By \eqref{to1.2}, $\bJ$ is an element of\/ $\End(T^{(1,0)}X)$.  
Consequently, we can diagonalize $\bJ_{x_0}$\/, i.e., choose an orthonormal
basis $\{w_j\}_{j=1}^n$ of $T^{(1,0)}_{x_0}X$ such that 
\begin{equation} \label{3ue60}
\bJ_{x_0}\om_j =\frac{\sqrt{-1}}{2\pi} a_j(x_0) w_j\,, \quad\text{for all $j=1,2,\ldots,n$}\,,
\end{equation}
where $0<a_1(x_0)\leqslant a_2(x_0)\leqslant\dotsc\leqslant a_n(x_0)$.
Then $\{e_{j}\}_{j=1}^{2n}$ defined in \eqref{frame1}
forms an orthonormal basis of $T_{x_0}X$.
The diffeomorphism 
\begin{equation}\label{alm4.22}
\R^{2n}\ni(Z_1,\dotsc, Z_{2n}) \longmapsto \sum_i
Z_i e_i\in T_{x_0}X
\end{equation}
induces coordinates on $T_{x_0}X$, which we use throughout the paper. In these coordinates we have $e_j={\partial}/{\partial Z_j}$\,.
%-------------------------------------------------------------------    
\end{itemize}
\end{lemma}

\noindent
Let $\nabla_U$ denote the ordinary differentiation
operator on $T_{x_0}X$ in the direction $U$.
We introduce the \textbf{\em model operator} $\cL$
 on $T_{x_0}X\simeq \R^{2n}$ by setting
\begin{equation} \label{3ue63}
\begin{split}
&\nabla_{0,U}:= \nabla_{U} +\frac{1}{2}R^L_{x_0}(Z,U)\,,
\quad\text{for $U\in T_{x_0}X$},\\
&\cL:=- \sum_j (\nabla_{0,e_j} )^2   -\tau(x_0)\,.
\end{split}\end{equation}
By \eqref{to1.13} and \eqref{3ue60}, $\tau(x_0)=\sum_j a_j(x_0)$.
The operator $\cL$ defined in \eqref{3ue63} coincides with the operator $\cL$ given by \eqref{toe1.1} and \eqref{toe1.1a}, with $a_j=a_j(x_0)$ for $1\leqslant j\leqslant n$.

We denote by $\det_{\C}$ for the determinant function on the complex
bundle $T^{(1,0)}X$ and set $|\bJ_{x_0}|=(-\bJ^2_{x_0})^{1/2}$.
The Bergman kernel $T_{x_0}X\ni Z,Z'\mapsto\cP(Z,Z')$
of $\cL$ has the following form in view of \eqref{toe1.3}:
\begin{equation}\label{3ue62}
\cP(Z,Z')  ={\det}_{\C} (|\bJ_{x_0}|)
\exp\Big (- \frac{\pi}{2} \left \langle |\bJ_{x_0}|(Z-Z'),(Z-Z') \right \rangle
-\pi \sqrt{-1} \left \langle \bJ_{x_0} Z,Z' \right \rangle\Big ).
\end{equation}

\noindent
By \cite[Theorem 4.18$^\prime$]{DLM04a} we have the 
\textbf{\em off diagonal expansion} of the Bergman kernel:
%%%%%%%%%%%%%%%%%%%%%%%%%%%%%%%%%%%%%%%%%%%%%%%%%%%%%%%%%%%%%%%%%%%%%%%%%%%%
\begin{thm} \label{tue17}
Let $\varepsilon\in]0, a_X/4[$. For every $x_0\in X$ and $r\in\N$ there exist polynomials $J_{r,\,x_0}(Z,Z^{\prime})\in
\End (\Lambda (T^{*(0,1)}X) \otimes E)_{x_0}$ ,
in $Z,Z^{\prime}$ with the same parity as $r$
and with $\deg J_{r,\,x_0}\leqslant 3r$, whose coefficients
are polynomials in $R^{TX}$, $R^{T^{(1,0)}X}$,
$R^E$ {\rm (}and $R^L${\rm )}  and their derivatives of order
$\leqslant r-1$  {\rm (}resp. $\leqslant r${\rm )} and reciprocals
of linear combinations of eigenvalues of $\bJ$  at $x_0$, 
 such that by setting
\begin{equation}\label{a0.7}
\cP^{(r)}_{x_0}(Z,Z^{\prime})=J_{r,\,x_0}(Z,Z^{\prime})\cP(Z,Z^{\prime}),
\quad J_{0,\,x_0}(Z,Z^{\prime})
=I_{\C\otimes E}\,\, ,
\end{equation}
the following statement holds:
There exists $C''>0$ such that for every $k,m,m'\in\N$, there exist $N\in\N$ 
and $C>0$ such that the following estimate holds
\begin{equation}\label{aue66}
\begin{split}
\left |\frac{\partial^{|\alpha|+|\alpha'|}}
{\partial Z^{\alpha} {\partial Z'}^{\alpha'}}
\left (\frac{1}{p^n}  P_p(Z,Z')
-\sum_{r=0}^k  \cP^{(r)} (\sqrt{p} Z,\sqrt{p} Z')
\kappa ^{-1/2}(Z)\kappa ^{-1/2}(Z')
p^{-r/2}\right )\right |_{\cC^{m'}(X)}\\
\leqslant C  p^{-(k+1-m)/2}  (1+|\sqrt{p} Z|+|\sqrt{p} Z'|)^N
\exp (- \sqrt{C''\mu_0 } \sqrt{p} |Z-Z'|)+ \cO(p^{-\infty}),
\end{split}
\end{equation}
for any $\alpha,\alpha'\in\N^{n}$, with $|\alpha|+|\alpha'|\leqslant m$, 
any $Z,Z'\in T_{x_0}X$ with $|Z|, |Z'|\leqslant  \var$ 
and any $x_0\in X$, $p\geqslant 1$.
\end{thm}
%%%%%%%%%%%%%%%%%%%%%%%%%%%%%%%%%%%%%%%%%%%%%%%%%%%%%%%%%%%%%%%%%%%%%%%%%%%%
Here $\cC^{m'}(X)$ is the  $\cC^{m'}$-norm for the parameter $x_0\in X$.
We say that a term $G_p=\cO(p^{-\infty})$ if for any $l,l_1\in\N$, there 
exists $C_{l,l_1}>0$ such that the $\cC^{l_1}$-norm of $G_p$ is dominated 
by $C_{l,l_1} p^{-l}$.
%%%%%%%%%%%%%%%%%%%%%%%%%%%%%%%%%%%%%%%%%%%%%%%%%%%%%%%%%%%%%%%%%%%%%%
\begin{rem}\label{absyt1.2} Set 
$\bE^+:=\oplus_{j}\Lambda^{2j}(T^{*(0,1)}X)\otimes E$ and 
$\bE^-:= \oplus_{j}\Lambda^{2j+1}(T^{*(0,1)}X)\otimes E$;
$E^-_p:=\bE^-\otimes L^p$ and  $E^+_p:=\bE^+\otimes L^p$.
By Theorem \ref{specDirac} and because $D_{p}^2$ 
preserves the $\Z_2$-grading of $\Omega^{0,\bullet}(X$, $L^p\otimes E)$, 
$P_p$ is the orthogonal projection 
from $\cC^\infty(X,E^+_p)$ onto $\ke(D_p)$ for $p$ large enough.  
Thus $P_p(x,x)\in \End(\bE^+)_{x}$ 
and $J_r(Z,Z^\prime)\in \End(\bE^+)_{x_0}$ for $p$ large enough.
\end{rem}
%%%%%%%%%%%%%%%%%%%%%%%%%%%%%%%%%%%%%%%%%%%%%%%%%%%%%%%%%%%%%%%%%%%%%%

Let $\nabla^X\bJ\in T^*X \otimes \End(TX)$ be the covariant derivative 
of $\bJ$ induced by $\nabla^{TX}$. 
We denote by
$\mR= \sum_i Z_i e_i =Z$ the radial vector field on $\R^{2n}$.

For $s\in \cC^{\infty}( T_{x_0}X, \bE_{x_0}) $, set
\begin{align}\label{u0}
&\|s\|_{0,0}^2 = \int_{\R^{2n}} 
|s(Z)|^2_{h^{\Lambda ( T^{*(0,1)}X)\otimes E}_{x_0}} dv_{TX}(Z).
\end{align}

We adopt
the convention that all tensors will be evaluated at the base point
$x_0\in X$, and most of the time, we will omit the subscript $x_0$. 
{}From \eqref{to1.13} and \eqref{3ue63}, let us set
\begin{align}\label{1c31}
\begin{split}
\cL_2^0 :=& \cL-2\om_d = \sum_j (b_j b^+_j 
+ 2 a_j \overline{w}^j\wedge i_{\overline{w}_j}),\\
\boldsymbol{\mO}_1 := &
-\frac{2}{3}\partial_j (R^L (e_k,e_i))_{x_0}Z_j Z_k\nabla_{0,e_i}
 -\frac{1}{3} \partial_i (R^L (e_j,e_i))_{x_0} Z_j\\
&-\pi\sqrt{-1}\left \langle(\nabla_{\mR}^X \bJ)_{x_0}\,e_l, e_m\right \rangle
\,c(e_l)\,c(e_m) \,.
\end{split}\end{align}
Let $P^N$ be the orthogonal projection from  
$(L^2 (\R^{2n}, \bE_{x_0}),\|\cdot\|_{0,0})$
 onto $N=\ke (\cL^0_2)$, and $P^N(Z,Z^\prime)$ its smooth kernel 
with respect to $dv_{TX}(Z^\prime)$. Let $P^ {N^\bot}= \Id -P^N$. 
Since $a_j>0$ we get from \eqref{1c31} that
\begin{equation}\label{abk2.71}
P^N(Z,Z')= \cP(Z,Z') I_{\C\otimes E}.
\end{equation}

By \cite[Theorem 4.6, (4.107), (4.115) and  (4.117)]{DLM04a}
(or proceeding as in  \cite[(1.111)]{MM04a}, 
or by \cite[Theorem 2.2]{MM05a}), we obtain:
%-------------------------------------------------------------------------
\begin{thm}
The following identity holds\,{\rm:}
\begin{equation}\label{abk2.76}
\begin{split}
&\cP^{(1)}_{x_0}=- P^{N^\bot}(\cL^0_2)^{-1} \boldsymbol{\mO}_1 P^N
-  P^N\boldsymbol{\mO}_1(\cL^0_2)^{-1} P^{N^\bot}\,.
\end{split}
\end{equation}
\end{thm}
\begin{rem}\label{rue}
It is interesting to observe the role of $\boldsymbol{\mO}_1$ 
in different geometric situations.
Firstly, if $(X,J,\omega)$ is K{\"a}hler, $\bJ=J$ and $L$, $E$ 
are holomorphic vector bundles, we have $\boldsymbol{\mO}_1=0$.
Secondly, if $(X,J,\omega)$ is symplectic and $E$ is trivial, we do not need
the precise formula of $\boldsymbol{\mO}_1$ for the proof of
 Lemma \ref{toe2.5}, but just the information that
 $\boldsymbol{\mO}_1$ acts as the identity on $E$. 
Thirdly, to compute the coefficient $J_{2,\,x_0}(0,0)$ in \eqref{a0.7} 
as in \cite[Theorem 2.1]{MM05a}, we need certainly 
the precise formula of $\boldsymbol{\mO}_1$ given in \eqref{1c31}.

Finally, for the proof of Theorem \ref{toet4.1}, the precise formulas
for $\boldsymbol{\mO}_1$ or $\cP^{(1)}_{x_0}$ are not needed 
(cf. Remark \ref{toet2.6} and formulas \eqref{toe4.12a}, \eqref{toe4.14}).
\end{rem}
%------------------------------------------------------------------------

%%%%%%%%%%%%%%%%%%%%%%%%%%%%%%%%%%%%%%%%%%%%%%%%%%%%%%%%%%%%%%%%%%%%%%
\section{Berezin-Toeplitz quantization on symplectic manifolds} \label{ts2}

We give a brief summary of this Section.
We begin in Section \ref{toes2} by establishing the asymptotic expansion 
for the kernel of Toeplitz operators.
In Section \ref{toes3}, we show that the asymptotic expansion 
is also a sufficient condition for a family of operators to be Toeplitz.
Finally, in Section \ref{toes4}, we conclude that set of Toeplitz operators forms 
an algebra. 
%%%%%%%%%%%%%%%%%%%%%%%%%%%%%%%%%%%%%%%%%%%%%%%%%%%%%%%%%%%%%%%%%%%%%%%
\subsection{Asymptotic expansion of Toeplitz operators} \label{toes2}
%%%%%%%%%%%%%%%%%%%%%%%%%%%%%%%%%%%%%%%%%%%%%%%%%%%%%%%%%%%%%%%%%%%%%%%
In this Section we define the Toeplitz operators and deduce 
the asymptotic expansion of their Schwartz kernels.

We use the same setting and notations as in Section \ref{s3}.
Let $(X,J, \om)$ be a compact symplectic manifold of real dimension $2n$, 
a Hermitian line bundle $(L,h^L)$ over $X$ endowed with
a Hermitian connection $\nabla^L$ with curvature $R^L=(\nabla^L)^2$
satisfying the prequantization condition \eqref{bk1.1}.
Let $g^{TX}$ be an arbitrary Riemannian metric on $X$ 
compatible with the almost complex structure $J$.
%and $J$ be an almost complex structure which is separately compatible 
%with $g^{TX}$ and $\om$ and $\om(\cdot,J\cdot)>0$. 
We consider a Hermitian vector bundle $(E,h^E)$ 
on $X$ with Hermitian connection $\nabla^E$, 
and the space $\big(L^2(X,E_p),\langle\cdot,\cdot\rangle\big)$
introduced in \eqref{l2}.

\noindent
A section $g\in\cC^\infty(X,\End(E))$ defines a vector bundle morphism
$\Id_{\Lambda (T^{*(0,1)}X)\otimes L^p}\otimes g$ of 
$E_p:=\Lambda (T^{*(0,1)}X)\otimes L^p\otimes E$, which we still denote by $g$.
%--------------------------------------------------------------------------
\begin{defn}\label{toe-def}
A \textbf{\em Toeplitz operator}\index{Toeplitz operator} 
is a sequence $\{T_p\}=\{T_p\}_{p\in\N}$ of linear operators
%---------------------------------------------------------------------------
\begin{equation}\label{toe2.1}
T_{p}:L^2(X,E_p)\longrightarrow L^2(X,E_p)\,,
\end{equation} 
%---------------------------------------------------------------------------- 
with the properties:
\begin{itemize}
\item[(i)] For any $p\in \N$, we have 
%---------------------------------------------------------------------------
\begin{equation}\label{toe2.2}
T_{p}=P_p\,T_p\,P_p\,.
\end{equation} 
%----------------------------------------------------------------------------
\item[(ii)] There exist a sequence $g_l\in\cC^\infty(X,\End(E))$ such that
for all $k\geqslant0$ there exists $C_k>0$ with
%----------------------------------------------------------------------------
\begin{equation}\label{toe2.3}
\Big\|T_p-P_p\Big(\sum_{l=0}^k p^{-l}g_l\Big) P_p\Big\|
\leqslant C_k\,p^{-k-1},
\end{equation}
where $\norm{\cdot}$ denotes the operator norm on the space of 
bounded operators.
\end{itemize}
The full symbol of $\{T_p\}$ is the formal series 
$\sum_{l=0}^\infty \hbar^{l}g_l\in\cC^\infty(X,\End(E))[[\hbar]]$ 
and the \textbf{\em principal symbol}\index{principal symbol} 
of $\{T_p\}$ is $g_0$.  
If each $T_p$ is self-adjoint, $\{T_p\}$ is called self-adjoint.
\end{defn}
%---------------------------------------------------------------------
\noindent
We express \eqref{toe2.3} symbolically by
\begin{equation}\label{atoe2.1}
T_p= P_p\Big(\sum_{l=0}^k p^{-l}g_l\Big) P_p+\mO(p^{-k-1}).
\end{equation}
If \eqref{toe2.3} holds for any $k\in \N$, then we write
\begin{equation}\label{atoe2.2}
T_p= P_p\Big(\sum_{l=0}^\infty p^{-l}g_l\Big) P_p+\mO(p^{-\infty}).
\end{equation}

\noindent
An important particular case is when $g_l=0$ for $l\geqslant1$. 
We set $g_0=f$. We denote then
%--------------------------------------------------------------------
\begin{equation}\label{toe2.4}
T_{f,\,p}:L^2(X,E_p)\longrightarrow L^2(X,E_p)\,,
\quad T_{f,\,p}=P_p\,f\,P_p\,.
\end{equation} 
%----------------------------------------------------------------------------
The Schwartz kernel of $T_{f,\,p}$ is given by
\begin{equation} \label{toe2.5}
T_{f,\,p}(x,x')=\int_XP_p(x,x'')f(x'')P_p(x'',x')\,dv_X(x'')\,.
\end{equation}
Let us remark that if $f\in\cC^{\infty}(X,\End(E))$ is self--adjoint,
i.e. $f(x)=f(x)^*$ for all $x\in X$, then the operators 
$\Id_{\Lambda (T^{*(0,1)}X)\otimes L^p}\otimes f$ 
and $T_{f,\,p}$ are self--adjoint.

The map which associates to a section $f\in \cC^{\infty}(X,\End(E))$ 
the bounded operator $T_{f,\,p}$ on $L^2(X,E_p)$ is called 
the  \textbf{\em Berezin-Toeplitz quantization}.

We examine now the asymptotic expansion of the kernel of the Toeplitz 
operators $T_{f,\,p}$.
The first observation is that outside the diagonal of $X\times X$,
the kernel of $T_{f,\,p}$ has the growth $\cO(p^{-\infty})$. 
%---------------------------------------------------------------------
\begin{lemma} \label{toet2.1}
For every $\varepsilon>0$ and every $l,m\in\N$, 
there exists $C_{l,m,\varepsilon}>0$ such that 
%-------------------------------------------------------------------------
\begin{equation} \label{toe2.6}
|T_{f,\,p}(x,x')|_{\cC^m(X\times X)}\leqslant C_{l,m,\varepsilon}p^{-l}
\end{equation}
%---------------------------------------------------------------------
for all $p\geqslant 1$ and all $(x,x')\in X\times X$ 
with $d(x,x')>\varepsilon$,
where the $\cC^m$-norm is induced by $\nabla^L,\nabla^E$ and $h^L,h^E,g^{TX}$.
\end{lemma} 
%----------------------------------------------------------------------------
\begin{proof}
Due to \eqref{1c3},  \eqref{toe2.6} holds 
if we replace $T_{f,\,p}$ by $P_p$. 
Moreover, from \eqref{aue66}, for any $m\in \N$, 
there exist $C_m>0, M_m>0$ such that 
$|P_p(x,x')|_{\cC^m(X\times X)}<Cp^{M_m}$ for all $(x,x')\in X\times X$.
These two facts and formula \eqref{toe2.5} imply the Lemma. 
\end{proof}

We concentrate next on a neighborhood of the diagonal in order 
to obtain the asymptotic expansion of the kernel $T_{f,\,p}(x,x')$. 

We adhere to the identifications made in Lemma \ref{0l3.0}.
We also identify in the sequel $f\in \cC^\infty(X,\End(E))$ with a family  
$f_{x_0}(Z)\in \End(E_{x_0})$ (with parameter $x_0\in X$)
of functions in $Z$ in normal coordinates near $x_0$.
In general, for functions in the normal coordinates, 
we will add a subscript $x_0$ to indicate the base point $x_0\in X$.

Let $\{\Xi_p\}_{p\in\N}$ be a sequence of linear operators 
$\Xi_p: L^2(X,E_p)\longrightarrow L^2(X,E_p)$
with smooth kernel $\Xi_p(x,y)$ with respect to $dv_X(y)$.

Recall that $\pi : TX\times_{X} TX \to X$ is the
natural projection from the fiberwise product of $TX$ on $X$. 
Under our trivialization, $\Xi_p(x,y)$ induces a smooth section 
$\Xi_{p,\,x_0}(Z,Z^\prime)$ of $\pi^*(\End(\Lambda(T^{*(0,1)}X)\otimes E))$
over $TX\times_{X} TX$ with $Z,Z^\prime\in T_{x_0}X$, $\abs{Z},\abs{Z^{\prime}}<a_X$. %$\abs{Z},\abs{Z^{\prime}}<{\color{red}a_X}4\var$.
Recall also that $\cP_{x_0}=\cP$ was defined in \eqref{toe1.3}. 

Consider the following condition for $\{\Xi_p\}_{p\in\N}$.
\begin{condition}\label{coe2.7}
Let $k\in\N$. There exists a family $\{Q_{r,\,x_0}\}_{0\leqslant r\leqslant k,\,x_0\in X}$ such that
\begin{itemize}
\item[(a)] $Q_{r,\,x_0}\in \End(\Lambda(T^{*(0,1)}X)\otimes E)_{x_0}[Z,Z^{\prime}]$, 
\item[(b)] $\{Q_{r,\,x_0}\}_{r\in\N,\,x_0\in X}$ is smooth with respect to the parameter $x_0\in X$,
%of polynomials in $Z,Z^\prime$ with values in $\End(E_{x_0})$, 
\item[(c)] there exist constants $\var^\prime\in\,]0,a_X[$ and $C_0>0$ with the following property:
for every $l\in \N$, there exist $C_{k,\,l}>0$, $M>0$ 
such that for every $x_0\in X$, $Z,Z^\prime \in T_{x_0}X$, $\abs{Z},\abs{Z^{\prime}}<\var^\prime$ and
$p\in \N$ the following estimate holds (in the sense of \eqref{aue66}):
\begin{equation} \label{toe2.8}
\begin{split}
&\Big|p^{-n} \Xi_{p,\,x_0}(Z,Z^{\prime})
\kappa_{x_0}^{1/2}(Z)\kappa_{x_0}^{1/2}(Z^{\prime})
-\sum^k_{r=0}(Q_{r,x_0}\cP_{x_0})(\sqrt{p}Z,\sqrt{p}Z^{\prime})
p^{-\frac{r}{2}}\Big|_{\cC^l(X)}\\
\leqslant & \,C_{k,\,l}\,p^{-\frac{k+1}{2}}
(1+\sqrt{p}\,|Z|+\sqrt{p}\,|Z^{\prime}|)^M
\exp(-\sqrt{C_0 p}\,|Z-Z^{\prime}|)+\cO(p^{-\infty})\,.
\end{split}
\end{equation}
\end{itemize}
\end{condition}

\begin{notation}\label{noe2.7}
Assume that $\{\Xi_p\}_{p\in\N}$ is subject to the Condition \ref{coe2.7}.
Then we write 
\begin{equation} \label{toe2.7}
p^{-n} \Xi_{p,x_0}(Z,Z^\prime)\cong \sum_{r=0}^k 
(Q_{r,\,x_0} \cP_{x_0})(\sqrt{p}Z,\sqrt{p}Z^{\prime})p^{-\frac{r}{2}}
+\mO(p^{-\frac{k+1}{2}})\,.
\end{equation}
\end{notation}
The family $\{J_{r,\,x_0}\}_{r\in\N,\,x_0\in X}$ of polynomials 
$J_{r,\,x_0}(Z,Z')\in \End(\Lambda (T^{*(0,1)}X)\otimes E)_{x_0}$
was defined in Theorem \ref{tue17}. Moreover, $J_{r,\,x_0}(Z,Z')$ 
have the same parity as $r$, $\deg J_{r,\,x_0}\leqslant 3r$, and
\begin{equation} \label{toe2.10}
J_{0,\,x_0}= I_{\C\otimes E}.
\end{equation}
%--------------------------------------------------------------------------
\begin{lemma} \label{toet2.2} For any $k\in \N$, $\varepsilon\in]0, a_X/4[$\,,
 $Z,Z^\prime \in T_{x_0}X$, $\abs{Z},\abs{Z^{\prime}}<2 \var$,
we have
\begin{equation} \label{toe2.9}
p^{-n} P_{p,\,x_0}(Z,Z^\prime)\cong \sum_{r=0}^k 
(J_{r,\,x_0} \cP_{x_0})(\sqrt{p}Z,\sqrt{p}Z^{\prime})p^{-\frac{r}{2}}
+\mO(p^{-\frac{k+1}{2}})\,,
\end{equation}
 in the sense of Notation \ref{noe2.7}.
\end{lemma}
%--------------------------------------------------------------------------
\begin{proof}
Theorem \ref{tue17} shows that
for any $k,m'\in \N$, there exist $M\in \N, C>0$ such that 
%------------------------------------------------------------------------ 
\begin{multline}\label{toe2.11}
\Big| p^{-n}P_{p,\,x_0}(Z,Z^{\prime})
\kappa ^{\frac{1}{2}}_{x_0}(Z)\kappa^{\frac{1}{2}}_{x_0}(Z^{\prime })
 - \sum_{r=0}^k  (J_{r,\,x_0} \cP_{x_0})  (\sqrt{p} Z,\sqrt{p} Z^{\prime})
p^{-\frac{r}{2}}\Big|_{\cC ^{m'}(X)}\\
\leqslant C p^{-(k+1)/2}  (1+\sqrt{p}\,|Z|+\sqrt{p}\,|Z^{\prime}|)^M
\exp (- \sqrt{C''\mu_0}\,\sqrt{p}\, |Z-Z^{\prime}|)+\cO(p^{-\infty}),
\end{multline}
%------------------------------------------------------------------------ 
for $Z,Z^{\prime}\in T_{x_0}X$, $|Z|, |Z^{\prime }|\leqslant 2 \var$.
%------------------------------------------------------------------------ 
Hence \eqref{toe2.11} immediately entails \eqref{toe2.9}.
\end{proof}
%--------------------------------------------------------------------------
\begin{lemma} \label{toet2.3}
Let $f\in\cC^\infty(X,\End(E))$.
There exists a family $\{Q_{r,\,x_0}(f)\}_{r\in\N,\,x_0\in X}$ such that
\begin{itemize}
\item[(a)] $Q_{r,\,x_0}(f)\in\End(\Lambda(T^{*(0,1)}X)\otimes E)_{x_0}[Z,Z^{\prime}]$
are polynomials with the same parity as $r$, 
\item[(b)] $\{Q_{r,\,x_0}(f)\}_{r\in\N,\,x_0\in X}$ is smooth with respect to $x_0\in X$,
\item[(c)] for every $k\in \N$, $\varepsilon\in]0, a_X/4[$\,, $x_0\in X$,
 $Z,Z^\prime \in T_{x_0}X$, $\abs{Z},\abs{Z^{\prime}}<\var/2$ we have 
 %--------------------------------------------------------------------------
\begin{equation} \label{toe2.13}
p^{-n}T_{f,\,p,\,x_0}(Z,Z^{\prime})
\cong \sum^k_{r=0}(Q_{r,\,x_0}(f)\cP_{x_0})(\sqrt{p}Z,\sqrt{p}Z^{\prime})
p^{-\frac{r}{2}} + \mO(p^{-\frac{k+1}{2}})\,,
\end{equation}
%--------------------------------------------------------------------------
in the sense of Notation \ref{noe2.7}.
\end{itemize}
$Q_{r,\,x_0}(f)$ are expressed by 
\begin{equation} \label{toe2.14}
Q_{r,\,x_0}(f) = \sum_{r_1+r_2+|\alpha|=r}
  \cK\Big[J_{r_1,\,x_0}\;,\; 
\frac{\partial ^\alpha f_{\,x_0}}{\partial Z^\alpha}(0) 
\frac{Z^\alpha}{\alpha !} J_{r_2,\,x_0}\Big]\,.
\end{equation}
%--------------------------------------------------------------------------
Especially,
\begin{align} \label{toe2.15}
Q_{0,\,x_0}(f)= f(x_0)\,I_{\C\otimes E} .
\end{align}
\end{lemma}
%------------------------------------------------------------------
\noindent
We have used here the notations \eqref{toe1.6} and \eqref{2c3}.
%------------------------------------------------------------------
\begin{proof} {}From \eqref{toe2.5} and \eqref{toe2.6}, 
we know that for 
$\abs{Z},\abs{Z^{\prime}}<\var/2$, $T_{f,\,p,\,x_0}(Z,Z^\prime)$ 
is determined up to terms of order $\cO(p^{-\infty})$
 by the behavior of $f$ in $B^X(x_0, \varepsilon)$.
Let $\rho: \R\to [0,1]$ be a smooth even function such that
%------------------------------------------------------------------------ 
\begin{equation}\label{alm4.19}
\rho (v)=1  \  \  {\rm if} \  \  |v|<2;
\quad \rho (v)=0 \   \   {\rm if} \  |v|>4.
\end{equation}
%------------------------------------------------------------------------ 
For $\abs{Z},\abs{Z^{\prime}}<\var/2$, we get 
%------------------------------------------------------------------------ 
\begin{multline}\label{toe2.17}
T_{f,\,p,\,x_0}(Z,Z^{\prime})=\int_{{ T_{x_0}X}}
%\int_{\stackrel{Z^{\prime\prime}
%\in T_{x_0}X}{\abs{Z^{\prime\prime}}\leqslant\var}}
P_{p,x_0}(Z,Z^{\prime\prime})
\rho(2|Z^{\prime\prime}|/\var)
f_{x_0}(Z^{\prime\prime})P_{p,x_0}(Z^{\prime\prime},Z^{\prime}) \\
\times \kappa_{x_0}(Z^{\prime\prime})
\,dv_{TX}(Z^{\prime\prime})+ \cO(p^{-\infty}).
\end{multline}  
%------------------------------------------------------------------------ 
We consider the Taylor expansion of $f_{x_0}$:
%------------------------------------------------------------------------ 
\begin{multline}\label{toe2.18}
f_{x_0}(Z)=\sum_{|\alpha|\leqslant k}
\frac{\partial^\alpha f_{x_0}}{\partial Z^\alpha}(0)
\frac{ Z^\alpha}{\alpha!}+\cO(|Z|^{k+1})\\
=\sum_{|\alpha|\leqslant k} p^{-|\alpha|/2}
\frac{\partial^\alpha f_{x_0}}{\partial Z^\alpha}(0)
\frac{ (\sqrt{p}Z)^\alpha}{\alpha!}
+p^{-\frac{k+1}{2}}\cO(|\sqrt{p}Z|^{k+1}).
\end{multline}
%------------------------------------------------------------------------ 
We multiply now the expansions given in \eqref{toe2.18} and \eqref{toe2.11} 
and obtain the expansion of 
%------------------------------------------------------------------------ 
\begin{equation*}
\kappa_{x_0}^{1/2}(Z)P_{p,\,x_0}(Z,Z^{\prime\prime})
(\kappa_{x_0}f_{x_0})(Z^{\prime\prime})
P_{p,\,x_0}(Z^{\prime\prime},Z^{\prime})\kappa_{x_0}^{1/2}(Z^{\prime})
\end{equation*}
%------------------------------------------------------------------------ 
which we substitute in \eqref{toe2.17}. We integrate then on 
$T_{x_0}X$ by using the change of variable $\sqrt{p}\,Z^{\prime\prime}=W$
and conclude \eqref{toe2.13} and \eqref{toe2.14} 
by using formulas \eqref{toe1.6} and \eqref{toe1.15}.

\noindent
{}From \eqref{toe2.10} and \eqref{toe2.14}, we get 
\begin{align} \label{toe2.19}
Q_{0,\,x_0}(f)= \cK[1, f_{x_0}(0)]I_{\C\otimes E} = f_{x_0}(0)I_{\C\otimes E} = f(x_0)I_{\C\otimes E}.
\end{align}
The proof of Lemma \ref{toet2.3} is complete.
\end{proof}
%------------------------------------------------------------------------ 
As an example, we compute $Q_{1,\,x_0}(f)$.
%-----------------------------------------------------------
\begin{lemma}\label{toet2.5} 
$Q_{1,\,x_0}(f)$ appearing in \eqref{toe2.13} is given by
\begin{equation} \label{toe2.20}
Q_{1,\,x_0}(f)= f(x_0) J_{1,\,x_0}  + \cK\Big[J_{0,\,x_0}, 
\frac{\partial f_{x_0}}{\partial Z_j}(0) Z_j J_{0,\,x_0}\Big].
\end{equation}
\end{lemma}
%------------------------------------------------------------------
\begin{proof} At first, by taking $f=1$ in \eqref{toe2.14}, we get 
\begin{equation} \label{toe2.21}
J_{1,\,x_0}= \cK[J_{0,\,x_0}, J_{1,\,x_0}]+ \cK[J_{1,\,x_0}, J_{0,\,x_0}].
\end{equation}
The operator $\boldsymbol{\mO}_1$ defined in \eqref{1c31} 
(considered as a differential operator with coefficients 
in $\End(\Lambda(T^{*(1,0)}X)\otimes E)_{x_0}$) acts as the identity on 
the $E$-component.
%and preserves the $\Z_2$-grading on $\Lambda(T^{*(1,0)}X)$. 
Thus from \eqref{abk2.76} and \eqref{toe1.6}, we obtain
\begin{equation} \label{toe2.22}
 \cK[J_{1,\,x_0}, f(x_0)J_{0,\,x_0}]= f(x_0) \cK[J_{1,\,x_0}, J_{0,\,x_0}].
\end{equation}
{}From \eqref{toe2.14}, \eqref{toe2.21} and \eqref{toe2.22}, 
we get \eqref{toe2.20}.
\end{proof}

\begin{rem}\label{toet2.6} If $f\in \cC^\infty(X)$ 
we get \eqref{toe2.22}, thus also \eqref{toe2.20},
without using the precise formulas of $\boldsymbol{\mO}_1$ or $J_1$.
\end{rem}

%%%%%%%%%%%%%%%%%%%%%%%%%%%%%%%%%%%%%%%%%%%%%%%%%%%%%%%%%%%%%%%%%
\subsection{A criterion for Toeplitz operators}\label{toes3}
%%%%%%%%%%%%%%%%%%%%%%%%%%%%%%%%%%%%%%%%%%%%%%%%%%%%%%%%%%%%%%%%%
We will prove next a useful criterion which ensures that
a given family is a Toeplitz operator. 
%Let $X$ be a compact manifold and $(L,h^L)$ and $(E,h^E)$ 
%two arbitrary holomorphic Hermitian vector bundles on $X$.
%-----------------------------------------------------------------------
\begin{thm}\label{toet3.1}
Let $\{T_p:L^2(X,E_p)\longrightarrow L^2(X,E_p)\}$ 
be a family of bounded linear operators which satisfies 
the following three conditions:
\begin{itemize}
\item[(i)] For any $p\in \N$,  $P_p\,T_p\,P_p=T_p$\,.
\item[(ii)] For any $\varepsilon_0>0$ and any $l\in\N$, 
there exists $C_{l,\varepsilon_0}>0$ such that 
for all $p\geqslant 1$ and all $(x,x')\in X\times X$ 
with $d(x,x')>\varepsilon_0$,
%--------------------------------------------------------------------
\begin{equation} \label{toe3.1}
|T_{p}(x,x')|\leqslant C_{l,\varepsilon_0}p^{-l}.
\end{equation}
%-----------------------------------------------------------------------
\item[(iii)] There exists a family of polynomials 
$\{\mQ_{r,\,x_0}\in\End(\Lambda(T^{*(0,1)}X)\otimes E)_{x_0}
[Z,Z^{\prime}]\}_{x_0\in X}$ 
such that: 
\begin{itemize}
\item[(a)] each $\mQ_{r,\,x_0}$ has the same parity as $r$, 
\item[(b)] the family is smooth in $x_0\in X$ and 
\item[(c)] there exists $0<\var^\prime<a_X/4$ such that for every  $x_0\in X$,
 every $Z,Z^\prime \in T_{x_0}X$ with  $\abs{Z},\abs{Z^{\prime}}<\var^\prime$ 
and every $k\in\N$ we have 
\begin{equation} \label{toe3.2}
p^{-n}T_{p,\,x_0}(Z,Z^{\prime})\cong 
\sum^k_{r=0}(\mQ_{r,\,x_0}\cP_{x_0})
(\sqrt{p}Z,\sqrt{p}Z^{\prime})p^{-\frac{r}{2}} + \mO(p^{-\frac{k+1}{2}}),
\end{equation}
in the sense of \eqref{toe2.7} and \eqref{toe2.8}.
\end{itemize}
\end{itemize}
Then $\{T_p\}$ is a Toeplitz operator. 
\end{thm}
%-----------------------------------------------------------------------------
\begin{rem} \label{toet3.0} By Lemmas \ref{toet2.1} and \ref{toet2.3}, 
and by \eqref{toe2.2}, \eqref{toe2.3} and the Sobolev inequality
(cf. \cite[(4.14)]{DLM04a}), 
it follows that every Toeplitz operator in the sense
 of Definition \ref{toe-def} verifies the conditions (i), (ii), (iii) 
of Theorem \ref{toet3.1}.
\end{rem}

We start the proof of Theorem \ref{toet3.1}. 
 Let $T_p^*$ be the adjoint of $T_p$. By writing 
\begin{equation}\label{toe3.3}
T_p=\frac{1}{2}(T_p+T_p^*)+\sqrt{-1}\frac{1}{2\sqrt{-1}}(T_p-T_p^*),
\end{equation} 
we may and will assume from now on that $T_p$ is self-adjoint.

We will define inductively the sequence
$(g_l)_{l\geqslant0}$, $g_l\in \cC^\infty(X, \End(E))$ such that 
\begin{equation}\label{toe3.4}
T_{p} = \sum_{l=0}^m P_{p}\, g_l \, p^{-l}\, P_{p} +\mO(p^{-m-1})\,,\quad\text{for every $m\geqslant 0$}\,.
\end{equation} 
Moreover, we can make these $g_l$'s to be self-adjoint.

%\noindent
Let us start with the case $m=0$ of \eqref{toe3.4}. For an arbitrary but fixed $x_0\in X$, we set %according to \eqref{toe3.5m},
\begin{equation}\label{toe3.5}
g_0(x_0)=\mQ_{0,\,x_0}(0,0)|_{\,\C\otimes E}\,\in\End(E_{x_0})\,.
\end{equation}
We will show that
\begin{equation}\label{toe3.60}
p^{-n} (T_{p} -  T_{g_0,\,p})_{x_0}(Z,Z^\prime)\cong \mO(p^{-1})\,,
\end{equation} 
which implies the case $m=0$ of \eqref{toe3.4}, namely,
\begin{equation}\label{toe3.6}
T_{p}=P_p\,g_0\,P_p + \mO(p^{-1}).
\end{equation} 
The proof of \eqref{toe3.60}-\eqref{toe3.6} will be done in Propositions \ref{toet3.2} and \ref{toet3.8}.
%-----------------------------------------------------------------------------
\begin{prop}\label{toet3.2}
In the conditions of Theorem \ref{toet3.1} we have 
$\mQ_{0,\,x_0}(Z,Z^{\prime})=\mQ_{0,\,x_0}(0,0)\in 
\End(E_{x_0})\circ I_{\C\otimes E}$
for all $x_0\in X$ and all $Z,Z^{\prime}\in T_{x_0}X$. 
\end{prop}
%-----------------------------------------------------------------------------
\begin{proof}
The proof is divided in the series of Lemmas \ref{toet3.3}, \ref{toet3.4}, \ref{toet3.5}, \ref{toet3.6} and \ref{toet3.7}.
Our first observation is as follows.
\begin{lemma}\label{toet3.3}
$\mQ_{0,\,x_0}\in\End(E_{x_0}) \circ I_{\C\otimes E} [Z,Z^{\prime}]$,
and $\mQ_{0,\,x_0}$ is a polynomial in $z,\ov{z}^\prime$.
\end{lemma}
%-----------------------------------------------------------------------------
\begin{proof}

Indeed, by \eqref{toe3.2}
\begin{equation} \label{toe3.10}
p^{-n}T_{p,\,x_0}(Z,Z^{\prime})
\cong (\mQ_{0,\,x_0}\cP_{x_0})(\sqrt{p}Z,\sqrt{p}Z^{\prime})
+ \mO(p^{-1/2}).
\end{equation}
%-----------------------------------------------------------------------------

Moreover, by \eqref{toe2.10} and \eqref{toe2.9}, we have
%-----------------------------------------------------------------------------
\begin{multline} \label{bsy1.19}
p^{-n}(P_p\,T_{p}\,P_p)_{x_0}(Z,Z^{\prime})
\cong ((\cP J_0)  \circ(\mQ_{0}\cP)\circ(\cP J_0))_{x_0}
(\sqrt{p}Z,\sqrt{p}Z^{\prime}) + \mO(p^{-1/2}).
\end{multline}
%-----------------------------------------------------------------------------
Since $P_p\,T_{p}\,P_p=T_p$,
we deduce from \eqref{toe3.10}, \eqref{toe2.10} and \eqref{bsy1.19} that 
\begin{equation} \label{bsy1.20}
\mQ_{0,\,x_0}\cP_{x_0}=I_{\C\otimes E} \cP_{x_0}\circ(\mQ_{0,\,x_0}
\cP_{x_0})\circ\cP_{x_0}I_{\C\otimes E},
\end{equation}
hence $\mQ_{0,\,x_0}\in\End(E_{x_0}) \circ I_{\C\otimes E}  
[z,\ov{z}^{\,\prime}]$ 
by \eqref{toe1.5} and \eqref{toe1.8}.
\end{proof}
%-----------------------------------------------------------------------------

\noindent
For simplicity we denote in the rest of the proof 
$F_{x}=\mQ_{0,\,x}|_{\C\otimes E}\in \End(E_x)$. 
Let $F_x=\sum_{i\geqslant0}F^{(i)}_x$ be the decomposition of $F_x$
in homogeneous polynomials $F^{(i)}_x$ of degree $i$. 
We will show that $F^{(i)}_x$ vanish identically for $i>0$, that is,
%-----------------------------------------------------------------------------
\begin{equation}\label{6.38}
F^{(i)}_x(z,\ov{z}^{\prime})=0 \quad \text{for all $i>0$ 
and $z,z^{\prime}\in\C^n$}\,.
\end{equation}
%-----------------------------------------------------------------------------
The first step is to prove
%-----------------------------------------------------------------------------
\begin{equation}\label{6.381}
F^{(i)}_x(0,\ov{z}^{\prime})=0 \quad \text{for all $i>0$ 
and all $z^{\prime}\in\C^n$}\,.
\end{equation}
%-----------------------------------------------------------------------------
Let us remark that since $T_{p}$ are self-adjoint we have
%-----------------------------------------------------------------------------
\begin{equation}\label{c6.31}
F^{(i)}_{x}(z, \ov{z}^{\prime})
= (F^{(i)}_{x}(z^{\prime}, \ov{z}))^*.
\end{equation}
Consider $\var^\prime>0$ as in hypothesis (iii) (c) of Theorem \ref{toet3.1}. 
For $Z^{\prime}\in\R^{2n}\simeq T_{x}X$ with $\abs{Z^{\prime}}<\var^\prime$ 
and $y= \exp_{x}^{X}(Z^{\prime})$,
%as explained above \eqref{toe2.7}, 
set
%-----------------------------------------------------------------------------
\begin{equation}\label{b6.32}
\begin{split}
&F^{(i)}(x,y)= F^{(i)}_{x}(0,\ov{z}^{\prime})
\in \End(E_{x}), \\
&\wi{F}^{(i)}(x,y)=(F^{(i)}(y,x))^* 
\in  \End(E_{y}).
\end{split} 
\end{equation}
%-----------------------------------------------------------------------------
$F^{(i)}$ and $\wi{F}^{(i)}$ define smooth
sections on a neighborhood of the diagonal of $X\times X$.
%under our trivializations of $E_G$ and $L_G$ by using parallel transport
%along the path $[0,1]\ni u\to uZ^{\prime 0}$. 
Clearly, the $\wi{F}^{(i)}(x,y)$'s need not be polynomials in $z$ and
$\ov{z}^\prime$.

Since we wish to define global operators induced by these kernels, we use a
cut-off function in the neighborhood of the diagonal. Pick a smooth function
$\eta\in\cC^\infty(\R)$, such that $\eta(u)=1$ 
for $\abs{u}\leqslant\var^\prime/2$ and 
$\eta(u)=0$ for $\abs{u}\geqslant\var^\prime$.  
%Let $\varphi:X\times X\to\R$, $\varphi=\eta(d(\cdot,\cdot))$.

We denote by $F^{(i)}P_{p}$ and
$P_{p}\wi{F}^{(i)}$ the operators defined by the kernels
$$\eta(d(x,y))F^{(i)}(x,y)P_{p}(x,y)\quad \mbox{ and }\quad
\eta(d(x,y))P_{p}(x,y)\wi{F}^{(i)}(x,y)$$
with respect to $dv_X(y)$. Set
%-----------------------------------------------------------------------------
\begin{equation}\label{b6.33}
\cT_{p}= T_{p}
-\sum_{i\leqslant\deg F_{x}} (F^{(i)}P_{p})\, p^{i/2}.
\end{equation}
%-----------------------------------------------------------------------------
The operators $\cT_p$ extend naturally to bounded
operators on $L^2(X,E_p)$.

{}From \eqref{toe3.2} and \eqref{b6.33} we deduce that for 
all $k\geqslant1$ and $\abs{Z^{\prime}}\leqslant\var^{\prime}$,
we have the following expansion in the normal coordinates
around  $x_0\in X$ (which has to be understood ,
in the sense of \eqref{toe2.7}):
\begin{align}\label{b6.34}
 p^{-n} \cT_{p,x_0} (0,Z^{\prime})
 \cong  \sum_{r=1}^k  (R_{r,x_0} \cP_{x_0}) (0,\sqrt{p} Z^{\prime}) 
p^{-r/2} +\mO(p^{-(k+1)/2}),
%\leqslant C p^{-(k+1)/2}  (1+\sqrt{p} \,|Z^{\prime}|)^M
%\exp (- \sqrt{C''\,p } \,|Z^{\prime}|)+\cO(p^{-\infty}).
\end{align}
for some polynomials $R_{r,x_0}$ of the same parity as $r$.
For simplicity let us define similarly to \eqref{b6.32} the kernel
\begin{equation}\label{b6.341}
R_{r,\,p}(x,y) = p^{n} (R_{r,x} \cP_x)(0,\sqrt{p} Z^{\prime})
\kappa^{-1/2}_{x}(Z^\prime)\eta(d(x,y)) \,,
\end{equation}
where $y= \exp_{x}^{X}(Z^{\prime})$, and denote by $R_{r,\,p}$
the operator defined by this kernel.
%------------------------------------------------------------------------
\begin{lemma}\label{toet3.4}
There exists $C>0$ such that for every $p>p_0$ and
$s\in L^2(X, E_p)$ we have
\begin{gather}
\label{b6.342}
\norm{\cT_p\,s}_{L^2}\leqslant C p^{-1/2}\norm{s}_{L^2}\,,\\
\|\cT_{p}^* s  \|_{L^2} \leqslant C p^{-1/2} \|s\|_{L^2}\,. \label{b6.39}
\end{gather}
\end{lemma}
%--------------------------------------------------------------------
\begin{proof}
In order to use \eqref{b6.34} we write
\begin{equation}
\norm{\cT_p\,s}_{L^2}\leqslant
\Big\|(\cT_p-\sum_{r=1}^k p^{-r/2}R_{r,\,p})\,s\Big\|_{L^2}+
\Big\|\sum_{r=1}^k p^{-r/2}R_{r,\,p}\,s\Big\|_{L^2}.
\end{equation}
By the Cauchy-Schwarz inequality we have
\begin{multline}\label{b6.381}
\big\|\big(\cT_p-\sum_{r=1}^k p^{-r/2}R_{r,\,p}\big)\,s \|_{L^2} ^2
\leqslant \int_{X} \Big(\int_{X}
\Big|\big(\cT_p-\sum_{r=1}^k p^{-r/2}R_{r,\,p}\big)(x,y)\Big| dv_{X}(y)\Big)\\
\times \Big(\int_{X}
\Big|\big(\cT_p-\sum_{r=1}^k p^{-r/2}R_{r,\,p}\big)(x,y)\Big| |s(y)|^2
dv_{X}(y)\Big)dv_{X}(x).
%&\leq C \int_{x_0,y_0\in X_G} p^{n-n_0} (1+\sqrt{p} \, d(x_0,y_0) )^M
%e^{-c\sqrt{p} \, d(x_0,y_0)} |s|^2(y_0)
%dv_{X_G}(y_0)dv_{X_G}(x_0) + \cO(p^{-\infty})\|s\|_{L^2}^2\\
%\leqslant C\|s\|_{L^2}^2.
\end{multline}
We split then the inner integrals into integrals over $B^X(x,\var^\prime)$ 
and $X\smallsetminus B^X(x,\var^\prime)$ and use the fact that the kernel 
of $\cT_p-\sum_{r=1}^k p^{-r/2}R_{r,\,p}$
has the growth $\cO(p^{-\infty})$ outside the diagonal. Indeed, this follows 
by \eqref{toe3.1}, the definition
of the operators $F^{(i)}P_{p}$ in \eqref{b6.33} 
(using the cut-off function $\eta$), 
and the definition \eqref{b6.341} of $R_{r,\,p}$ (which involves $\cP$).
We get for example, uniformly in $x\in X$,
\begin{multline}\label{b6.382}
\int_{X}\Big|\big(\cT_p-\sum_{r=1}^k p^{-r/2}R_{r,\,p}\big)(x,y)\Big||s(y)|^2 
dv_{X}(y)\\
=\int_{B^X(x,\var^\prime)}\Big|(\cT_p
-\sum_{r=1}^k p^{-r/2}R_{r,\,p}\big)(x,y)\Big||s(y)|^2 dv_{X}(y)\\+
\cO(p^{-\infty})\int_{X\smallsetminus B^X(x,\var^\prime)} |s(y)|^2 dv_{X}(y).
\end{multline} 
By \eqref{toe2.8} and \eqref{b6.34} applied for $k$ sufficiently large, 
which we fix from now on, we obtain
\begin{multline}\label{b6.383}
\int_{B^X(x,\var^\prime)}\Big|\big(\cT_p
-\sum_{r=1}^k p^{-r/2}R_{r,\,p}\big)(x,y)\Big||s(y)|^2 dv_{X}(y)\\
=\cO(p^{-1})\int_{B^X(x,\var^\prime)} |s(y)|^2dv_{X}(y).
\end{multline}
In the same vein we obtain
\begin{equation}\label{b6.384}
\int_{X}\Big|\big(\cT_p-\sum_{r=1}^k p^{-r/2}R_{r,\,p}\big)(x,y)\Big|dv_{X}(y)=
\cO(p^{-1})+\cO(p^{-\infty}).
\end{equation} 
Combining \eqref{b6.381}-\eqref{b6.384} we infer
\begin{equation}\label{b6.385}
\Big\|\big(\cT_p-\sum_{r=1}^k p^{-r/2}R_{r,\,p}\big)\,s \Big\|_{L^2} 
\leqslant
C\,\,p^{-1}\norm{s}_{L^2}\,,\quad s\in L^2(X, E_p)\,.
\end{equation}
A similar proof as for \eqref{b6.385} delivers for $s\in L^2(X, E_p)$
\begin{equation}\label{b6.386} 
\big\|R_{r,\,p}\,s\big\|_{L^2}\leqslant C \norm{s}_{L^2}\,,
\end{equation} 
which implies
\begin{equation}\label{b6.387} 
\Big\|\sum_{r=1}^k p^{-r/2}R_{r,\,p}\,s\Big\|_{L^2}
\leqslant C\,p^{-1/2}\norm{s}_{L^2}\,,
\quad \mbox{for } s\in L^2(X, E_p)\,,
\end{equation} 
for some constant $C>0$. Relations \eqref{b6.385} and \eqref{b6.387} entail
\eqref{b6.342}, which is equivalent to \eqref{b6.39}, 
by taking the adjoint.
%by an easy functional analysis argument.
\end{proof}
%---------------------------------------------------------------------------
Let us consider the Taylor development of $\wi{F}^{(i)}$ 
in normal coordinates around $x$ with $y=\exp_{x}^{X}(Z^{\prime})$:
\begin{equation}\label{taylor}
\wi{F}^{(i)}(x,y)
=\sum_{|\alpha|\leqslant k}
\frac{\partial^\alpha\wi{F}^{(i)}}
{\partial Z^{\prime \alpha}}(x,0)
\frac{(\sqrt{p}Z^{\prime}) ^\alpha}{\alpha !}p^{-|\alpha|/2} 
+ \cO(|Z^\prime|^{k+1}).
\end{equation}
The next step in the proof of Proposition \ref{toet3.2} is the following.
%------------------------------------------------------------------------
\begin{lemma}\label{toet3.5}
For every $j>0$ we have
\begin{equation}\label{a6.41}
\frac{\partial ^\alpha \wi{F}^{(i)}}
{\partial Z^{\prime \alpha}}(x,0)=0\,, \quad \text{for} \,\,
  i-|\alpha|\geqslant j>0.
\end{equation}
\end{lemma}
%-------------------------------------------------------------------------
\begin{proof}
The definition \eqref{b6.33} of $\cT_p$ shows that
\begin{equation}\label{b6.40}
\cT_p^* = T_{p} - \sum_{i\leqslant\deg F_{x}} p^{i/2}(P_{p}\wi{F}^{(i)})\,.
\end{equation} 
Let us develop the sum in the right hand-side. 
Combining the Taylor development \eqref{taylor} 
with the expansion \eqref{toe2.9} of the Bergman kernel we obtain:
\begin{multline}\label{a6.39}
p^{-n} \sum_{i} (P_{p}\,\wi{F}^{(i)})_{x_0}(0,Z^{\prime})p^{i/2}\\
\cong \sum_i \sum_{|\alpha|,\,r\leqslant k}\left(J_{r,\,x_0} 
\cP_{x_0} \right)(0, \sqrt{p} Z^{\prime})
\frac{\partial ^\alpha \wi{F}^{(i)}}{\partial Z^{\prime\alpha}}(x_0,0)
\frac{(\sqrt{p}Z^{\prime})^\alpha}{\alpha !}
p^{(i -|\alpha|-r)/2}\\
 + \mO(p^{(\deg F-k-1)/2}) ,
\end{multline}
where $k\geqslant\deg F_{x}+1$.
Having in mind \eqref{b6.39}, 
this is only possible if for every $j>0$ the coefficients of $p^{j/2}$
in the right hand side of \eqref{a6.39} vanish.
Thus we have for every $j>0$:
\begin{equation}\label{a6.40}
\sum_{l=j}^{\deg F_{x}} \sum_{|\alpha|+r=l-j}
J_{r,\,x_0} (0, \sqrt{p} Z^{\prime})
\frac{\partial ^\alpha \wi{F}^{(l)}}
{\partial Z^{\prime\alpha}}(x_0,0)
\frac{(\sqrt{p}Z^{\prime})^\alpha}{\alpha !}=0.
\end{equation}
{}From \eqref{a6.40}, we will prove by recurrence that 
for any $j>0$, \eqref{a6.41} holds.
As the first step of the recurrence let us take $j=\deg F_{x}$ in \eqref{a6.40}. 
Since $J_{0,\,x_0}=I_{\C\otimes E}$ (see \eqref{toe2.10}),
we get immediately $\wi{F}^{(\deg F_{x})}(x_0,0)=0$. 
Hence \eqref{a6.41} holds for  $j=\deg F_{x}$.

Assume that \eqref{a6.41} holds for $j>j_0>0$.
Then for $j=j_0$, the coefficient with $r>0$ in \eqref{a6.40} is zero.
Since $J_{0,\,x_0}=I_{\C\otimes E}$, \eqref{a6.40} reads 
\begin{equation}\label{a6.42}
 \sum_{\alpha}
\frac{\partial ^\alpha \wi{F}^{(j_0+|\alpha|)}}
{\partial Z^{\prime\alpha}}(x_0,0)
\frac{(\sqrt{p}Z^{\prime})^\alpha}{\alpha !}=0\,,
\end{equation}
which entails \eqref{a6.41} for $j=j_0$.
The proof of \eqref{a6.41} is complete.
\end{proof}
%------------------------------------------------------------------------------
\begin{lemma}\label{toet3.6} For $i>0$, we have 
\begin{equation}\label{d6.38}
\frac{\partial ^\alpha F^{(i)}_x}
{\partial \ov{z}^{\prime\alpha}}(0,0)=0\,,\quad |\alpha|\leqslant i\,.
\end{equation}
Therefore $F^{(i)}_x(0,\ov{z}^{\prime})=0$ for all $i>0$ 
and $z^{\prime}\in\C^n$ i.e. \eqref{6.381} holds true. Moreover,
\begin{equation}\label{c6.38}
F^{(i)}_x(z,0)=0 \quad \text{for all $i>0$ and all $z\in\C^n$}\,.
\end{equation}
\end{lemma}
%------------------------------------------------------------------------------
\begin{proof}
Let us start with some preliminary observations.

In view of \eqref{b6.39}, \eqref{a6.41} and \eqref{a6.39},
a comparison the coefficient of $p^0$ in
\eqref{toe3.10} and \eqref{b6.40} yields
\begin{equation}\label{a6.43}
\wi{F}^{(i)}(x,Z^{\prime})
=F^{(i)}_{x} (0, \ov{z}^{\prime}) + \cO(|Z^{\prime}|^{i+1}).
\end{equation}
Using the definition \eqref{b6.32} of $\wi{F}^{(i)}(x,Z')$, 
and taking the adjoint of \eqref{a6.43} we get
\begin{equation}\label{b6.41}
F^{(i)}(Z^{\prime}, x)
= (F^{(i)}_{x}(0, \ov{z}^{\prime}))^*
+ \cO(|Z^{\prime}|^{i+1}),
\end{equation}
which implies 
\begin{equation}\label{b6.43}
\frac{\partial ^\alpha}{\partial z^{\alpha}} F^{(i)}(\cdot, x)|_{x}
= \Big(\Big(\frac{\partial^\alpha}{\partial \ov{z}^{\prime\alpha}}
F^{(i)}_{x}\Big)(0, \ov{z}^{\prime})\Big)^*\,,
\quad\text{for $|\alpha|\leqslant i$}\,,
\end{equation}
so in order to prove the Lemma it suffices to show that 
\begin{equation}\label{b6.432}
\frac{\partial ^\alpha}{\partial z^{\alpha}} F^{(i)}(\cdot, x)|_{x}
=0\,,\quad\text{for $|\alpha|\leqslant i$}\,.
\end{equation}
We prove this by induction over $|\alpha|$. For $|\alpha|=0$, it is obvious 
that $F^{(i)}(0, x)=0$, since $F^{(i)}(\cdot, x)$ is a homogeneous polynomial of degree $i>0$.
For the induction step let $j_{X}: X\to X\times X$ be the diagonal injection.
By Lemma \ref{toet3.3} and the definition \eqref{b6.32} of $F^{(i)}(x,y)$,
\begin{equation}\label{b6.410}
\frac{\partial}{\partial z^{\prime}_j}
F^{(i)}(x,y)=0\,, \quad \text{near $j_{X}(X)$},
\end{equation}
where $y= \exp_{x}^{X}(Z^{\,\prime})$.
Assume now that $\alpha \in\N^{n}$ and \eqref{b6.432} holds for $|\alpha|-1$. 
Consider $j$ with $\alpha_j>0$ and set $\alpha^\prime=(\alpha_1,
\cdots, \alpha_j-1,\cdots, \alpha_{n})$.

Taking the derivative of \eqref{b6.32} and using the induction hypothesis 
and \eqref{b6.410}, we have
\begin{equation}\label{b6.42}
\frac{\partial ^\alpha}{\partial z^{\alpha}} F^{(i)}(\cdot, x)|_{x}
=\frac{\partial }{\partial z_j}
 j_{X}^* \Big (\frac{\partial ^{\alpha^\prime}}{\partial z^{\alpha^\prime}}
 F^{(i)}\Big)\Big|_{x}
- \frac{\partial ^{\alpha^\prime}}{\partial z^{\alpha^\prime}}
\frac{\partial }{\partial z^{\prime}_j}F^{(i)}(\cdot, \cdot)\Big|_{0,0}
=0.
\end{equation}
Thus \eqref{d6.38} is proved. The identity \eqref{6.381} follows too, since it is equivalent to \eqref{d6.38}.
Further, \eqref{c6.38} results from \eqref{6.381} and \eqref{c6.31}.
This finishes the proof of Lemma \ref{toet3.6}.
\end{proof}
%------------------------------------------------------------------------------
\begin{lemma}\label{toet3.7}
We have
$F^{(i)}_x(z,\ov{z}^{\prime})=0$ for all $i>0$ and $z,z^{\prime}\in\C^n$\,.
\end{lemma}
%------------------------------------------------------------------------------
\begin{proof}
Let us consider the operator
\begin{equation}\label{c6.39}
\frac{1}{\sqrt{p}} P_{p} \left(\nabla^{E_p}_{X,x} 
T_{p}\right)P_{p}\, \mbox{  with  }\, X\in \cC^\infty(X,TX), \quad
X(x_0)= \frac{\partial}{\partial z_j}
+ \frac{\partial}{\partial\ov{z}_j}.
\end{equation}
The leading term of its asymptotic expansion \eqref{toe2.7} is %\eqref{toe3.2}
\begin{equation}\label{c6.40}
\Big(\frac{\partial}{\partial z_j}F_{x_0}\Big)
(\sqrt{p}\,z, \sqrt{p}\,\ov{z}^{\,\prime})
\cP_{x_0}(\sqrt{p}\,Z, \sqrt{p}\,Z^{\,\prime})\, .
\end{equation}
By \eqref{6.381} and \eqref{c6.38}, 
$(\frac{\partial}{\partial z_j}F_{x_0})(z,\ov{z}^{\prime})$
is an odd polynomial in $z$,$\ov{z}^{\prime }$ whose constant
term vanishes. We reiterate the arguments from \eqref{b6.33}--\eqref{b6.43} 
by replacing the operator $T_p$ with the operator \eqref{c6.39}; 
we get for $i>0$,
\begin{equation}\label{c6.41}
\frac{\partial}{\partial z_j}F_{x}^{(i)}(0,\ov{z}^{\,\prime})=0.
\end{equation}
By \eqref{c6.31} and  \eqref{c6.41},
\begin{equation}\label{c6.412}
\frac{\partial}{\partial \ov{z}^{\,\prime}_j}
F_{x}^{(i)}(z,0)=0.
\end{equation}
By continuing this process, we show that
for all $i>0, \alpha\in \Z^n$, $z,z^\prime\in \C^n$,
\begin{equation}\label{c6.413}
\frac{\partial^\alpha}{\partial z^\alpha}
F_{x}^{(i)}(0,\ov{z}^{\,\prime})=
\frac{\partial^\alpha}{\partial\ov{z}^{\,\prime\alpha}}
F_{x}^{(i)}(z,0)=0.
\end{equation}
Thus the Lemma is proved and \eqref{6.38} holds true.
\end{proof}
%------------------------------------------------------------------------------
\noindent
The Lemma \ref{toet3.7} finishes the proof of Proposition \ref{toet3.2}.
\end{proof}
%-------------------------------------------------------------------------
\noindent
We come now to the proof of the first induction step leading to \eqref{toe3.4}.
%-------------------------------------------------------------------------
\begin{prop}\label{toet3.8}
We have 
$p^{-n} (T_{p} -  T_{g_0,\,p})_{x_0}(Z,Z^\prime)\cong \mO(p^{-1})$ 
{\rm(}in the sense of Notation \ref{noe2.7}{\rm)}.
Consequently $T_{p}=P_p\,g_0\,P_p + \mO(p^{-1})$ {\rm(}i.e., 
relation \eqref{toe3.6}{\rm)} holds true in the sense of \eqref{atoe2.1}\,.
\end{prop}
%----------------------------------------------------------------------
\begin{proof}
Let us compare the asymptotic expansion of $T_p$ and 
$T_{g_0,\,p}=P_p\,g_0\,P_p$. Using the Notation \ref{noe2.7}, 
the expansion \eqref{toe2.13} (for $k=1$) reads
%----------------------------------------------------------------------
\begin{equation}\label{6.33a}
p^{-n}T_{g_0,\,p,\,x_0}(Z,Z^\prime)  \cong 
(g_{0}(x_0) I_{\C\otimes E}\cP_{x_0}
+Q_{1,\,x_0}(g_0)\cP_{x_0}\,p^{-1/2})(\sqrt{p}Z,\sqrt{p}Z^\prime) 
+\mO(p^{-1})\,,
\end{equation}
%----------------------------------------------------------------------
since $\mQ_{0,\,x_0}(g_0)=g_0(x_0) I_{\C\otimes E}$ by \eqref{toe2.15}. 
The expansion \eqref{toe3.2} (also for $k=1$) takes the form
\begin{equation}\label{6.33b}
p^{-n} T_{p,\,x_0} \cong (g_{0}(x_0) I_{\C\otimes E}\cP_{x_0}
+\mQ_{1,\,x_0}\cP_{x_0}\,p^{-1/2})(\sqrt{p}Z,\sqrt{p}Z^\prime)+\mO(p^{-1})\,,
\end{equation}
where we have used Proposition \ref{toet3.2} and the definition \eqref{toe3.5} 
of $g_0$.
Thus, subtracting \eqref{6.33a} from \eqref{6.33b} we obtain
\begin{equation}\label{6.33d}
p^{-n} (T_{p} -  T_{g_0,\,p})_{x_0}(Z,Z^\prime)
 \cong \big((\mQ_{1,\,x_0}-Q_{1,\,x_0}(g_0))\cP_{x_0}\big)
(\sqrt{p}Z,\sqrt{p}Z^\prime)\,p^{-1/2}+\mO(p^{-1})\,.
\end{equation}
Thus it suffices to prove:
%----------------------------------------------------------------------
\begin{lemma}\label{toet3.9}
\begin{equation}\label{6.33e}
F_{1,\,x}:=\mQ_{1,\,x}-Q_{1,\,x}(g_0)\equiv0\,.
\end{equation}
\end{lemma}
%----------------------------------------------------------------------
\begin{proof}
We note first that $F_{1,\,x}$ is an odd polynomial in $z$ and
$\ov{z}^{\,\prime}$; we verify this statement as in Lemma \ref{toet3.3}. Thus the constant term of $F_{1,\,x}$ vanishes.
To show that the rest of the terms vanish, we consider the decomposition
$F_{1,\,x}=\sum_{i\geqslant0}F^{(i)}_{1,\,x}$
in homogeneous polynomials $F^{(i)}_{1,\,x}$ of degree $i$. To prove \eqref{6.33e} it suffices to show that
%-----------------------------------------------------------------------------
\begin{equation}\label{6.38a}
F^{(i)}_{1,\,x}(z,\ov{z}^{\prime})=0 \quad \text{for all $i>0$ 
and $z,z^{\prime}\in \C^n$}\,.
\end{equation}
%-----------------------------------------------------------------------------
The proof of \eqref{6.38a} is similar to that of \eqref{6.38}. 
Namely, we define as in \eqref{b6.32} the operator $F^{(i)}_{1}$, 
by replacing $F^{(i)}_{x}(0,\ov{z}^{\,\prime})$ by 
$F^{(i)}_{1,\,x}(0,\ov{z}^{\,\prime})$, and we set 
(analogously to \eqref{b6.33})
\begin{equation}\label{b6.33a}
\cT_{p,1}= T_{p}-P_p\,g_0\,P_p 
-\sum_{i\leqslant\deg F_{1}} (F^{(i)}_{1}P_{p})\, p^{(i-1)/2}.
\end{equation}
Due to \eqref{toe2.13} and  \eqref{toe3.2}, 
there exist polynomials $\wi{R}_{r,\,x_0}\in\C[Z,Z^{\,\prime}]$
of the same parity as $r$ such that 
the following expansion in the normal coordinates
around $x_0\in X$ holds for $k\geqslant2$ and 
$\abs{Z^{\prime}}\leqslant\var^\prime/2$:
\begin{align}\label{b6.34a}
 p^{-n} \cT_{p,1,\,x_0} (0,Z^{\prime})
 \cong \sum_{r=2}^k  (\wi{R}_{r,\,x_0} \cP_{x_0}) (0,\sqrt{p} Z^{\prime}) 
p^{-\frac{r}{2}} + \mO(p^{-(k+1)/2}),
\end{align}
This is the analogue of \eqref{b6.34}.
%-----------------------------------------------------------------------------
Now we can repeat with obvious modifications the proof of \eqref{6.38} 
and obtain the analogue of \eqref{6.38} with $F_x$ replaced by $F_{1,\,x}$\,. 
This completes the proof of Lemma \ref{toet3.9}.
\end{proof}
\noindent
Lemma \ref{toet3.9} and the expansion \eqref{6.33d} imply immediately 
Proposition \ref{toet3.8}.
\end{proof}

\begin{proof}[Proof of Theorem \ref{toet3.1}]
Proposition \ref{toet3.8} shows that the asymptotic expansion 
\eqref{toe3.4} of $T_p$  holds for $m=0$. Moreover, if $T_{p}$ is self-adjoint, then from (4.70), (4.71) follows that $g_{0}$ is also self-adjoint.  We show inductively that
 \eqref{toe3.4} holds for every $m\in\N$. To prove \eqref{toe3.4} 
for $m=1$ let us consider the operator $p(T_p-P_pg_0P_p)$.
We have to show now that $p\big(T_p-T_{g_0,\,p}\big)$ 
satisfies the hypotheses of Theorem \ref{toet3.1}.
The first two conditions are easily verified. To prove the third, 
just substract the asymptotics of 
$T_{p,\,x_0}(Z,Z^\prime)$ (given by \eqref{toe3.2}) and 
$T_{g_0,\,p,\,x_0}(Z,Z^\prime)$ (given by \eqref{toe2.13}). 
Taking into account Proposition \ref{toet3.2} and \eqref{6.33e} 
the coefficients of $p^0$ and 
$p^{-1/2}$ in the difference vanish, which yields the desired conclusion. 

Propositions \ref{toet3.2} and \ref{toet3.8} applied to
$p(T_p-P_pg_0P_p)$ yield $g_1\in \cC^\infty(X,\End(E))$
such that \eqref{toe3.4} holds true for $m=1$.

We continue in this way the induction process to get \eqref{toe3.4}
for any $m$. This completes the proof of Theorem \ref{toet3.1}.
\end{proof}

%%%%%%%%%%%%%%%%%%%%%%%%%%%%%%%%%%%%%%%%%%%%%%%%%%%%%%%%%%%%%%%%%%%%%%%
\subsection{Algebra of Toeplitz operators}\label{toes4}

The Poisson bracket 
$\{ \,\cdot\, , \,\cdot\, \}$
on $(X,2\pi \om)$ is defined as follows.  For $f, g\in \cC^\infty (X)$,
let $\xi_{f}$ be the Hamiltonian vector field generated by $f$, 
which is defined by $2 \pi i_{\xi_{f}}\om=df$. Then 
\begin{align}\label{toe4.1}
\{f, g\}:= \xi_{f}(dg).
\end{align}

%-----------------------------------------------------------------------------
One of our main goals is to show that Theorem \ref{toet4.1} holds, 
thus the set of Toeplitz operators is 
closed under the composition of operators, so forms an algebra. 

%-------------------------------------------------------------------------
\begin{proof}[Proof of Theorem \ref{toet4.1}]
Firstly, it is obvious that $P_p\,T_{f,\,p}\,T_{g,\,p}\,P_p=T_{f,\,p}\,T_{g,\,p}$. 
Lemmas \ref{toet2.1} and \ref{toet2.3} imply
$T_{f,\,p}\,T_{g,\,p}$ verifies \eqref{toe3.1}.
Like in \eqref{toe2.17}, we have
for $Z,Z^\prime \in T_{x_0}X$, $\abs{Z},\abs{Z^{\prime}}<\var/4$:
\begin{multline}\label{toe4.5}
(T_{f,\,p}\,T_{g,\,p})_{x_0}(Z,Z^\prime)= \int_{{ T_{x_0}X}}
T_{f,\,p,\,x_0}(Z,Z^{\prime\prime})
\rho(4|Z^{\prime\prime}|/\var)
T_{g,\,p,\,x_0}(Z^{\prime\prime},Z^{\prime}) \\
\times \kappa_{x_0}(Z^{\prime\prime})
\,dv_{TX}(Z^{\prime\prime})+ \cO(p^{-\infty}).
\end{multline}  
By Lemma \ref{toet2.3} and \eqref{toe4.5}, we deduce as in the proof of 
Lemma \ref{toet2.3}, that for $Z,Z^\prime \in T_{x_0}X$, 
$\abs{Z},\abs{Z^{\prime}}<\var/4$, we have
\begin{align} \label{toe4.6}
p^{-n}(T_{f,\,p}\,T_{g,\,p})_{x_0}(Z,Z^\prime)\cong 
\sum^k_{r=0}(Q_{r,\,x_0}(f,g)\cP_{x_0})(\sqrt{p}\,Z,\sqrt{p}\,Z^{\prime})
p^{-\frac{r}{2}} + \mO(p^{-\frac{k+1}{2}}),
\end{align}
and with the notation \eqref{toe1.6}, 
\begin{align} \label{toe4.7}
Q_{r,\,x_0}(f,g)= \sum_{r_1+r_2=r}  \cK[Q_{r_1,\,x_0}(f), Q_{r_2,\,x_0}(g)].
\end{align}
Thus $T_{f,\,p}\,T_{g,\,p}$ is a Toeplitz operator by Theorem \ref{toet3.1}.
Moreover, it follows from the proofs of Lemma \ref{toet2.3} 
and Theorem \ref{toet3.1}
that $g_l=C_l(f,g)$, where $C_l$ are bidifferential operators. 

Recall that we denote by
$I_{\C\otimes E}: \Lambda (T^{*(0,1)}X)\otimes E\to \C\otimes E$ 
the natural projection. 
{}From \eqref{toe1.6}, \eqref{toe2.15} and \eqref{toe4.7}, we get
\begin{equation} \label{toe4.8}
C_0(f,g)(x)= I_{\C\otimes E}Q_{0,\,x}(f,g)|_{\C\otimes E}
= I_{\C\otimes E}\cK[Q_{0,\,x}(f), Q_{0,\,x}(g)]|_{\C\otimes E}
= f(x)g(x)\,.
\end{equation}

By the proof of Theorem \ref{toet3.1} 
(cf. Proposition \ref{toet3.2}, Lemma \ref{toet3.9} and \eqref{toe3.5}), we get
\begin{align} \label{toe4.10}
\begin{split}
&Q_{1,\,x}(f,g)=Q_{1,\,x}(C_0(f,g)),\\
&C_1(f,\,g)= I_{\C\otimes E}(Q_{2,\,x}(f,g)
- Q_{2,\,x}(C_0(f,\,g)))(0,0)|_{\C\otimes E}.
\end{split}\end{align}
Moreover, by \eqref{toe2.15} and \eqref{toe4.7}, we get
\begin{multline} \label{toe4.11}
Q_{2,\,x}(f,g) = \cK[f(x)I_{\C\otimes E}, Q_{2,\,x}(g)]
+\cK[Q_{1,\,x}(f), Q_{1,\,x}(g)]\\
 +\cK[Q_{2,\,x}(f), g(x) I_{\C\otimes E}].
\end{multline}

Now $T_{f,\,p}P_p=P_p T_{f,\,p}$ implies 
$Q_{r,x}(f,1)= Q_{r,x}(1,f)$, so we get from \eqref{toe4.11}:
\begin{multline} \label{toe4.13}
\cK[J_{0,\,x}, Q_{2,\,x}(f)] -\cK[Q_{2,\,x}(f),J_{0,\,x}]\\
=  \cK[Q_{1,\,x}(f),J_{1,\,x}]
 -\cK[J_{1,\,x}, Q_{1,\,x}(f)]
+ \cK[f(x)J_{0,\,x},J_{2,\,x}]-  \cK[J_{2,\,x}, f(x)J_{0,\,x}].
\end{multline}

Assume now that $f,g\in \cC^\infty (X)$.
By \eqref{toe4.10}, \eqref{toe4.11} and  \eqref{toe4.13}, we get
\begin{multline} \label{toe4.12}
C_1(f,g)(x)-C_1(g,f)(x)
= I_{\C\otimes E}\Big[ \cK[Q_{1,\,x}(f), Q_{1,\,x}(g)]
- \cK[Q_{1,\,x}(g), Q_{1,\,x}(f)]\\
+f(x)\Big( \cK[Q_{1,\,x}(g),J_{1,\,x}]- \cK[J_{1,\,x},Q_{1,\,x}(g)]\Big)\\
- g(x)\Big( \cK[Q_{1,\,x}(f),J_{1,\,x}]- \cK[J_{1,\,x},Q_{1,\,x}(f)]\Big)
\Big]\Big|_{\C\otimes E} .
\end{multline}

By Lemma \ref{toet2.5}, Remark \ref{toet2.6}, we have 
\begin{multline}\label{toe4.12a}
\cK[Q_{1,\,x}(f), Q_{1,\,x}(g)]
= \cK\Big[\cK[1, \tfrac{\partial f_x}{\partial Z_j}(0) Z_j], \,
\cK[1, \tfrac{\partial g_x}{\partial Z_j}(0) Z_j]\Big]\\
+ \cK[f(x)J_1,  Q_{1,\,x}(g)]+ \cK[Q_{1,\,x}(f), g(x)J_1]
- \cK[f(x)J_1, g(x)J_1].
\end{multline}

{}From \eqref{toe2.10}, 
\eqref{toe4.12} and \eqref{toe4.12a}, we get
\begin{multline}\label{toe4.14}
C_1(f,g)(x)-C_1(g,f)(x)
=\cK\Big[\cK[1, \tfrac{\partial f_x}{\partial Z_j}(0) Z_j], \,
\cK[1, \tfrac{\partial g_x}{\partial Z_j}(0) Z_j]\Big]\\
- \cK\Big[\cK[1, \tfrac{\partial g_x}{\partial Z_j}(0) Z_j], \,
\cK[1, \tfrac{\partial f_x}{\partial Z_j}(0) Z_j]\Big].
\end{multline}

{}From \eqref{toe1.13} we get
\begin{equation}\label{toe4.15}
\begin{split}
\cK\Big[1, \tfrac{\partial f_x}{\partial Z_j}(0) Z_j\Big]
= \tfrac{\partial f_x}{\partial z_i}(0)
z_i+\tfrac{\partial f_x}{\partial \ov{z}_i}(0)
\ov{z}^{\,\prime}_i.
\end{split}
\end{equation}
 
Plugging \eqref{toe4.15} into \eqref{toe4.14} and using \eqref{toe1.13} we finally obtain:
\begin{equation}\label{toe4.16}
\begin{split}
C_1(f,g)(x)-C_1(g,f)(x)&=\sum_{i=1}^n \frac{2}{a_i}
\Big[\tfrac{\partial f_x}{\partial \ov{z}_i}(0)\,
\tfrac{\partial g_x}{\partial z_i}(0)-\tfrac{\partial f_x}{\partial z_i}(0)\,
\tfrac{\partial g_x}{\partial \ov{z}_i}(0)\Big]\,\Id_{E}\\
&=\imat\{f,g\} \,\Id_{E}\,.
\end{split}
\end{equation} 
This finishes the proof of Theorem \ref{toet4.1}.
\end{proof}

%-----------------------------------------------------------------------------
The next result and Theorem \ref{toet4.1} show that 
the Berezin-Toeplitz quantization has the correct semi-classical behavior.
\begin{thm}\label{toet4.2} For $f\in \cC^\infty(X, \End(E))$,
the norm of $T_{f,\,p}$ satisfies
\begin{equation}\label{toe4.17}
\lim_{p\to\infty}\norm{T_{f,\,p}}={\norm f}_\infty
:=\sup_{0\neq u\in E_x, x\in X} |f(x)(u)|_{h^{E}}/ |u|_{h^{E}}.
\end{equation}
\end{thm}
\begin{proof}
Take a point $x_0\in X$ and $u_0\in E_{x_0}$ with $|u_0|_{h^{E}}=1$ 
such that $|f(x_0)(u_0)|={\norm f}_\infty$.
Recall that in Section \ref{toes2}, we trivialize the bundles $L$, $E$ 
in our normal coordinates near $x_0$, 
and $e_L$ is the unit frame of $L$ which trivialize $L$. 
Moreover, in this normal coordinates, $u_0$ is a trivial section of $E$.
Considering the sequence of sections 
$S^p_{x_0}=p^{-n/2}P_p(e_L^{\otimes p}\otimes u_0)$, 
we have by \eqref{aue66},
\begin{equation}\label{toe4.18}
\norm{T_{f,\,p}\,S^p_{x_0}-f(x_0)S^p_{x_0}}_{L^2}
\leqslant \tfrac{C}{\sqrt{p}}\norm{S^p_{x_0}}_{L^2} .
\end{equation}
If $f$ is a real function, then $df(x_0)=0$, so we can improve 
the constant  $\tfrac{C}{\sqrt{p}}$ in \eqref{toe4.18} to $\tfrac{C}{p}$.
The proof of \eqref{toe4.17} is complete.
\end{proof}

\begin{rem}\label{toet4.3}
For $E=\C$, Theorem \ref{toet4.1} shows that we can associate to $f,g\in \cC^\infty(X)$
a formal power series $\sum_{l=0}^\infty \hbar^{l}C_l(f,g)\in\cC^\infty(X)[[\hbar]]$, 
where $C_l$ are  bidifferential operators. 
Therefore, we have constructed in a canonical way an associative star-product
$f*g=\sum_{l=0}^\infty \hbar^{l}C_l(f,g)$, called the \textbf{\emph{Berezin-Toeplitz star-product}}\/.
\end{rem}

%%%%%%%%%%%%%%%%%%%%%%%%%%%%%%%%%%%%%%%%%%%%%%%%%%%%%
\section{Berezin-Toeplitz quantizations on non-compact manifolds}\label{toes5}
        
In this Section we extend our results to non-compact manifolds. 
We consider for simplicity only complex 
manifolds, that is, we suppose that $(X, J)$ 
is a complex manifold with complex structure $J$ 
and $E$, $L$ are holomorphic vector bundles on $X$ with ${\rm rk}(L)=1$. 
We assume that $\nabla ^E$, $\nabla ^L$ are the holomorphic Hermitian 
(i.e. Chern) connections on $(E,h^E)$, $(L,h^L)$.
Let $g^{TX}$ be any Riemannian metric on $TX$ compatible with $J$. 
Since $g^{TX}$ is not necessarily K{\"a}hler, the endomorphism $\bJ$ 
defined in \eqref{to1.2} does not satisfy $\bJ\neq J$ in general.
Set 
\begin{equation} \label{toe5.03}
\Theta(X,Y)=  g^{TX}(JX,Y).
\end{equation}
Then the 2-form $\Theta$ need not be closed.

Let $\overline{\partial} ^{L^p\otimes E,*}$ be the formal adjoint of
the Dolbeault operator $\overline{\partial} ^{L^p\otimes E}$ 
on the Dolbeault complex $\Omega ^{0,\sbullet}(X, L^p\otimes E)$ 
with the Hermitian product 
induced by $g^{TX}$, $h^L$, $h^E$ as in \eqref{l2}. Set 
\begin{equation}\label{toe5.04}
\begin{split}
&D_p = \sqrt{2}\big(\, \overline{\partial} ^{L^p\otimes E}
+ \,\overline{\partial} ^{L^p\otimes E,*}\big)\,,\\
&\Box_p= 
\overline{\partial}^{L^p\otimes E}\,\overline{\partial}^{L^p\otimes E,*}
+\,\overline{\partial}^{L^p\otimes E,*}\,\overline{\partial}^{L^p\otimes E}.
\end{split}\end{equation}
Then $\Box_p$ is the Kodaira-Laplacian 
which preserves the $\Z$-grading of 
$\Omega ^{0,\sbullet}(X, L^p\otimes E)$ and
\begin{equation}\label{toe5.05}
D^2_p=2\Box_p.
\end{equation}
Note that $D_{p}$ is not a spin$^c$ Dirac operator on 
$\Omega ^{0,\sbullet}(X, L^p\otimes E)$.

The space of holomorphic sections of $L^p\otimes{E}$ which are $L^2$ 
with respect to the norm given by \eqref{l2} is denoted 
by $H^0_{(2)}(X,L^p\otimes{E})$. 
Let $P_p(x,x')$, $(x,x'\in X)$ be the Schwartz kernel of 
the orthogonal projection $P_p$, from the space of $L^2$ sections 
of $L^p\otimes{E}$ onto $H^0_{(2)}(X,L^p\otimes{E})$, 
with respect to the Riemannian volume form $dv_X(x')$
associated to $(X,g^{TX})$. 
Then $P_p(x,x')$ is smooth by the ellipticity of the Kodaira Laplacian 
and the Schwartz kernel theorem (cf. also \cite[Remark 1.3.3]{MM:07}).

\begin{rem} \label{toet5.6}
If $\bJ=J$, then $(X,J,\Theta)$ is K{\"a}hler and $D_p$ in \eqref{defDirac} 
and \eqref{toe5.04} coincide. Assume moreover $X$ is compact. Then
by the Kodaira vanishing theorem and the Dolbeault isomorphism we have 
\begin{equation}\label{toe5.06}
H^0(X,L^p\otimes E)=\ke (D_p)\,,
\end{equation} 
for $p$ large enough.
Thus if $(X,J,\Theta)$ is a compact K{\"a}hler manifold and $\bJ=J$, $E=\C$,
Theorems \ref{toet4.1}, \ref{toet4.2} recover the main results of 
Bordemann, Meinrenken and Schlichenmaier \cite{BMS,Schlich:00,KS01,Charles1}.
\end{rem}

We denote by $R^{\det}$ the curvature of the holomorphic Hermitian 
connection $\nabla^{\det}$ on $K_X^*=\det (T^{(1,0)}X)$.

For a $(1,1)$-form $\Omega$, we write $\Omega>0$ (resp. $\geqslant 0$) 
if $\Omega(\cdot, J\cdot)>0$ (resp. $\geqslant 0$).

The following result, obtained in \cite[Theorem\,3.11]{MM04a}, extends 
the asymptotic expansion of the Bergman kernel to non-compact manifolds.

%----------------------------------------------------------------------
\begin{thm} \label{noncompact}
Suppose that that $(X,g^{TX})$ is a complete Hermitian manifold 
and there exist $\varepsilon>0\,,\,C>0$ such that\,{\rm:} 
\begin{equation}\label{toe5.08}
\sqrt{-1}R^L >\varepsilon\Theta\,,
\quad\,\sqrt{-1}(R^{\det}+R^E)> -C\Theta \Id_E\,,\quad\,
|\partial \Theta|_{g^{TX}}< C,
\end{equation} 
then the kernel $P_p(x,x')$
has a full off--diagonal asymptotic expansion analogous to 
that of Theorem \ref{tue17} uniformly for any $x,x'\in K$, 
a compact set of $X$.
If $L=K_X:=\det(T^{*(1,0)}X)$ is the canonical line bundle on $X$, the first two conditions in (5.5) are to be replaced by    \begin{equation*}    \text{$h^L$ is induced by $\Theta$ and     $\sqrt{-1}R^{\det}<-\varepsilon\Theta$,   $\sqrt{-1}R^E> -C\Theta \Id_E$. }    \end{equation*}
\end{thm} 
%-----------------------------------------------------------------------
The idea of the proof is that \eqref{toe5.08} together with 
the Bochner-Kodaira-Nakano formula imply the
existence of the spectral gap for $\Box_p$ acting on 
$L^2(X,L^p\otimes E)$ as in \eqref{diag5}.

Let $\cC^\infty_{const}(X, \End(E))$ denote the algebra of smooth sections 
of $X$ which are constant map outside a compact set.
For any $f\in\cC^\infty_{const}(X, \End(E))$, we consider the Toeplitz 
operator $(T_{f,\,p})_{p\in\N}$ as in \eqref{toe2.4}:
%----------------------------------------------------------------------------
\begin{equation}\label{toe5.1}
T_{f,\,p}:L^2(X,L^p\otimes E)\longrightarrow L^2(X,L^p\otimes E)\,,
\quad T_{f,\,p}=P_p\,f\,P_p\,.
\end{equation} 
%-----------------------------------------------------------------------
The following result generalizes Theorems \ref{toet4.1} and \ref{toet4.2}
to non-compact manifolds.  
%-----------------------------------------------------------------------
\begin{thm}\label{toet5.1}
Assume that $(X,g^{TX})$ is a complete Hermitian manifold, $(L,h^L)$ and 
$(E,h^E)$ are holomorphic vector bundles satisfying 
the hypotheses of Theorem \ref{noncompact} with ${\rm rk}(L)=1$.
Let $f,g\in\cC^\infty_{const}(X, \End(E))$. Then the following assertions hold:

(i) The product of the two corresponding Toeplitz operators admits 
the asymptotic expansion \eqref{toe4.2} in the sense of \eqref{atoe2.2}, 
where $C_r$ are bidifferential operators, 
especially, $\supp (C_r(f,g))\subset \supp(f)$ $\cap \supp(g)$,
and $C_0(f,g)=fg$.

(ii) If $f,g\in\cC^\infty_{const}(X)$, then \eqref{toe4.4} holds.

(iii) Relation \eqref{toe4.17} also holds for any 
$f\in\cC^\infty_{const}(X, \End(E))$.
\end{thm}
%-----------------------------------------------------------------------
\begin{proof} 
The most important observation here is that the spectral gap property \eqref{diag5} and a similar argument as in 
Proposition \ref{0t3.0} deliver
\begin{equation} \label{toe5.3}
F(D_p)s = P_p\, s, \quad \|F(D_p)-P_p\|= \cO(p^{-\infty}),
\end{equation}
for $p$ large enough and each $s\in H^0_{(2)}(X, L^p\otimes E)$. 
Moreover, by the proof of Proposition \ref{0t3.0},
for any compact set $K$, and any $l,m\in \N$, $\var>0$,
there exists $C_{l,m,\var}>0$
such that  
\begin{equation}\label{toe5.4}
|F(D_p)(x,x') - P_p(x,x')|_{\cC^m(K \times K)} 
\leqslant C_{l,m,\var} \, p^{-l}\,,
\end{equation}
for $p\geqslant 1$, $x,x'\in K$.
By the finite propagation speed for solutions of hyperbolic equations
\cite[\S 2.8]{Tay1:96}, \cite[Appendix D.2]{MM:07}
 (cf. also  \cite[Proposition 4.1]{DLM04a}), $F(D_p)(x,\cdot)$ 
only depends on the restriction of $D_p$ to $B^X(x,\var)$ 
and is zero outside  $B^X(x,\var)$. 

For $g\in\cC^\infty_0(X, \End(E))$, let $(F(D_p)\, g\,F(D_p))(x,x')$ 
be the smooth kernel of $F(D_p)\, g\,F(D_p)$ with respect to $dv_X(x')$.
Then for any relative compact open set $U$ in $X$ such that 
$\supp(g)\subset U$, we have from \eqref{toe5.3} and \eqref{toe5.4},
\begin{align}\label{toe5.5}
\begin{split}
&T_{g,\,p}-F(D_p)\, g\,F(D_p)= \mO(p^{-\infty}),\\
&T_{g,\,p}(x,x')-(F(D_p)\, g\,F(D_p))(x,x')
= \cO(p^{-\infty}) \quad \text{ on } U\times U.
\end{split}\end{align}

Now we fix $f,g\in\cC^\infty_0(X, \End(E))$.
Let $U$ be relative compact open sets in $X$ such that
$\supp (f)\cup\, \supp (g)\subset U$ and 
$d(x,y)>2\var$ for any $x\in \supp (f)\cup \supp (g), y\in X\smallsetminus U$.
{}From \eqref{toe5.3}, we have 
\begin{align} \label{toe5.6}
T_{f,\,p}\,T_{g,\,p} = P_p  F(D_p)f P_p \, g\,  F(D_p)P_p . 
\end{align}

Let  $(F(D_p)f P_p \, g\,  F(D_p))(x,x^\prime)$,
be the smooth kernel of $F(D_p)f P_p \, g\,  F(D_p)$ with respect to 
$dv_X(x^\prime)$. Then the support of $(F(D_p)f P_p \, g\, F(D_p))(\cdot,\cdot)$
is contained in $U\times U$.
If we fix $x_0\in U$, it follows from \eqref{toe5.4} that the kernel 
of $F(D_p)f P_p \, g\,  F(D_p)$ has exactly 
the same asymptotic expansion as in the compact case. 
More precisely, as in \eqref{toe4.6}, we have 
\begin{multline} \label{toe5.7}
p^{-n}(F(D_p)f P_p\, g\,  F(D_p))_{x_0}(Z,Z^\prime)\\
\cong 
\sum^k_{r=0}(Q_{r,x_0}(f,g)\cP_{x_0})(\sqrt{p}Z,\sqrt{p}Z^{\prime})
p^{-\frac{r}{2}} + \mO(p^{-\frac{k+1}{2}}),
\end{multline}
with the same local formula for $Q_{r,\,x_0}(f,g)$ given in \eqref{toe4.7}.

But since all formal computations are local, 
$Q_{r,\,x_0}(f,g)$ are the same as in the compact case, i.e. polynomials with 
coefficients bidifferential operators acting on $f$ and $g$.

Thus we know from \eqref{toe5.5} that there exist
$(g_l)_{l\geqslant0}$, where $g_l\in \cC^\infty_0(X, \End(E))$,
$\supp (g_l)\subset\supp (f)\cap \supp (g)$ such that 
for any ${k}\geqslant 1$, $s\in L^2(X, E_p)$,
\begin{equation} \label{toe5.8}
\Big\|F(D_p)f P_p g  F(D_p) s
- \sum_{l=0}^k F(D_p) P_{p}\, g_l \, p^{-l}\, P_{p} F(D_p)s\Big\|_{L^2}
\leqslant \frac{C}{p^{k+1}}\|s\|_{L^2}.
\end{equation}
\eqref{toe5.6} and \eqref{toe5.8} imply that 
\begin{equation} \label{toe5.9}
\Big\|T_{f,\,p}T_{g,\,p} -\sum_{l=0}^k P_{p}\, g_l \, p^{-l}\, P_{p}\Big\|
\leqslant C_k\,p^{-k-1}.
\end{equation}
Therefore, (i) is proved. With the asymptotic expansion at hand, 
we have just to repeat the proofs given in the compact case in order to 
verify assertions (ii) and (iii). More precisely, (ii) follows exactly in 
the same way as in the proof of \eqref{toe4.4} given in Theorem \ref{toet4.1}. 
Finally, to derive assertion (iii), we apply verbatim the proof 
of Theorem \ref{toet4.2}. This completes the proof of Theorem \ref{toet5.1}.
\end{proof}

\begin{ex}
Theorem \ref{toet5.1} holds for every quasi-projective manifold 
with $L$ the restriction of the hyperplane line bundle associated to
some arbitrary projective embedding and $E$ the trivial bundle. 
By definition, a quasi-projective manifold $X$ has the form
$X=Y\smallsetminus Z$, where $Y$ and $Z$ are projective varieties, 
and $Z\subset Y$ contains the singular set of $Y$.
Let us consider a holomorphic embedding $Y\subset\C\proj^m$, 
the hyperplane line bundle $\cO(1)$ on $\C\proj^m$, and set $L=\cO(1)|_X$.

By Hironaka's theorem of resolution of singularities there exists
 a projective manifold $\widetilde{Y}$ and a holomorphic map
$\pi:\widetilde{Y}\longrightarrow Y$ (a composition of a finite succession 
of blow-ups with smooth centers) such that
$\pi:\widetilde{Y}\smallsetminus\pi^{-1}(Z)\longrightarrow Y\smallsetminus Z$ 
is biholomorphic and $\pi^{-1}(Z)$ is a divisor with normal crossings. 
  
In this situation it is shown in \cite[\S\,3.6]{MM04a}, \cite[\S\,6.2]{MM:07} 
that there exist a complete K{\"a}hler metric $g^{TX}$ on 
$\widetilde{Y}\smallsetminus\pi^{-1}(Z)\simeq X$, 
called the generalized Poincar{\'e} metric, and a metric 
$h^L$ on $\pi^*L\simeq L$ satisfying the hypotheses of 
Theorem \ref{noncompact} (with $E$ trivial).
\end{ex}

\begin{rem}\label{toet5.3}
It is appropriate to remark that the results of 
Sections \ref{toes2}-\ref{toes4} learn that
we can associate to any $f,g\in \cC^\infty(X,\End(E))$ a formal power series 
$\sum_{l=0}^\infty \hbar^{l}C_l(f,g)\in\cC^\infty(X,\End(E))[[\hbar]]$, 
where $C_l$ are bidifferential operators. 
This follows from the fact that the construction 
in Section \ref{toes4} is local. 
However, the problem we addressed in this section is which Hilbert space 
the Toeplitz operators act on in the case of a non-compact manifold.
Theorem \ref{toe5.1} shows that 
the space of holomorphic $L^2$-sections $H^0_{(2)}(X,L^p\otimes E)$ 
of $L^p\otimes E$, is a suitable Hilbert space which allows 
the Berezin-Toeplitz quantization of the algebra 
$\cC^\infty_{const}(X,\End(E))$.
\end{rem}

%%%%%%%%%%%%%%%%%%%%%%%%%%%%%%%%%%%%%%%%%%%%%%%%%%%%%%%%%%%%%%%%%%%%%%%
\section{Berezin-Toeplitz quantization on orbifolds} \label{pbs4}
%%%%%%%%%%%%%%%%%%%%%%%%%%%%%%%%%%%%%%%%%%%%%%%%%%%%%%%%%%%%%%%%%%%%%%%

In this Section we establish the theory of Berezin-Toeplitz quantization on symplectic orbifolds, especially 
we show that set of Toeplitz operators forms an algebra.
For convenience of exposition, we explain the results in detail in the K{\"a}hler orbifold case.
In \cite[\S 5.4]{MM:07} we find more complete explanations and references 
for Sections \ref{pbs4.1} and \ref{pbs4.2}.
For related topics about orbifolds we refer to \cite{ALR}.

This Section is organized as follows. 
In Section \ref{pbs4.1} we recall the basic definitions about orbifolds.
In Section \ref{pbs4.2} we explain the asymptotic expansion
of Bergman kernel on complex orbifolds \cite[\S 5.2]{DLM04a}, which we apply in
Section \ref{pbs4.3} to derive the Berezin-Toeplitz quantization on K{\"a}hler orbifolds. Finally, we 
state in Section \ref{pbs4.4} the corresponding version for symplectic orbifolds.

\subsection{Basic definitions on orbifolds} \label{pbs4.1}

We define at first a  category $\mathcal{M}_s$ as follows~: 
The objects of $\mathcal{M}_s$
are the class of pairs $(G,M)$ where $M$ is a connected smooth manifold
and $G$ is a finite group acting effectively on $M$
(i.e., if $g\in G$ such that $gx=x$ for any $x\in M$, 
then $g$ is the unit element of $G$). If $(G,M)$ and
 $(G',M')$ are two objects, then a morphism $\Phi: (G,M)\rightarrow
(G',M')$ is a family of open embeddings $\varphi: M\rightarrow M'$
satisfying~:

{i)} For  each $\varphi \in \Phi $, there is an injective group
 homomorphism
$\lambda_{\varphi}~: G\rightarrow G' $ that makes $\varphi$ be
$\lambda_{\varphi}$-equivariant.

{ii)} For $g\in G', \varphi \in \Phi $, we define
$g\varphi : M \rightarrow M'$
by $(g\varphi)(x) = g\varphi(x)$ for $x\in M$.
If $(g\varphi)(M) \cap \varphi(M) \neq \emptyset$,
 then $g\in  \lambda_{\varphi}(G)$.

{ iii)} For $\varphi \in \Phi $, we have $\Phi = \{g\varphi, g\in G'\}$.

\begin{defn}\label{pbt4.1} Let $X$ be a paracompact Hausdorff space.
An $m$-dimensional \emph{orbifold chart} on $X$ consists of a
connected open set $U$ of $X$,
an object $(G_U,\widetilde{U})$ of $\mathcal{M}_s$ with $\dim\widetilde{U}=m$,
and a ramified covering $\tau_U:\widetilde{U}\to U$ which is 
$G_U$-invariant and induces a homeomorphism
$U \simeq \widetilde{U}/G_U$. We denote the chart by 
$(G_U,\widetilde{U})\stackrel{\tau_U}{\longrightarrow}U$. 

An $m$-dimensional \emph{orbifold atlas} $\mathcal{V}$ on $X$
consists of a family of $m$-dimensional orbifold charts 
$\mathcal{V} (U)
=((G_U,\widetilde{U})\stackrel{\tau_U}{\longrightarrow} U)$ 
satisfying the following conditions\,:

{i)} The open sets $U\subset X$ form a covering $\mathcal{U}$ with
 the property:
\begin{equation}\label{pb4.1}
 \text{For any $U, U'\in \mathcal{U}$ and $x\in U\cap U'$,
there is  
$U''\in \mathcal{U}$ such that $x\in U''\subset U\cap U'$}.
\end{equation}

{ii)} for any $U, V\in \mathcal{U}, U\subset V$ there exists a morphism
$\varphi_{VU}:(G_U,\widetilde{U})\rightarrow (G_V,\widetilde{V})$, 
which covers the inclusion $U\subset V$ and satisfies 
$\varphi_{WU}=\varphi_{WV} \circ \varphi_{VU}$
for any $U,V,W\in \mathcal{U}$, with $U\subset V \subset W$. 
\end{defn}

It is easy to see that there exists a unique maximal orbifold atlas 
$\mathcal{V}_{max}$ containing $\mathcal{V}$;
$\mathcal{V}_{max}$ consists of all orbifold charts 
$(G_U,\widetilde{U})\stackrel{\tau_U}{\longrightarrow} U$,
which are locally isomorphic to charts from $\mathcal{V}$ in the neighborhood 
of each point of $U$. A maximal orbifold atlas $\mathcal{V}_{max}$ 
is called an \emph{orbifold structure} and the pair $(X,\mathcal{V}_{max})$ 
is called an orbifold. As usual, once we have an orbifold atlas 
$\mathcal{V}$ on $X$ we denote the orbifold by $(X,\mathcal{V})$, 
since $\mathcal{V}$ determines uniquely $\mathcal{V}_{max}$\,.

Note that if $\mathcal{U}^\prime$ is a refinement of $\mathcal{U}$ 
satisfying \eqref{pb4.1}, then there is an orbifold
atlas $\mathcal{V}^\prime$ such that $\mathcal{V} \cup \mathcal{V}^\prime $
is an orbifold atlas, hence 
$\mathcal{V} \cup \mathcal{V}^\prime \subset \mathcal{V}_{max}$.  
This shows that we may choose $\mathcal{U}$ arbitrarily fine.

Let $(X,\mathcal{V})$ be an orbifold. For each $x\in X$, we can choose a
small neighborhood $(G_x, \widetilde{U}_x)\to U_x$ such
that $x\in \widetilde{U}_x$ is a fixed point of $G_x$
(it follows from the definition that such a $G_x$ is unique up to 
isomorphisms for each $x\in X$). 
We denote by  $|G_x|$ the cardinal of $G_x$.
If $|G_x|=1$, then $X$ has a smooth manifold structure in the neighborhood 
of $x$, which is called a smooth point of $X$. 
If  $|G_x|>1$, then $X$ is not a smooth manifold in the neighborhood of $x$, 
which is called a singular point of $X$. We denote by
$X_{sing}= \{x\in X; |G_x|>1\}$ the singular set of $X$,
and $X_{reg}= \{x\in X; |G_x|=1\}$ the regular set of $X$.

It is useful to note that on an orbifold $(X,\mathcal{V})$ we can 
construct partitions of unity. First, let us call a function on $X$ smooth, 
if its lift to any chart of the orbifold atlas $\mathcal{V}$ is smooth 
in the usual sense. Then the definition and construction of a smooth 
partition of unity associated to a locally finite covering carries over 
easily from the manifold case. The point is to construct 
smooth $G_U$-invariant functions with compact support on
$(G_U,\widetilde{U})$.

In Definition \ref{pbt4.1} we can replace $\mathcal{M}_s$ by a category 
of manifolds with an additional structure such as orientation, 
Riemannian metric, almost-complex structure or complex structure.
We impose that the morphisms (and the groups) preserve the specified
structure. So we can define oriented, Riemannian,
almost-complex or complex orbifolds.

Let $(X,\mathcal{V})$ be an arbitrary orbifold. By the above definition, 
a \emph{Riemannian metric} on $X$ is a Riemannian metric $g^{TX}$
on $X_{reg}$ such that the lift of $g^{TX}$ to any chart of the orbifold 
atlas $\mathcal{V}$ can be extended to a smooth Riemannian metric. 
Certainly, for any $(G_U, \wi{U})\in \mathcal{V}$, we can always 
construct a $G_U$-invariant Riemannian metric on $\wi{U}$. 
By a partition of unity argument, we see that there exist Riemannian metrics 
on the orbifold $(X,\mathcal{V})$.

\begin{defn}\label{pbt4.4}
 An orbifold vector bundle $E$ over an
orbifold $(X,\mathcal{V})$ is defined as follows\,: $E$ is an orbifold
and for $U\in \mathcal{U}$, $(G_U^{E}, \widetilde{p}_U:
\widetilde{E}_U \rightarrow \widetilde{U})$ is
 a $G_U^{E}$-equivariant vector bundle and $(G_U^{E}, \widetilde{E}_U)$
(resp. $(G_U=G_U^{E}/K_U^{E}, \widetilde{U})$,
$K_U^{E}= \ke (G_U^{E}\rightarrow\mbox{{\rm Diffeo}} (\widetilde{U})))$
is the orbifold structure of $E$ (resp. $X$). 
If $G_U^E$ acts effectively on $\widetilde{U}$ for $U\in \mathcal{U}$, 
i.e. $K_U^{E} = \{ 1\}$, we call $E$ a proper orbifold vector bundle.
\end{defn}

Note that any structure on $X$ or $E$ is locally
$G_x$ or $G_{U_x}^{E}$-equivariant.

\begin{rem}\label{pbt4.5}
Let $E$ be an orbifold vector bundle on $(X,\mathcal{V})$. 
For $U\in \mathcal{U}$,
let $\wi{E ^{\pr}_U}$ be the maximal $K_U^{E}$-invariant sub-bundle of
 $\wi{E}_U$ on $\wi{U}$. Then $(G_U, \wi{E ^{\pr}_U})$  defines a 
proper orbifold vector bundle on $(X, \mathcal{V})$, denoted by $E ^{\pr}$.

The  (proper)  orbifold tangent bundle $TX$ on an orbifold $X$ is defined by 
$(G_U, T\widetilde{U} \rightarrow \widetilde{U})$, for $U\in \mathcal{U}$.
In the same vein we introduce the cotangent bundle $T^*X$.
We can form tensor products of bundles by taking the tensor
products of their local expressions in the charts of an orbifold atlas.
Note that a Riemannian metric on $X$ induces a section of $T^*X\otimes T^*X$ over $X$
which is a positive definite bilinear form on $T_xX$ at each point $x\in X$.
\end{rem}

Let $E \rightarrow X$ be an orbifold vector bundle and 
$k\in \N\cup\{\infty\}$. A section $s: X\rightarrow E$ is called ${\cC}^k$
 if for each $U\in \mathcal{U}$, $s|_{U}$ is covered by 
a $G_U^{E}$-invariant ${\cC}^k$
section $\widetilde{s}_U : \widetilde{U} \rightarrow \widetilde{E}_U$.
 We denote by $\cC^k(X,E)$  
the space of $\cC^k$ sections of $E$ on $X$. 

If $X$ is oriented, we define the integral
 $\int_{X} \alpha$ for a form $\alpha$
 over $X$ (i.e. a section of $ \Lambda  (T^*X)$ over $X$) as follows. 
If $\supp(\alpha) \subset U\in \mathcal U$ set
\begin{equation}\label{pb4.5}
\int_{X} \alpha: = \frac{1}{|G_U|} \int_{\widetilde{U}}
\widetilde{\alpha}_U.
\end{equation}
It is easy to see that the definition is independent of the chart. For general $\alpha$ we extend the definition by using a partition of unity.

If $X$ is an oriented Riemannian orbifold,
there exists a canonical volume element $dv_X$ on $X$, which is a section of 
$\Lambda^m(T^*X)$, $m=\dim X$. Hence, we can also integrate functions on $X$.

Assume now that the Riemannian orbifold $(X,\mathcal V)$ is compact. 
For $x,y\in X$, put
\begin{eqnarray}\begin{array}{l}
d(x,y) = \mbox{Inf}_\gamma  \Big \{ \sum_i \int_{t_{i-1}}^{t_i}
|\frac{\partial }{\partial t}\widetilde{\gamma}_i(t)| dt \Big |
 \gamma: [0,1] \to X, \gamma(0) =x, \gamma(1) = y,\\
 \hspace*{15mm}  \mbox{such that there exist }   t_0=0< t_1 < \cdots < t_k=1,
\gamma([t_{i-1}, t_i])\subset U_i,  \\
\hspace*{15mm}U_i \in \mathcal{U}, \mbox{ and a } \cC^{\infty}
\mbox{ map  } \widetilde{\gamma}_i: [t_{i-1}, t_i] \to \widetilde{U}_i
 \mbox{ that covers } \gamma|_{[t_{i-1}, t_i]}   \Big \}.
\end{array}\nonumber
\end{eqnarray}
Then $(X, d)$ is a metric space.
For $x\in X$, set $d(x,X_{sing}):=\inf_{y\in X_{sing}} d(x,y)$.

Let us discuss briefly kernels and operators on orbifolds.
For any open set $U\subset X$ and orbifold chart
$(G_U,\widetilde{U})\stackrel{\tau_U}{\longrightarrow} U$, 
we will add a superscript $\, \wi{}\,$ to indicate the corresponding
objects on $\widetilde{U}$.
Assume that  
$\wi{\mK}(\wi{x},\wi{x}^{\,\prime})\in \cC^\infty(\wi{U} \times \wi{U},
\pi_1^* \wi{E}\otimes \pi_2^* \wi{E}^*)$
verifies 
\begin{align}\label{pb4.3}
(g,1)\wi{\mK}(g^{-1}\wi{x},\wi{x}^{\,\prime})
=(1,g^{-1})\wi{\mK}(\wi{x},g\wi{x}^{\,\prime})
\quad \text{ for any $g\in G_{U}$,}
\end{align}
where $(g_1,g_2)$ acts on $\wi{E}_{\wi{x}}\times \wi{E}_{\wi{x}^{\,\prime}}^*$ 
by $(g_1,g_2)(\xi_1,\xi_2)= (g_1\xi_1,g_2 \xi_2)$. 

We define the operator $\wi{\mK}: \cC^\infty_0(\wi{U}, \wi{E})
\to \cC^\infty(\wi{U}, \wi{E})$ by
\begin{equation}\label{pb4.4}
(\wi{\mK}\, \wi{s}) (\wi{x})= \int_{\wi{U}} \wi{\mK}(\wi{x},\wi{x}^{\,\prime}) 
 \wi{s}(\wi{x}^{\,\prime}) dv_{\wi{U}} (\wi{x}^{\,\prime}) 
\quad\text{for $\wi{s} \in \cC^\infty_0(\wi{U}, \wi{E})$\,.}
\end{equation}
For $\wi{s} \in \cC^\infty(\wi{U}, \wi{E})$ and $g\in G_U$,
$g$ acts on $\cC^\infty(\wi{U}, \wi{E})$ by: 
$(g \cdot \wi{s})(\wi{x}):=g \cdot \wi{s}(g^{-1}\wi{x})$. We can then identify 
an element $s \in \cC^\infty({U}, {E})$ 
with an element $\wi{s} \in \cC^\infty(\wi{U}, \wi{E})$ 
verifying $g\cdot \wi{s}=\wi{s}$ for any $g\in G_U$.

With this identification, we define the operator 
$\mK: \cC^\infty_0(U,{E})\to \cC^\infty({U},{E})$ by 
\begin{equation}\label{pb4.6}
({\mK} s)(x)=\frac{1}{|G_U|}\int_{\wi{U}} \wi{\mK}(\wi{x},\wi{x}^{\,\prime}) 
 \wi{s}(\wi{x}^{\,\prime}) dv_{\wi{U}} (\wi{x}^{\,\prime}) 
\quad\text{for $s \in \cC^\infty_0({U},{E})$\,,}
\end{equation}
where $\wi{x}\in\tau^{-1}_U(x)$.
Then the smooth kernel $\mK(x,x^\prime)$ of the operator $\mK$ 
with respect to $dv_X$ is
\begin{align}\label{pb4.7}
\mK(x,x^\prime)= \sum_{g\in G_U} (g,1)\wi{\mK}(g^{-1}\wi{x},\wi{x}^{\,\prime}).
\end{align}
Indeed, if $s \in \cC^\infty_0({U}, {E})$,
by \eqref{pb4.3} and \eqref{pb4.6}, we have
 \begin{equation}\label{pb4.8}
 \begin{split}
 (\mK s)(x) &= \frac{1}{|G_{U}|} \sum_{g\in G_{U}} 
\int_{\wi{U}}  \wi{\mK} (\wi{x},\wi{x}^{\,\prime})
g \cdot \wi{s}(g^{-1}\wi{x}^{\,\prime})(\wi{x}^{\,\prime}) 
dv_{\wi{U}}(\wi{x}^{\,\prime})\\
 &=  \frac{1}{|G_{U}|} \sum_{g\in G_{U}} 
\int_{\wi{U}} (g,1) \wi{\mK} (g^{-1}\wi{x},\wi{x}^{\,\prime})
s(\wi{x}^{\,\prime}) dv_{\wi{U}}(\wi{x}^{\,\prime})\\
 &=  \int_{U} \sum_{g\in G_{U}}
(g,1)\wi{\mK} (g^{-1}\wi{x},\wi{x}^{\,\prime}) 
s(x^\prime) dv_{X}(x^\prime). 
\end{split}
\end{equation}

Let $\mK_1,\mK_2$ be two operators as above and assume that the kernel of 
one of $\wi{\mK}_1, \wi{\mK}_2$ has compact support.
By \eqref{pb4.5}, \eqref{pb4.3} and \eqref{pb4.6}, 
the kernel of $\mK_1\circ \mK_2$ is given by
\begin{equation}\label{pb4.9}
(\mK_1\circ \mK_2)(x,x^\prime)= \sum_{g\in G_U} 
(g,1)(\wi{\mK}_1\circ \wi{\mK}_2)(g^{-1}\wi{x},\wi{x}^{\,\prime}).
\end{equation}

\subsection{Bergman kernel on K{\"a}hler orbifolds} \label{pbs4.2}
%%%%%%%%%%%%%%%%%%%%%%%%%%%%%%%%%%%%%%%%%%%%%%%%%%%%%%%%%%%%%%%%%%%%%%%

Let $X$ be a compact complex orbifold of complex dimension $n$
with complex structure $J$.
Let $E$ be a holomorphic orbifold vector bundle on $X$.

Let $\cO_X$ be the sheaf over $X$ of local $G_U$-invariant holomorphic
functions over $\widetilde{U}$, for $U\in {\mathcal U}$.
The local $G^{E}_U$ -invariant holomorphic sections of
$\widetilde{E} \rightarrow \widetilde{U}$ define 
a sheaf $\cO_X(E)$ over $X$.
Let $H^\bullet(X, \cO_X(E)$ be the cohomology of the sheaf
$\cO_X(E)$ over $X$.

Notice that by Definition, we have
\begin{equation}\label{pb4.10}
\cO_X(E)=\cO_X(E ^{\pr}).
\end{equation}
Thus without lost generality, we may and will assume that 
$E$ is a proper orbifold vector bundle on $X$.

Consider a section $s\in\cC^\infty(X,E)$ and a local section 
$\wi{s}\in\cC^\infty(\wi{U},\wi{E}_U)$ covering $s$.
Then $\db^{\wi{E}_U}\wi{s}$ covers a section of $T^{*(0,1)}X\otimes E$ 
over $U$, denoted $\db^Es|_U$.
The family of sections $\{\db^Es|_U\,:\,U\in\mathcal{U}\}$ patch together 
to define a global section $\db^Es$ of $T^{*(0,1)}X\otimes E$ over $X$. 
In a similar manner we define $\overline\partial^E\alpha$ for a 
$\cC^{\infty}$ section $\alpha$ of $\Lambda (T^{*(0,1)}X) \otimes E$ over $X$. 
We obtain thus the Dolbeault complex 
($\Omega^{0,\bullet}(X,E), \overline{\partial}^E$)\,:
\begin{equation}\label{pb4.12}
0 \longrightarrow \Omega^{0,0}(X,E)
\stackrel{\overline{\partial}^E}{\longrightarrow} \cdots
\stackrel{\overline{\partial}^E}{\longrightarrow} \Omega^{0,n}(X,E)
\longrightarrow 0 .
\end{equation}
{}From the abstract de Rham theorem there exists a canonical isomorphism
\begin{eqnarray}\label{pb4.13}
H^\bullet(\Omega^{0,\bullet}(X,E), \overline{\partial}^E) \simeq
H^\bullet(X,\cO_X(E)).
\end{eqnarray}
In the sequel, we also denote  $H^\bullet(X,\cO_X(E))$ 
by $H^\bullet(X,E)$. 

We consider a complex orbifold $(X,J)$ endowed with the complex structure $J$. 
Let $g^{TX}$ be a Riemannian metric on $TX$ compatible with $J$.
There is then an associated $(1,1)$-form $\Theta$ given by 
$\Theta(U,V)=g^{TX}(JU,V)$. The metric $g^{TX}$ is called a K{\"a}hler metric 
and the orbifold $(X,J)$ is called a \textbf{\emph{K{\"a}hler orbifold}} 
if $\Theta$ is a closed form, that is, $d\Theta=0$.
In this case $\Theta$ is a symplectic form, called K{\"a}hler form. We will 
denote the K{\"a}hler orbifold by $(X,J,\Theta)$ or shortly by $(X,\Theta)$.

Let $(L,h^L)$ be a holomorphic Hermitian proper orbifold line bundle 
on an orbifold $X$, and let $(E,h^E)$ be a holomorphic Hermitian proper 
orbifold vector bundle on $X$.

We assume that the associated curvature $R^L$ of $(L,h^L)$ verifies 
\eqref{bk1.1}, i.e., $(L,h^L)$ is a positive proper orbifold line bundle 
on $X$. This implies that $\om:=\frac{\sqrt{-1}}{\pi}R^L$ is 
a K{\"a}hler form on $X$, $(X,\om)$ is a K{\"a}hler orbifold and 
$(L,h^L,\nabla^L)$ is a prequantum line bundle on $(X,\om)$.

Note that the existence of a positive line bundle $L$ on 
a compact complex orbifold $X$ implies that the Kodaira map
associated to high powers of $L$ gives a holomorphic embedding of $X$ 
in the projective space. This is the generalization due to Baily of 
the Kodaira embedding theorem (see e.g. \cite[Theorem\,5.4.20]{MM:07}).

Let $g^{TX}=\om(\cdot,J\cdot)$ be the Riemannian metric on $X$
induced by $\om=\frac{\sqrt{-1}}{2\pi}R^L$.

Using the Hermitian product along the fibers of $L^p$, $E$, 
$\Lambda(T^{*(0,1)}X)$, the Riemannian volume form $dv_X$ and 
the definition \eqref{pb4.5} of the integral on an orbifold, we introduce 
an $L^2$-Hermitian product on $\Omega^{0,\bullet}(X,L^p\otimes E)$ 
similar to \eqref{l2}. This allows to define the formal adjoint 
$\overline{\partial}^{L^p\otimes E,*}$ of $\overline{\partial}^{L^p\otimes E}$ 
and as in \eqref{toe5.04}, the operators $D_p$ and $\square_p$\,.
Then $D_p^2$ preserves the $\Z$-grading of 
$\Omega^{0,\bullet} (X,L^p\otimes E)$.
We note that Hodge theory extends to compact orbifolds and delivers a canonical isomorphism
\begin{equation}\label{pb4.16}
H^{q}(X,L^p\otimes E)\simeq \ke(D_p^2|_{\Omega^{0,q}}).
\end{equation}

By the same proof as in \cite[Theorems\,1.1,\,2.5]{MM02}, 
\cite[Theorem 1]{BVa89},
we get vanishing results and the spectral gap property.
%--------------------------------------------------------------------
\begin{thm}\label{pbt4.11} 
Let $(X,\om)$ be a compact K{\"a}hler orbifold, 
$(L,h^L)$ be a prequantum holomorphic Hermitian proper orbifold line bundle 
on $(X,\om)$ and $(E,h^E)$ be an arbitrary holomorphic Hermitian 
proper orbifold vector bundle on $X$.

Then there exists $C>0$ such that the Dirac operator $D_p$ satisfies 
for any $p\in \N$ 
\begin{equation}\label{pb4.17}
\spec(D_{p}^2)\subset \{0\}\, \cup\, ]4\pi p-C,+\infty[,
\end{equation}
 and $D_{p}^2|_{\Omega^{0, >0}}$ is invertible for $p$ large enough.
Consequently, we have the Kodaira-Serre vanishing theorem, namely, 
for $p$ large enough,
\begin{equation}\label{pb4.18}
H^{q}(X,L^p\otimes E)= 0\,, \quad \text{\rm for every $q>0$.}
\end{equation}
\end{thm}
%----------------------------------------------------------------------
\noindent
In view of Theorem \ref{pbt4.11} and of the isomorphism \eqref{pb4.16}, 
we can define for $p>C(2\pi)^{-1}$ the Bergman kernel 
$$P_p(\cdot,\cdot)\in\cC^\infty(X\times X,
\pi_1^*(L^p\otimes E)\otimes \pi_2^*((L^p\otimes E)^*))$$ 
like in Definition \ref{defBergman}. Namely, the Bergman kernel is 
the smooth kernel with respect to the Riemannian volume form $dv_X(x')$ of 
the orthogonal projection (Bergman projection) $P_p$ from 
$\cC^\infty(X, L^p\otimes E)$ onto $H^0(X,L^p\otimes E)$.

{}From now on, we assume  $p>C(2\pi)^{-1}$.
Let $d_p = \dim H^0(X,L^p\otimes E)$ and consider an arbitrary orthonormal 
basis $\{S^p_i\}_{i=1}^{d_p}$ of $H^0(X,L^p\otimes E)$ with respect to 
the Hermitian product \eqref{l2} and \eqref{pb4.5}. 
In fact, in the local coordinate above, $\widetilde{S}^p_i(\widetilde{z})$ 
are $G_x$-invariant on $\widetilde{U}_x$, and
\begin{align}\label{pb4.19}
P_p (y,y') &= \sum_{i=1}^{d_p}  \widetilde{S}^p_i ( \widetilde{y})
 \otimes (\widetilde{S}^p_i(\widetilde{y}^\prime))^*,
\end{align}
where we use $\widetilde{y}$ to denote the point in $\widetilde{U}_x$
representing $y\in {U}_x$. 

The spectral gap property \eqref{pb4.17} shows that we have the
analogue of Proposition \ref{0t3.0}, with the same $F$ as given in \eqref{0c3}:
\be\label{pb4.21}
|P_p(x,x') -  F(D_p)(x,x')|_{\cC^m(X\times X)}
\leqslant C_{l,m,\var} p^{-l}.
\ee

As pointed out in \cite{Ma05}, the property of the finite propagation speed
of solutions of hyperbolic equations still holds on an orbifold 
(see the proof in \cite[Appendix D.2]{MM:07}).
Thus $F(D_p)(x,x')= 0$ for every for $x,x'\in X$ satisfying 
$d(x, x')\geqslant \var$. 
Likewise, given $x\in X$, $F(D_p)(x,\cdot)$ only depends on the restriction 
of $D_p$ to $B^X(x,\var)$. Thus the problem of the asymptotic expansion 
of $P_p(x,\cdot)$ is local.

We recall that for every open set $U\subset X$ and orbifold chart
$(G_U,\widetilde{U})\stackrel{\tau_U}{\longrightarrow} U$,
we add a superscript $\, \wi{}\,$ to indicate the corresponding
objects on $\widetilde{U}$. 
Let $\partial U=\ov{U}\setminus U$, $U_1=\{ x\in U, d(x,\partial U)<\var\}$. 
Then $F(\wi{D}_p)(\wi{x},\wi{x}^{\,\prime})$  is well defined for 
$\wi{x},\wi{x}^{\,\prime}\in \wi{U}_1=\tau_U^{-1}(U_1)$. Since 
$g\cdot F(\wi{D}_p)=F(\wi{D}_p) g$, we get
\begin{equation}\label{pb4.22a}
(g, 1) F(\wi{D}_p)( g^{-1} \wi{x},\wi{x}^\prime)
= (1, g^{-1}) F(\wi{D}_p)(\wi{x}, g\wi{x}^\prime)\,,
\end{equation}
for every $g\in G_U$, 
$\wi{x},\wi{x}^{\,\prime}\in \wi{U}_1$. 
Formula \eqref{pb4.7} shows that for every $x,x^\prime\in U_1$ 
and $\wi{x},\wi{x}^{\,\prime}\in \wi{U}_1$ representing $x,x^\prime$, we have
\begin{equation}\label{pb4.22}
F(D_p)(x,x^\prime) = \sum_{g\in G_{U}}
(g, 1) F(\wi{D}_p)( g^{-1}\wi{x},\wi{x}^{\,\prime}).
\end{equation}

For $\wi{x}_0\in \wi{U}_2 :=\{ x\in \wi{U}, d(x,\partial \wi{U})<2\var\}$, 
and $\wi{Z},\wi{Z}^\prime \in T_{\wi{x}_0}X$ with 
$|\wi{Z}|,|\wi{Z}^\prime|\leqslant \var$, 
the kernel $F(\wi{D}_p)(\wi{Z},\wi{Z}^\prime)$ has an asymptotic expansion as 
in Theorem \ref{tue17} by the same argument as in Proposition \ref{0t3.0}. 
In the present situation $\bJ=J$, so that $a_j=2\pi$ 
and the kernel $\cP$ defined in \eqref{3ue62} takes the form
\begin{equation}\label{pb4.25}
\cP(\wi{Z},\wi{Z}^\prime) =\exp\Big(-\frac{\pi}{2}\sum_i\big(|\wi{z}_i|^2
+|\wi{z}^{\prime}_i|^2 -2\wi{z}_i\wi{\overline{z}}_i^\prime\big)\Big)\,.
\end{equation}

\subsection{Berezin-Toeplitz quantization on K{\"a}hler orbifolds} \label{pbs4.3}

We apply now the results of Section \ref{pbs4.2} to establish 
the Berezin-Toeplitz quantization on K{\"a}hler orbifolds.
We use the notations and assumptions of that Section.% \ref{pbs4.2}.

Since we consider the holomorphic case, we denote directly by $P_p$ 
the orthogonal projection
from $\cC^\infty(X, L^p\otimes E)$ onto $H^0(X,L^p\otimes E)$ and
we replace in \eqref{toe2.1} the space $L^2(X,E_p)$ with $L^2(X,L^p\otimes E)$.
Thus we have the following definition.

\begin{defn}\label{toet6.0}
A \textbf{\em Toeplitz operator}
is a family $\{T_p\}$ of linear operators
\begin{equation}\label{toe6.1}
T_{p}:L^2(X, L^p\otimes E)\longrightarrow L^2(X, L^p\otimes E)\,,
\end{equation} 
verifying \eqref{toe2.2} and \eqref{toe2.3}.
\end{defn}

For any section $f\in \cC^{\infty}(X,\End(E))$,
  the  \textbf{\em Berezin-Toeplitz quantization} of $f$ is defined by
\begin{equation}\label{toe6.3}
T_{f,\,p}:L^2(X,L^p\otimes E)\longrightarrow L^2(X,L^p\otimes E)\,,
\quad T_{f,\,p}=P_p\,f\,P_p\,.
\end{equation} 

Now, by the same argument as in Lemma \ref{toet2.1}, we get
%---------------------------------------------------------------------
\begin{lemma} \label{toet6.1}
For any $\varepsilon>0$ and any $l,m\in\N$ there exists $C_{l,m,{\varepsilon}}>0$ such that 
%-------------------------------------------------------------------------
\begin{equation} \label{toe6.4}
|T_{f,\,p}(x,x')|_{\cC^m(X\times X)}\leqslant C_{l,m,{\varepsilon}}p^{-l}
\end{equation}
%---------------------------------------------------------------------
for all $p\geqslant 1$ and all $(x,x')\in X\times X$ 
with $d(x,x')>\varepsilon$, where the $\cC^m$-norm is induced by $\nabla^L,\nabla^E$ and $h^L,h^E,g^{TX}$.
\end{lemma} 

As in Section \ref{ts2} we obtain next the asymptotic expansion of 
the kernel $T_{f,\,p}(x,x')$ in a neighborhood of the diagonal. 

We need to introduce the appropriate analogue of the Condition \ref{coe2.7}
 in the orbifold case, in order to take into account the group action 
associated to an orbifold chart.
Let $\{\Xi_p\}_{p\in\N}$ be a sequence of linear operators 
$\Xi_p: L^2(X,L^p\otimes E)\longrightarrow L^2(X,L^p\otimes E)$ 
with smooth kernel $\Xi_p(x,y)$ with respect to $dv_X(y)$.

\begin{condition}\label{coe2.71}
Let $k\in\N$. Assume that for every open set $U\in\mathcal{U}$ and 
every orbifold chart
$(G_U,\widetilde{U})\stackrel{\tau_U}{\longrightarrow} U$, 
there exists a sequence of kernels 
$\{\wi{\Xi}_{p, U}(\wi{x},\wi{x}^{\,\prime})\}_{p\in\N}$ and a family 
$\{Q_{r,\,x_0}\}_{0\leqslant r\leqslant k,\,x_0\in X}$ such that 
\begin{itemize}
\item[(a)] $Q_{r,\,x_0}\in \End( E)_{x_0}[\wi{Z},\wi{Z}^{\prime}]$\,, 
\item[(b)] $\{Q_{r,\,x_0}\}_{r\in\N,\,x_0\in X}$ is smooth with respect 
to the parameter $x_0\in X$,
\item[(c)] for every fixed $\var''>0$ and every 
$\wi{x},\wi{x}^{\,\prime}\in \wi{U}$ the following holds
\begin{equation} \label{toe6.5}
\begin{split}
& (g,1)\wi{\Xi}_{p, U}(g^{-1}\wi{x},\wi{x}^{\,\prime})=
(1,g^{-1})\wi{\Xi}_{p, U}(\wi{x},g\wi{x}^{\,\prime})
\quad \text{for any } \, \, g\in G_{U}\; \text{(cf. \eqref{pb4.22a})}, \\
&\wi{\Xi}_{p, U}(\wi{x},\wi{x}^{\,\prime})= \cO(p^{-\infty}) \quad
 \quad\text{for}\, \,   d(x,x^\prime)>\var'',\\
&\Xi_{p}(x,x^\prime)
= \sum_{g\in G_{U}} (g,1)\wi{\Xi}_{p, U}(g^{-1}\wi{x},\wi{x}^{\,\prime})
+ \cO(p^{-\infty}),
\end{split}\end{equation}
and moreover, for every relatively compact open subset $\wi{V}\subset \wi{U}$, the relation
\begin{equation} \label{toe6.6}
p^{-n}\, \wi{\Xi}_{p,U,\wi{x}_0}(\wi{Z},\wi{Z}^\prime)\cong \sum_{r=0}^k 
(Q_{r,\,\wi{x}_0} \cP_{\wi{x}_0})
(\sqrt{p}\wi{Z},\sqrt{p}\wi{Z}^{\prime})p^{-\frac{r}{2}}
+\mO(p^{-\frac{k+1}{2}})\, \quad\mbox{  for }\, \,  \wi{x}_0\in \wi{V},
\end{equation}
holds in the sense of \eqref{toe2.7}. 
\end{itemize}
\end{condition}
\begin{notation}\label{noe2.71}
If the sequence $\{\Xi_p\}_{p\in\N}$ satisfies Condition \ref{coe2.71},
 we write
\begin{equation} \label{toe6.7}
p^{-n}\, \Xi_{p,\,x_0}(Z,Z^\prime)\cong \sum_{r=0}^k 
(Q_{r,\,x_0} \cP_{x_0})(\sqrt{p}Z,\sqrt{p}Z^{\prime})p^{-\frac{r}{2}}
+\mO(p^{-\frac{k+1}{2}})\,.
\end{equation}
\end{notation}
Note that although the Notations \ref{noe2.71} and \ref{noe2.7} 
are formally similar, they have different meaning.

\begin{lemma}\label{toet6.2} The smooth family
$Q_{r,\,x_0}\in \End( E)_{x_0}[\wi{Z},\wi{Z}^{\prime}]$ 
in Condition \ref{coe2.71} is uniquely determined by $\Xi_p$.
\end{lemma}
\begin{proof} Clearly, for $W\subset U$, the restriction of $\wi{\Xi}_{p,U}$ 
to $\wi{W} \times \wi{W}$ verifies \eqref{toe6.5}, thus we can take 
$\wi{\Xi}_{p,W}=\wi{\Xi}_{p,U}|_{\wi{W} \times \wi{W}}$.
Since $G_U$ acts freely on $\tau_U^{-1}(U_{reg})\subset \wi{U}$, 
we deduce from \eqref{toe6.5} and \eqref{toe6.6} that 
\begin{align}\label{toe6.8}
\Xi_{p,\,x_0}(Z,Z^\prime)= \wi{\Xi}_{p,\,U,\,\wi{x}_0}(\wi{Z},\wi{Z}^\prime)
+ \cO(p^{-\infty})\,,
\end{align}
for every $x_0\in U_{reg}$
and $|\wi{Z}|,|\wi{Z}^\prime|$ small enough.
We infer from  \eqref{toe6.6} and \eqref{toe6.8} that
$Q_{r,\,x_0}\in \End( E)_{x_0}[\wi{Z},\wi{Z}^{\prime}]$ 
is uniquely determined for $x_0\in X_{reg}$\,. 
Since $Q_{r,\,x_0}$ depends smoothly on $x_0$, 
its lift to $\wi{U}$ is smooth. Since the set $\tau_U^{-1}(U_{reg})$ 
is dense in $\wi{U}$, we see that
the smooth family $Q_{r,\,x_0}$ is uniquely determined by $\Xi_p$.
\end{proof}

\begin{lemma}\label{toet6.3} There exist polynomials $J_{r,\,x_0}, Q_{r,x_0}(f)$ $\in \End(E)_{x_0}[\wi{Z},\wi{Z}^{\prime}]$ such that  
Lemmas \ref{toet2.1}, \ref{toet2.2}, \ref{toet2.3} and \ref{toet2.5} 
still hold under the notation \eqref{toe6.7}.
Moreover, 
\begin{equation}\label{toe6.15}
J_{0,\,x_0}= \Id_E, \quad J_{1,\,x_0}=0.
\end{equation}
\end{lemma}
\begin{proof}   
The analogues of Proposition \ref{0t3.0}, Theorem \ref{tue17} for 
the current situation and \eqref{pb4.22a}, \eqref{pb4.22} show that
Lemmas \ref{toet2.1} and \ref{toet2.2} still hold under 
the notation \eqref{toe6.7}.
Since in our case $\om$ is a K{\"a}hler form with respect to 
the complex structure $J$ and $\bJ=J$, we have $\mO_1=0$ 
(cf. \eqref{1c31} and Remark \ref{rue}). 
Hence \eqref{abk2.76} entails \eqref{toe6.15}.
Moreover, \eqref{pb4.21} implies
\begin{align}\label{toe6.17}
T_{f,\,p}(x,x^\prime)= \int_{X} F(D_p)(x,x'')f(x'')
 F(D_p)(x'',x^\prime) dv_X(x'') + \cO(p^{-\infty}).
\end{align}
Therefore, we deduce from \eqref{pb4.9}, \eqref{pb4.22a},
\eqref{pb4.22} and \eqref{toe6.17} that Lemmas \ref{toet2.3} 
and \ref{toet2.5} still hold under the notation \eqref{toe6.7}.
\end{proof}

We will prove next a useful criterion (an analogue of 
Theorem \ref{toet3.1}) which ensures that
a given family is a Toeplitz operator. 
%-----------------------------------------------------------------------
\begin{thm}\label{toet6.4}
Let $\{T_p:L^2(X,L^p\otimes E)\longrightarrow L^2(X,L^p\otimes E)\}$ 
be a family of bounded linear operators which satisfies 
the following three conditions:
\begin{itemize}
\item[(i)] For any $p\in \N$,  $P_p\,T_p\,P_p=T_p$\,.
\item[(ii)] For any $\varepsilon_0>0$ and any $l\in\N$, 
there exists $C_{l,\varepsilon_0}>0$ such that 
for all $p\geqslant 1$ and all $(x,x')\in X\times X$ 
with $d(x,x')>\varepsilon_0$,
%--------------------------------------------------------------------
\begin{equation} \label{toe6.21}
|T_{p}(x,x')|\leqslant C_{l,{\varepsilon_0}}p^{-l}.
\end{equation}
%-----------------------------------------------------------------------
\item[(iii)] There exists a family of polynomials 
$\{\mQ_{r,\,x_0}\in\End(E)_{x_0}[Z,Z^{\prime}]\}_{x_0\in X}$ 
such that\,{\rm:} 
\begin{itemize}
\item[(a)] each $\mQ_{r,\,x_0}$ has the same parity as $r$, 

\item[(b)] the family is smooth in $x_0\in X$ and 

\item[(c)] there exists $\var^\prime\in]0,a_X/4[$ such that for every 
$x_0\in X$, every $Z,Z^\prime \in T_{x_0}X$ with 
$\abs{Z},\abs{Z^{\prime}}<\var^\prime$ and every $k\in\N$, we have 
\begin{equation} \label{toe6.22}
p^{-n}T_{p,\,x_0}(Z,Z^{\prime})\cong 
\sum^k_{r=0}(\mQ_{r,\,x_0}\cP_{x_0})
(\sqrt{p}Z,\sqrt{p}Z^{\prime})p^{-\frac{r}{2}} + \mO(p^{-\frac{k+1}{2}}).
\end{equation}
in the sense of \eqref{toe6.7}.
\end{itemize}
\end{itemize}

Then $\{T_p\}$ is a Toeplitz operator. 
\end{thm}
\begin{proof} As explained in \eqref{toe3.3}, we can
assume that $T_p$ is self-adjoint.
We will define inductively the sequence
$(g_l)_{l\geqslant0}$, $g_l\in \cC^\infty(X, \End(E))$ 
such that 
\begin{equation}\label{toe6.23}
T_{p} = \sum_{l=0}^m P_{p}\, g_l \, p^{-l}\, P_{p}
+\mO(p^{-m-1})\quad\text{for every $m\geqslant 0$,}
\end{equation} 
using the same procedure as in \eqref{toe3.4}. %\eqref{toe3m}-\eqref{toe3.2m}, 
Moreover, we can take these $g_l$'s to be self-adjoint.
For $x_0\in X$, we set
\begin{equation}\label{toe6.24}
g_0(x_0)=\mQ_{0,\,x_0}(0,0)\,\in\End(E_{x_0})\,.
\end{equation}
We will show that
\begin{equation}\label{toe6.25}
T_{p}=P_p\,g_0\,P_p + \mO(p^{-1}).
\end{equation} 
We need to establish the following analogue of Proposition \ref{toet3.2}.
%-----------------------------------------------------------------------------
\begin{prop}\label{toet6.5}
In the conditions of Theorem \ref{toet6.4}, we have 
$\mQ_{0,\,x_0}(Z,Z^{\prime})=\mQ_{0,\,x_0}(0,0)\in 
\End(E_{x_0})$
for all $x_0\in X$ and all $Z,Z^{\prime}\in T_{x_0}X$. 
\end{prop}
%-----------------------------------------------------------------------------
\begin{proof}
The key observation is the following. Let
$\{\Xi_{p}\}_{p\in\N}$, 
$\{Q_{0,\,x_0}\}_{x_0\in X}$ and $\{\Xi_{p}^\prime\}_{p\in\N}$, 
$\{Q_{0,\,x_0}^\prime\}_{x_0\in X}$ two pairs satisfying
 Condition \ref{coe2.71} for $k=0$. Then \eqref{pb4.9}, 
\eqref{toe6.5} and \eqref{toe6.6} imply that
\begin{equation} \label{toe6.26}
p^{-n} (\Xi_{p}\circ\Xi_{p}^\prime)_{x_0} (Z,Z^\prime)
\cong  ((Q_{0,\,x_0} \cP_{x_0})\circ (Q_{0,\,x_0}^\prime \cP_{x_0}))
(\sqrt{p}Z,\sqrt{p}Z^{\prime})
+\mO(p^{-\frac{1}{2}})\,,
\end{equation}
in the sense of Notation \ref{noe2.71} and \eqref{toe1.6}.
% and using with the notation \eqref{toe1.6} for .

We modify now the proof of Lemma \ref{toet3.3}.
 Formula \eqref{toe6.22} for $k=0$ gives
\begin{equation}\label{toe3.101}
p^{-n}T_{p,\,x_0}(Z,Z^{\prime})
\cong (\mQ_{0,\,x_0}\cP_{x_0})(\sqrt{p}Z,\sqrt{p}Z^{\prime})
+ \mO(p^{-1/2}).
\end{equation}
Moreover, the analogue of Lemma \ref{toet2.2} shows that
\begin{equation} \label{toe2.91}
p^{-n} P_{p,\,x_0}(Z,Z^\prime)\cong  
(J_{0,\,x_0} \cP_{x_0})(\sqrt{p}Z,\sqrt{p}Z^{\prime})p^{-\frac{r}{2}}
+\mO(p^{-1/2})\,.
\end{equation}
By \eqref{toe3.101} and \eqref{toe2.91}  we can apply the observation 
at the beginning for $\Xi_p=T_p$ and $\Xi_p^\prime=P_p$ to obtain
\begin{equation} \label{bsy1.191}
p^{-n}(P_p\,T_{p}\,P_p)_{x_0}(Z,Z^{\prime})
\cong ((\cP J_0)  \circ(\mQ_{0}\cP)\circ(\cP J_0))_{x_0}
(\sqrt{p}Z,\sqrt{p}Z^{\prime}) + \mO(p^{-1/2}).
\end{equation}
Using the same argument as in the proof of Lemma \ref{toet3.3} (note also that
$J_{0,\,x_0}= \Id_E$ by \eqref{toe6.15}) we see that
$\mQ_{0,\,x_0}$ is a polynomial in $z,\ov{z}^\prime$.

Now, we need to establish the analogue of \eqref{6.38}.
We define $F^{(i)}(\wi{x},\wi{y})$, $\wi{F}^{(i)}(\wi{x},\wi{y})$
as in \eqref{b6.32}. Then from \eqref{pb4.22a}, \eqref{toe6.5}, 
we know that for $g\in G_U$, $\wi{x},\wi{y}\in \wi{U}_1$, 
\begin{equation} \label{toe6.27}
 g\cdot F^{(i)}(g^{-1}\wi{x},\wi{y})= F^{(i)}(\wi{x}, g\wi{y}).
\end{equation}
We denote by $F^{(i)}F(D_p)$ and
$F(D_p)\wi{F}^{(i)}$ the operators defined by the kernels
$$\eta(d(x,y))F^{(i)}(\wi{x},\wi{y})F(\wi{D}_p)(\wi{x},\wi{y})
\quad \mbox{ and }\quad
\eta(d(x,y))F(\wi{D}_p)(\wi{x},\wi{y})\wi{F}^{(i)}(\wi{x},\wi{y})$$
as in \eqref{pb4.6} and \eqref{pb4.7}. Set
%-----------------------------------------------------------------------------
\begin{equation}\label{toe6.28}
\cT_{p}= T_{p}
-\sum_{i\leqslant\deg F_x} (F^{(i)}F(D_p))\, p^{i/2}.
\end{equation}

Now using \eqref{toe6.28} instead of \eqref{b6.33},
by \eqref{pb4.6} and the proof of Proposition \ref{toet3.2}, 
we get the analogue of \eqref{6.38} and hence Proposition \ref{toet6.5}.
\end{proof}
We go on with the proof of Theorem \ref{toet6.4}. Applying Proposition \ref{toet6.5}
and the proof of Proposition \ref{toet3.8}, 
we obtain \eqref{toe6.25}.

Finally, we deduce \eqref{toe6.23} according to the pattern set down 
in the proof of Theorem \ref{toet3.1}.
This completes the proof of Theorem \ref{toet6.4}. 
\end{proof}

We can therefore show that the set of Toeplitz operators 
on a compact orbifold is 
closed under the composition of operators, so forms an algebra. 
%----------------------------------------------------------------------------
\begin{thm}\label{toet6.7}
Let $(X,\om)$ be a compact K{\"a}hler orbifold and $(L,h^L)$ be 
a holomorphic Hermitian proper orbifold line bundle satisfying
the prequantization condition \eqref{0.-1}.
Let $(E,h^E)$ be an arbitrary holomorphic Hermitian proper orbifold 
vector bundle on $X$.

Consider $f,g\in\cC^\infty(X,\End(E))$. 
Then the product of the Toeplitz operators 
$T_{f,\,p}$ and  $T_{g,\,p}$ is a Toeplitz operator, 
more precisely, it admits an asymptotic expansion
in the sense of \eqref{toe4.2},
where $C_r(f,g)\in\cC^\infty(X,\End(E))$
and $C_r$ are bidifferential operators defined locally as in \eqref{toe4.2}
on each covering $\wi{U}$ of an orbifold chart
 $(G_U,\widetilde{U})\stackrel{\tau_U}{\longrightarrow}U$. 
In particular $C_0(f,g)=fg$.

If $f,g\in\cC^\infty(X)$, then \eqref{toe4.4} holds.

 Relation \eqref{toe4.17} also holds for any $f\in\cC^\infty(X,\End(E))$.
\end{thm}
\begin{proof}
Notice that by using \eqref{toe6.26} we have 
\begin{multline}\label{toe6.29}
(T_{f,\,p}\,T_{g,\,p})(x,x^\prime)= \int_{X}
(F(D_p)fF(D_p))(x,x'')(F(D_p) g F(D_p))(x'',x^\prime)
\,dv_{X}(x'')\\+ \cO(p^{-\infty}).
\end{multline}  
{}From \eqref{pb4.9},  \eqref{toe6.29} and the proof of Theorem \ref{toet4.1}, 
we get Theorem \ref{toet6.7}.
\end{proof}

\begin{rem}\label{toet4.31} As in Remark \ref{toet4.3}, 
Theorem \ref{toet6.7} shows that
on every compact K{\"a}hler orbifold $X$ 
admitting a prequantum line bundle $(L,h^L)$,
we can define in a canonical way an associative star-product
$f*g=\sum_{l=0}^\infty \hbar^{l}C_l(f,g)\in\cC^\infty(X)[[\hbar]]$ 
for every $f,g\in \cC^\infty(X)$,
called the \textbf{\emph{Berezin-Toeplitz star-product}}\/.
Moreover, $C_l(f,g)$ are bidifferential operators 
defined locally as in the smooth case.
\end{rem}

\subsection{Symplectic orbifolds} \label{pbs4.4}

In this Section we state the result for symplectic orbifolds.

We work on a compact symplectic orbifold $(X,\om)$ of real dimension $2n$.
Assume that there exists a proper orbifold Hermitian line bundle
$L$ over $X$ endowed with a Hermitian connection $\nabla^L$ 
with the prequantization property $\frac{\sqrt{-1}}{2\pi}R^L=\omega$. 
This implies in particular that there exist $k\in \N$
such that $L^k$ is a line bundle in the usual sense.
Let $(E,h^E)$ be a proper orbifold Hermitian vector bundle on $X$ equipped 
with a Hermitian connection $\nabla^E$.

Let $J$ be an almost complex structure on $TX$ such that \eqref{bk1.1} holds.
We endow $X$ with a Riemannian metric $g^{TX}$ compatible with $J$.

Then the construction in Section \ref{s3.0} goes through,
especially, we can define the spin$^c$ Dirac operator $D_p:
\Omega^{0,\scriptscriptstyle{\bullet}}(X,L^p\otimes E)\longrightarrow
\Omega^{0,\scriptscriptstyle{\bullet}}(X,L^p\otimes E)$.
The orthogonal projection $P_p:L^2(X,E_p)\longrightarrow\ke(D_p)$
with $E_p:=\Lambda^{0,\scriptscriptstyle{\bullet}}\otimes L^p\otimes E$ is called the Bergman projection.
The smooth kernel $P_p(\cdot,\cdot)$ of $P_p$ with respect to 
the Riemannian volume form $dv_X$, is called the Bergman kernel of $D_p$.

We define the Toeplitz operator $T_p:L^2(X,E_p)\longrightarrow L^2(X,E_p)$
as in Definition \ref{toe-def} by using the orthogonal projection $P_p$ 
defined above. 
Especially $T_{f,\,p}=P_p\,f\,P_p$ for $f\in\cC^{\infty}(X,\End(E))$.

By the argument in Section \ref{pbs4.2} we see that 
Theorem \ref{specDirac} and Proposition \ref{0t3.0} still hold:

\begin{thm}\label{toet6.11}
Assume that $(X,J,\om)$ is a compact symplectic orbifold endowed with 
a prequantum proper line bundle $(L,h^L,\nabla^L)$.
We endow $X$ with a Riemannian metric $g^{TX}$ compatible with $J$.  
Let $(E,h^E)$ be a proper orbifold Hermitian vector bundle on $X$
with Hermitian connection $\nabla^E$. Then 
\begin{itemize}
\item[(i)] the associated Dirac operator $D_p$ has 
a spectral gap \eqref{diag5}, and 
\item[(ii)] $P_p(x,x^\prime)=\cO(p^{-\infty})$ for $d(x,x^\prime)>\var>0$ 
in the sense of \eqref{1c3}.
\end{itemize}
\end{thm}

Now by combining the argument in
 Sections \ref{s5.3} and \ref{pbs4.3}, we get  
the following extension of Theorem \ref{toet4.1}.
\begin{thm}\label{toet6.12}
Let us make the same assumptions as in Theorem \ref{toet6.11}.
Then for every $f,g\in\cC^\infty(X,\End(E))$ the product of 
the Toeplitz operators $T_{f,\,p}$ and  $T_{g,\,p}$ is a Toeplitz operator, 
more precisely, it admits an asymptotic expansion
in the sense of \eqref{toe4.2}, where
$C_r(f,g)\in\cC^\infty(X,\End(E))$ and $C_r$ are bidifferential operators 
defined locally as in \eqref{toe4.2} on each covering $\wi{U}$ of 
an orbifold chart $(G_U,\widetilde{U})\stackrel{\tau_U}{\longrightarrow}U$. 
In particular $C_0(f,g)=fg$.

If $f,g\in\cC^\infty(X)$, then \eqref{toe4.4} holds.

Relation \eqref{toe4.17} also holds for any $f\in\cC^\infty(X,\End(E))$.
\end{thm}
As before, for the given data $X, J, g^{TX}, L,h^L,\nabla^L$ 
from Theorem \ref{toet6.11} and $E=\C$, Theorem \ref{toet6.12} implies 
 a canonical construction of the (associative) Berezin-Toeplitz star-product
$f*g=\sum_{l=0}^\infty \hbar^{l}C_l(f,g)\in\cC^\infty(X)[[\hbar]]$ 
for every $f,g\in \cC^\infty(X)$.

%\bibliography{mmbook,mmabook}
%\bibliographystyle{amsplain} 

\def\cprime{$'$} \def\cprime{$'$}
\providecommand{\bysame}{\leavevmode\hbox to3em{\hrulefill}\thinspace}
\providecommand{\MR}{\relax\ifhmode\unskip\space\fi MR }
% \MRhref is called by the amsart/book/proc definition of \MR.
\providecommand{\MRhref}[2]{%
  \href{http://www.ams.org/mathscinet-getitem?mr=#1}{#2}
}
\providecommand{\href}[2]{#2}

\end{document}